%% file: HEG.tex
\documentclass[12pt,oneside]{amsart}
\usepackage[letterpaper]{geometry}
\input{macro}

\title[Quasiconvex subgroups in Acylindrically hyperbolic groups]{A notion of quasiconvex subgroups in acylindrically hyperbolic groups }

\author{Ping Wan}
\address{Department of Mathematics, Statistics, and Computer Science,
University of Illinois at Chicago,
322 Science and Engineering Offices (M/C 249),
851 S. Morgan St.,
Chicago, IL 60607-7045}
\email{pwan5@uic.edu}
\begin{document}

\begin{abstract}
In this paper, we present a notion of quasiconvexity in the setting of finitely-generated groups with hyperbolically embedded subgroups. Our main result shows that this notion yields uniform quasiconvex constants in the setting of coned-off cusped spaces. We also prove that this notion of quasiconvexity is preserved under sufficiently long Dehn filling. As an application, we generalize a theorem of Groves-Manning on groups acting on \CATz  cube complexes. 
\end{abstract}
\maketitle


\tableofcontents
\newpage

\section{Introduction}

Acylindrically hyperbolic groups were introduced by Osin in \cite{osinAcylindricallyHyperbolicGroups2016}. This class of groups is broad enough to cover many interesting groups, such as relatively hyperbolic groups, all but finitely many mapping class groups of finite type, and many more. In the meantime, acylindrical hyperbolicity is strong enough to prove many interesting results. For a detailed survey we refer to \cite{osinAcylindricallyHyperbolicGroups2016} and \cite{osinGroupsActingAcylindrically2018}. 

One equivalent characterization of acylindrical hyperbolicity is the existence of a proper infinite hyperbolically embedded subgroup, as proved in \cite[Theorem 1.2]{osinAcylindricallyHyperbolicGroups2016}. Groups with hyperbolically embedded subgroups were introduced by Dahmani, Guirardel and Osin in \cite{dahmaniHyperbolicallyEmbeddedSubgroups2017} as a generalization of relatively hyperbolic groups. They prove a Dehn filling theorem for groups with hyperbolically embedded subgroups, which could be viewed as a generalization of the relatively hyperbolic Dehn filling theorem in \cite{osinPeripheralFillingsRelatively2007} or \cite{grovesDehnFillingRelatively2008}.

It is natural to ask: what is a proper notion of quasiconvex subgroup of acylindrically hyperbolic groups? As Balasubramanya shows in \cite{balasubramanyaAcylindricalGroupActions2017}, the notion of quasiconvexity is not sufficient. Abbott and Manning introduce the notion of A/QI triple in \cite{abbottAcylindricallyHyperbolicGroups2024}. Examples of A/QI triples include strongly quasiconvex subgroups of relatively hyperbolic groups, convex cocompact subgroups of mapping class groups, and many more. 
While A/QI triple emphasizes the acylindrical action, it does not address the structure of hyperbolically embedded subgroups, and therefore has no direct application together with Dehn filling theorems. 

In this paper, we give an analog of relatively quasiconvexity in the setting of finitely generated groups with finitely many hyperbolically embedded subgroups. We adopt the equivalent definition of groups with hyperbolically embedded subgroups by Martínez-Pedroza and Rashid in \cite{martinez-pedrozaNoteHyperbolicallyEmbedded2022}, and define \emph{$\gp$--quasiconvex group} in this setting. 
In particular, we prove that a $\gp$--quasiconvex group has uniform quasiconvex constants in the setting of \cite{dahmaniHyperbolicallyEmbeddedSubgroups2017}.   

\begin{restatable}{theorem}{UniformQC}
    \label{thm:uniform_qc}
    Let $G$ be a finitely generated group, $\Pcal$ a finite family of infinite subgroups that are hyperbolically embedded into $G$. Let $H$ be a finitely generated subgroup of $G$. 
    If $H$ is $\gp$--quasiconvex,  then there exists a subset $Y\subseteq H$, a finite collection of infinite subgroups $\Dcal$, and a subset $X\subseteq G$ such that
    \begin{enumerate}
        \item $\Dcal\he(H,Y)$,
        \item $\Pcal\he (G,X)$, and
        \item $\forall r >0$, there exists a quasi-isometric embedding between coned cusped spaces (See \cref{def:coned-off-cayley-graph}) $\phi_r:  \K_r(H, Y, \cup_j Y_j, \Dcal) \to \K_r(G,X, \cup_i X_i, \Pcal)$.
    \end{enumerate}

    Moreover, there exists $r_\l, \l$ such that  $\forall r\geq r_\l$, $\phi_r(\K_r(H, Y, \cup_j Y_j, \Dcal)) \subseteq \K_r(G,X, \cup_i X_i, \Pcal)$ is $\l$--quasiconvex. That is, given any $\bar x, \bar  y \in \phi_r(\K_r(H, Y, \cup_j Y_j, \Dcal))$, any geodesic $\eta = [\bar x, \bar y]$ is contained in the $\l$ neighborhood of $\phi_r(\K_r(H, Y, \cup_j Y_j, \Dcal))$.

    In particular, $\l$ does not depend on $r$.
\end{restatable}

\cref{thm:uniform_qc} indicates that $\gp$-quasiconvexity is compatible with  \cite[Theorem 7.19]{dahmaniHyperbolicallyEmbeddedSubgroups2017} which states that hyperbolic embeddedness of subgroups is preserved by sufficiently long Dehn fillings. In particular, this notion of quasiconvexity is preserved under sufficiently long $H$-fillings. (The definition of sufficiently long $H$-fillings is in \cref{sec:Dehn-filling-prepare}.)

\begin{theorem}
    \label{prop:preserve-QC-rephrase}
    Under the setting of \cref{thm:uniform_qc}, for all sufficiently long \(H\)--fillings $\bar G$ of \(G\) with kernel \(K\), let $\bar H = H/H\cap K$, $\bar G = G/K$, and $\bar {\Pcal} = \{ P_\l / P_\l\cap K \mid P_\l \in \Pcal \}$. Then $\bar  H$ is $(\bar G, \bar {\Pcal})$--quasiconvex.
\end{theorem}

We can thus generalize certain theorems involving relative quasiconvex subgroups and relative Dehn fillings to the setting of acylindrically hyperbolic groups. As an example, we generalize \cite[Corollary 6.6]{grovesHyperbolicGroupsActing2023} to the context of finitely generated groups with finitely many hyperbolically embedded subgroups using \cref{thm:uniform_qc} and \cite[Theorem 7.19]{dahmaniHyperbolicallyEmbeddedSubgroups2017}.
See \cref{sec:Dehn-filling-prepare} for definition of $\Qcal$-filling.

\begin{restatable}{theorem}{MainCube}
\label{thm:main_cube}
        Let $G$ be a finitely generated group and $\Pcal$ be a finite collection of infinite subgroups of $G$. Suppose that $\Pcal\he G$ and that $G$ acts cocompactly on a $\operatorname{CAT}(0)$ cube complex $C$. Suppose each element in the hyperbolically embedded subgroups of $G$ fixes some point of $C$  and that cell-stabilizers are finitely generated and full $\gp$--quasiconvex. Let $\sigma_1, \dots, \sigma_k$ be representatives of the $G$-orbits of cubes of $C$. For each $i$ let $Q_i$ be the finite index subgroup of $G_i =\textnormal{Stab}(\sigma_i)$ consisting of elements which fix $\sigma_i$ pointwise. Let $\Qcal = \{Q_1,\dots,Q_k\}$.

    For sufficiently long \emph{$\Qcal$--fillings} 
    \begin{equation*}
        G \rightarrow \Bar{G} = G(N_1,\dots,N_m)
    \end{equation*}
    of $G$, with kernel $K$, the quotient $C/K$ is a $\operatorname{CAT}(0)$ cube complex.
\end{restatable}

Recall that a subgroup $H$ is \emph{full} if whenever $P$ is a parabolic subgroup so that $H\cap P$ is infinite, $H\cap P$ is a finite index subgroup of $P$.

This article is structured as follows. We first recall some basic notations in \cref{sec:prelim}. In particular, we prove \cref{prop:qi_to_conedoff} as an alternative to \cite[Proposition 5.4]{martinez-pedrozaNoteHyperbolicallyEmbedded2022}.  We then define the generalization of relatively quasiconvexity in \cref{sec:qc}. In \cref{sec:prepare-uniform-QC} we prove some useful technical Lemmas about angle metric. \cref{sec:uniformqc} is dedicated to the proof of \cref{thm:uniform_qc}. 

In \cref{sec:main}, we generalize a series of key ingredients for \cite[Corollary 6.6]{grovesHyperbolicGroupsActing2023} to the setting of hyperbolically embedding groups. 
In \cref{sec:image-of-H-under-dehn-filling}, we proved \cref{prop:induce-filling-map-injective} and \cref{prop:image-of-H-convex} (a more detailed version of \cref{prop:preserve-QC-rephrase}) which are not used in the proof of \cref{thm:main_cube} but may be of independent interest. 
Finally, we prove \cref{thm:main_cube} in \cref{subsec:main-proof}.

\textbf{Acknowledgments:}     The author would like to thank their advisor, Daniel Groves, for his support and guidance. They would also like to thank their academic siblings (in alphabetical order) Darius Alizadeh, Carl Tang, and Jagerynn Verano for helpful discussions. 

\section{Preliminaries}
\label{sec:prelim}

\subsection{Hyperbolically Embedded Subgroups}
\label{sec:heg}

Let $G$ be a group and $\Pcal= \{ P_\lambda \}_{\lambda\in \Lambda}$ be a collection of subgroups of $G$. Recall that $X$ is a \emph{relative generating set} for $(G,\Pcal)$ if $G$ is generated by the union of $X$ and $\Pcal$. Let $\G(G, X\sqcup \Pcal)$ be the Cayley graph of $G$ with respect to $X\sqcup \Pcal$. 

For $h,k\in P_\l$, let $\hat d_{\l}(h,k)$ be the length of the shortest edge-path from $h$ to $k$ with the property that if an edge is labeled by an element of $P_\l$ then the endpoints of the edge are not elements of the subgroup $P_\l$; if no such path between $h$ and $k$ exists, then $\hat d_{\lambda}(h,k) = \infty$.

\begin{definition}[{\cite[Definition  4.25]{dahmaniHyperbolicallyEmbeddedSubgroups2017}}]
    \label{def:heg-dgo}
    Let $G$ be a group, $\Pcal = \{ P_\lambda \}_{\lambda\in \Lambda}$ a collection of subgroups of $G$, $X$ a relative generating set of $G$ with respect to $\Pcal$. We say that $\Pcal$ is \emph{hyperbolically embedded in $G$ with respect to $X$}, denoted as $\Pcal\he (G,X)$, if the Cayley graph $\G(G,X\sqcup \Pcal)$ is hyperbolic and for every $\lambda\in\Lambda$, the metric space $(P_\lambda,\hat d_\lambda)$ is locally finite.
    
    We say $\Pcal$ is \emph{hyperbolically embedded in $G$} if $\Pcal\he (G,X)$ for some $X$, and we write  $\Pcal\he G$. 
\end{definition}

Note that $\Pcal$ is not necessarily a finite set. In the case where $G$ is finitely generated and $\Pcal$ contains only finitely many elements, Martínez-Pedroza and Rashid give a characterization of finitely-generated groups with hyperbolically embedded subgroups in \cite{martinez-pedrozaNoteHyperbolicallyEmbedded2022} using fineness via the angle metric. 

\begin{definition}
    \label{def:angle}
    Let $\G$ be a simplicial graph and $v$ be a vertex. We denote the set of vertices adjacent to $v$ by $T_v\G$. For $x,y\in T_v\G$, the \emph{angle between $x,y$ at $v$}, denoted  $\angle_v(x,y)$,  is  the (combinatorial) length of the shortest path between $x$ and $y$  in the graph $\G -\{v\}$. If no such path exists, then $\angle_v(x,y)=\infty$. 
\end{definition}

\begin{definition}
    \label{def:fine_local}
    The graph $\G$ is \emph{fine at $v$} if $(T_v\G, \angle_v)$ is a locally finite metric space. The graph is \emph{fine} if it is fine at every vertex.
\end{definition}

\begin{definition}
    \label{def:angle_of_paths}
    Let $c_1$ (resp. $c_2$) be two paths starting at $v$. Let $x$ (resp. $y$) be the first vertex on $c_1$ (resp. $c_2$) adjacent to $v$. We define $\angle_v(c_1,c_2) = \angle_v(x,y)$.
\end{definition}

\begin{definition}[Coned-off Cayley Graph]
    \label{def:coned-off-cayley-graph}
    Let $G$ be a group, $X\subset G$, and $\Pcal$ a finite collection of subgroups of $G$. The \emph{coned-off Cayley graph} $\hat \G(G,\Pcal,X)$ is the graph with vertex set $G\cup_{{P\in \Pcal}} \{G/P\}$, where $G/P$ is the set of left cosets of $P$ in $G$; and edge set $\{ \{g,gx\}:g\in G,x\in X \} \cup \{ \{g,gP\}: g\in G, P\in \Pcal \}$. A vertex in $G/P\subseteq \hat\G(G,\Pcal, X)$ is called an \emph{apex}. Edges in $\{ \{g,gP\}: g\in G, P\in \Pcal \}$ are called \emph{cone edges of $gP$ based at $g$}.
\end{definition}

It is well-known that $\GG$ is quasi-isometric to $\Gd$. Martínez-Pedroza and Rashid proved that \cref{def:heg_conedoff} is equivalent to \cref{def:heg-dgo} when $G$ is finitely-generated and $\Pcal$ is a finite collection of infinite subgroups.

\begin{definition}[{\cite[Proposition 5.8]{martinez-pedrozaNoteHyperbolicallyEmbedded2022}}]
    \label{def:heg_conedoff}
    Let $G$ be a group, $X\subseteq G$, and $\Pcal$ a finite collection of infinite subgroups. Then $\Pcal$ is \emph{hyperbolically embedded in $G$ with respect to $X$} if and only if $\GG$ is connected, hyperbolic, and fine at cone vertices.
\end{definition}

More generally, they captured the fineness at cone vertices in the definition of $\gp$--graphs.

\begin{definition}[{\cite[Definition 1.1]{martinez-pedrozaNoteHyperbolicallyEmbedded2022}}]
    \label{def:GPgraph}
     Let $G$ be a group and $\Pcal$ be a finite collection of subgroups. A graph $\G$ is a \emph{$\gp$--graph} if $G$ acts on $\G$ and
     \begin{enumerate}
         \item $\G$ is connected and hyperbolic,
         \item there are finitely many G-orbits of vertices,
         \item $G$-stabilizers of vertices are finite or conjugates of $P\in\Pcal$, and for every $P\in\Pcal$, there is a vertex with $G$-stabilizer equal to $P$,
         \item $G$-stabilizers of edges are finite, and
         \item $\G$ is fine at each vertex of $V_\infty(\G) = \{v \in V (\G) | v \text{ has infinite stabilizer}\}$.
     \end{enumerate}
\end{definition}

They then proved that \cref{def:heg_finegraph}, using the notion of $\gp$--graph, is equivalent to \cref{def:heg_conedoff}. 

\begin{definition}[{\cite[Theorem 5.9]{martinez-pedrozaNoteHyperbolicallyEmbedded2022}}]
    \label{def:heg_finegraph}
    Let $G$ be a finitely generated group. A finite collection of infinite subgroups $\Pcal$ is \emph{hyperbolically embedded in $G$} if and only if there exists a $\gp$--graph. 
\end{definition}

\subsection{Thick Graphs}

The key ingredient of \cite[Theorem 5.9]{martinez-pedrozaNoteHyperbolicallyEmbedded2022} is \cite[Proposition 5.4]{martinez-pedrozaNoteHyperbolicallyEmbedded2022}, which constructs a quasi-isometry between any thick $\gp$--graph and  $\GG$ for some well-chosen $X$. However, their construction does not give rise to the desired statement. For example, let $P$ be an infinite subgroup of $G$, and let $S$ be a finite relatively generating set of $G$ with respect to $P$. Suppose $\G =\hat\G (G,P,S)$ is a thick $(G,P)$--graph. That is, $\G$ contains vertices $u_0$ and $v_0$ such that $\{u_0 ,v_0\} \in E(\G)$, $G_{u_0} = \{1\}$ and $G_{v_0} = P$. We may also assume that for all $s\in S$, ${u_0, s.u_0}\in E(\G)$.
Following the construction of \cite[Proposition 5.4]{martinez-pedrozaNoteHyperbolicallyEmbedded2022}, we obtain:
\begin{align*}
    X = & \{ g\in G : d_\G (u_0, g.u_0)=1  \textit{ or } d_\G(u_0, g.v_0)=1 \} \\
    = & S\cup P.
\end{align*}
Since $P\subseteq X$, at the cone vertex $P$, there are infinitely many loops of length $3$. Hence $\hat \G(G,P,X)$ is not fine at cone vertices.

In this section we provide an alternative proof to \cite[Proposition 5.4]{martinez-pedrozaNoteHyperbolicallyEmbedded2022} in \cref{prop:qi_to_conedoff}. We denote the stabilizer of a cell $v$ under the action of $G$ as $G_v$.

\begin{definition}[Generalization of {\cite[Definition 5.1]{martinez-pedrozaNoteHyperbolicallyEmbedded2022}}]
    \label{def:thick_graph}
    Suppose $\Pcal = \{ P_1,\dots,P_n \}$. A $\gp$--graph $\G$ is \emph{thick} if it satisfies the following conditions: 
    \begin{enumerate}
        
        \item $\G$ contains vertices $u_0$ and $\{v_{1},\dots v_{n}\}$ such that $G_{ u_0} = 1$, and so that for $k = 1,\dots, n $, \[\{u_0,v_{k} \} \in E(\G), \quad  \text{and} \quad G_{ v_{k}} = P_k.\]
        
        \item There is a finite relative generating set $S$ of $G$ with respect to $\Pcal$ so that for all $s\in S$, \[\{u_0, s.u_0 \} \in E(\G).\]
        
        \item There exists a collection $\{u_0,u_1,...,u_l \}$ of representatives of $G$-orbits of vertices of $\G$ with finite $G$-stabilizers, where $u_0$ is the vertex with trivial stabilizer,  so that for $j=1,\dots, l$, \[    \{u_0,u_j \} \in E(\G).\]
         
    \end{enumerate}
\end{definition}

The proof of \cite[Proposition 5.2]{martinez-pedrozaNoteHyperbolicallyEmbedded2022} works mutatis mutandis to prove the following statement:

\begin{proposition}[Generalization of {\cite[Proposition 5.2]{martinez-pedrozaNoteHyperbolicallyEmbedded2022}}]
\label{prop:exist_thick_graph}
    Suppose $G$ is finitely generated relative to a finite collection of infinite subgroups $\Pcal$. If there exists a $\gp$--graph $\G$, then by attaching a vertex $u_0$ with trivial $G$--stabilizer and adding finitely many $G$--orbits of edges to $\G$ if necessary, we obtain a thick $\gp$--graph $\D$. Moreover, the inclusion map $\G\to\Delta$ is a $G$-equivariant quasi-isometry.
\end{proposition}

Given a $\gp$--graph $\G$, we denote the combinatorial metric on $\G$ as $d_\G$. That is, each edge has length 1.

\begin{proposition}[Correction of {\cite[Proposition 5.4]{martinez-pedrozaNoteHyperbolicallyEmbedded2022}}]
\label{prop:qi_to_conedoff}
    Let $G$ be a finitely generated group. Suppose that $\Pcal = \{ P_k \}_{1\leq k \leq n }$ is a collection of infinite subgroups of $G$. Let $\G$ be a thick $\gp$--graph, and let $\{v_1,\dots,v_n\}$, $\{u_0,\dots,u_l \}$ and $S$ be as in \cref{def:thick_graph}.
  
    Observe that for every $1\leq k \leq n$, the set $Z_k=\{ t \in G \mid \operatorname{dist}_\G (u_i, t.v_k)=1 \text{ for some } 0\leq i\leq l \}$ is a union of cosets of $P_k$. Choose one representative from each coset of $P_k$ in $Z_k$ and let $T_k$ be the set of all such representatives. Let $ T = \bigcup_{1\leq k\leq n} T_k,$ and 
    \[    X = \{ g\in G \mid \operatorname{dist}_\G (u_i, g. u_j)= 1 \text{ for some } 0\leq i,j\leq l  \} \cup T .\] 
    
    Then $q: \GG \to \G$ given by $g\mapsto g.u_0$ and $gP_k\mapsto g.v_k$ is a quasi-isometry. Moreover, $\GG$ is fine at cone vertices. In particular $\GG$ is a $\gp$--graph.
\end{proposition}
\begin{proof} 
    For simplicity we denote $\GG$ by $\hat \G$. 
    
    Let $f,g\in G$ be such that $\{f,g\}$ is an edge in $\hat \G$ labeled by $x$. If $x\in X\backslash T$, then for some $i,j$, $\{f.u_i, g.u_j\}$ is an edge in $\G$. Therefore $[f.u_0, f.u_i, g.u_j, g.u_0]$ is a path of length $3$ in $\G$. If $x\in T$, then for some $i,k$, $\{f.u_i, g.v_k\}$ is an edge.  Therefore $[f.u_0, f.u_i, g.v_k, g.u_0]$ is a path of length $3$ in $\G$. Observe now that if $\{g,gP_k\}$ is an edge in $\hat \G$, then there is an edge $\{g.u_0,g.v_k\}$ in $\G$. It follows that for any $a,b \in V(\hat \G)$ we have:
    \[
        \operatorname{dist}_\G(q(a), q(b) ) \leq 3 \operatorname{dist}_{\hat\G} (a,b).
    \]
    
    On the other hand, any vertex of $\G$ is of the form $g.u_i$ or $g.v_k$ for some $g\in G$, $0\leq i\leq l$ and $0\leq k\leq n$. If $\{f.u_i, g.u_j\}$ is an edge in $\G$, then $f^{-1}g \in X$. That is, $\{f,g\}$ is an edge in $\hat \G$. If $\{f.u_i, g.v_k\}$ is an edge, then there exists $t\in f^{-1}g P_k $ such that $t\in T_k\subseteq X$. Since $gP_k = ff^{-1}g P_k = ftP_k$, $[f,ft, gP_k]$ is a path of length $2$ in $\hat \G$. It follows that for any $a,b\in V(\hat \G)$ we have:
    \[
        \operatorname{dist}_{\hat\G} (a,b) \leq 2 \operatorname{dist}_{\G} (q(a),q(b)).
    \]
    
    Observe also $\{g.u_i,g.u_0\} \in E(\G)$ for all $1\leq i\leq l $ and all $g\in G$, so $\G\subseteq N_1(\hat \G)$. Therefore $q:\hat \G\to \G$ is a quasi-isometry. It is clear that $q$ is $G$-equivariant. 
    
    It remains to prove that $\hat\G$ is fine at cone vertices. It suffices to show that for every $N>0$, for every $1\leq k \leq n$, the set \( \Omega_k = \{ h\in P_k\mid \angle_{P_k}(e, h) = N \} \) is finite. We prove the statement by induction. 
    
    When $N=1$, the vertex $e$ is connected to $h$ by an edge. That is, $h\in X\cap P_k$.  If $h\in X\backslash T$, then $\{e,h\}$ corresponds to the path $[u_0, u_i, h.u_j, h.u_0]$ in $\G$ for some $0\leq i,j\leq l$.  This path does not go through $v_k$.  
    Since $\G$ is fine at $v_k$, there are finitely many $h.u_0 \in T_{v_k}\G$ such that $\angle_{v_k}(u_0, h.u_0)\leq 3$. 
    It follows that there are finitely many $h\in X\cap P_k$ which are adjacent to $e$.  
    Suppose on the other hand that $h\in T \cap P_k$. Since $T$ is the collection of representatives of cosets of $P_k$, there is only one such $h$. It follows that there are finitely many such $h\in P_k\cap X$ in total.  
    
    Suppose now the statement is true for $1, 2, \dots, N-1$. Consider a path $c\subseteq \hat \G$ of length $N$ between $e$ and some $h\in P_k$ so that $c$ does not go through $P_k$. As shown above, for each edge in $c$, there exists a path $\a(c)$ of length at most $3$ that connects the image of endpoints of $e$. Let $\a(c)$ be the concatenation of those paths, then $\a(c)$ is a path in $\G$. Moreover, different paths in $\hat \G$ give rise to different paths in $\G$. 
    
    Let $z_0, z_1,\dots, z_N $ be the vertices on the path. For all $0\leq t \leq N$, $z_k$ is labeled by either $g_t\in G$ or  $g_t P_t$ for some $g_t\in G$. Observe that if $z_t$ is labeled by $g_tP_t$ and $P_t = P_k$, then $g_t\notin P_k$. In particular $z_0$ is labeled by $e$ and $z_N$ is labeled by $h$.  
    
    Let $e_t = [z_{t-1},z_t]$ for all $1\leq t \leq N$. If one of $z_{t-1}, z_t$ is a coset of $P_k$,  then $e_t$ corresponds to an edge in $\G$ that does not go through $v_k$. If none of $z_{t-1}, z_t$ is a coset of $P_k$, then $e_t$ is labeled by some $x_t\in X$. 
    
    If $x_t\in X\backslash T$, then it corresponds to $[g_{t-1}.u_0, g_{t-1}u_i, g_t.u_j, g_t.u_0]$ for some $0\leq i,j\leq l$. This path does not contain $v_k$ and has length at most $3$. If $x_t\in T$, then  $e_t$ corresponds to the path $[g_{t-1}.u_0, g_{t-1}.u_i,g_t.v_k,g_t.u_0]$. If $g_t\notin P_k$, then this path does not go through $v_t$ neither.  Therefore, if the path $c\subseteq \hat\G$ does not contain vertices in $P_k$ other than $e,h$, then $\a(c)\subseteq \G$ does not go through $v_k$. By the fineness of $\G$ at $v_k$, there are only finitely many such $c$.
    
    This leaves us with the case where for some $t$,  $x_t\in T$ and $g_t\in P_k$. In this case $g_t.v_k = v_k$ so $\a(c)$ passes through $v_k$. 
    We remark first $\a(c)$ passes through $v_k$ exactly once for each such an edge, that is, $\a(c)$ passes through $v_k$ finitely many times. Let $x_{t^1},\dots,x_ {t^M}$ be all such edges in $c$, then $\a(c)$ is the concatenation of:
    \[
        \a_0,\{g_{t^1-1}.u_{i^1},v_k\}, \{v_k,g_{t^1}.u_0\}, \a_1, \dots, \{g_{t^M-1}.u_{i^M},v_k\}, \{v_k,g_{t^M}.u_0\}, \a_M,
    \]
    where $\a_j$ is the subpath of $\a(c)$ between $g_{t^j}.u_0$ and $g_{t^{j+1}-1}.u_{i^{j+1}}$. (For simplicity of notation we denote $e$ as $g_{t^0}$.) Each $\a_j$ does not contain $v_k$ and has length at most $3N$. Therefore there are only finitely many choices of $\a_j$. That means for each fixed  $g_{t^j}.u_0$, there are finitely many choices of $g_{t^{j+1}-1}.u_{i^{j+1}}$. 
    
    Fix $g.u_i$ and suppose $\{ g.u_i, v_k \}$ is an edge. Then there exists $t\in T$ such that $tg.u_i = v_k$. If there exists $t'\neq t \in T$ such that $t'g.u_i = v_k$, then $t^{-1}t'\in \operatorname{Stab}(g.u_i)$. Since $\operatorname{Stab}(u_i)$ is finite, there are only finitely many such $t'$.
    It follows that for each fixed $\a_j$, there are only finitely many choices of $g_{t^{j+1}}.u_0$. This yields to finitely many $\a(c)$ and hence finitely many $c$. It follows that $\hat\G$ is fine at cone vertices.
\end{proof}

We construct below a quasi-inverse for the quasi-isometry from \cref{prop:qi_to_conedoff}.
\begin{construction}
    \label{cons:inverse_qi_from_conedoff}
    We make the same assumption as \cref{prop:qi_to_conedoff}. Recall that any vertex in $\G$  we construct $i:\G \to \GG$ as follows:
    \begin{itemize}
        \item Any vertex of the form $g.u_j$ is sent to $g$, where $i\in \{ 0, 1,\dots, l\}$.
        \item Any vertex of the form $g.v_k$ is sent to $g.P_k$, where $k\in \{ 1, \dots, n \}$
    \end{itemize}
    Clearly $i\circ q (g) = g$, $i\circ q (g.P_k) = g.P_k$. On the other hand, $q\circ i (g.u_0) = g.u_0$ and $q\circ i(g.v_k) = g.v_k$. As for $u_j$ where $j\in \{1, \dots, l\}$, $q\circ i (g.u_j) = g.u_0$, which is connected to $g.u_j$ by $g.\{u_0, u_j\} $. 
\end{construction}

We make the following observation for future use.
\begin{lemma}
    \label{lem:coned-off-cayley-graph-add-edge-in-P}
    Let $G$ be a group, $\Pcal$ a collection of infinite subgroup of $G$, and $X\subseteq G$ a set such that $\Pcal\he(G,X)$. Let $\hat \G_0 = \hat \G(G,\Pcal ,X)$ be the coned-off Cayley graph. 
    Let $Y$ be a finite subset $G$. Then $\hat \G_0$ is quasi-isometric to $\hat \G = \hat \G_0 \cup (\G(G,X\cup Y)\backslash \G(G,X))$. 
    Moreover, if $\hat \G_0$ is a $\gp$-graph (respectively thick $\gp$--graph), then so is $\hat \G$.  
\end{lemma}
\begin{proof}
    Observe that we obtain $\hat \G$ by attaching finitely many $G$-orbits of edges to $\hat \G_0$. By \cite[Corollary 4.3]{martinez-pedrozaNoteHyperbolicallyEmbedded2022} $\hat \G_0\hookrightarrow \hat \G$ is a quasi-isometry and $\hat \G$ is a $\gp$--graph. It is clear that $\hat \G$ remains a thick $\gp$--graph if $\hat \G_0$ is thick.
\end{proof}

\subsection{Coned-off Cusped Cayley Graph}
\label{sec:kr}

Given a graph $\G$, we construct a combinatorial horoball $\Hcal(\G)$ as in \cite[Definition 3.1]{grovesDehnFillingRelatively2008},  and define $\mathcal{D}_r$ as the subset of $\Hcal(\G)$ with depth at least $r$. We call $\G$ the \emph{base graph} of this horoball. 

Geodesics in combinatorial horoballs are easy to understand. 

\begin{lemma}[{\cite[Lemma 3.10]{grovesDehnFillingRelatively2008}}]
\label{lem:shape-of-geodesic-horoball}
    Let $\Hcal (\G) $ be a combinatorial horoball. Suppose that $x, y \in \Hcal (\G)$ are distinct vertices. Then there is a geodesic $\g(x, y) = \g(y, x)$ between $x$ and $y$ which consists of at most two vertical segments and a single horizontal segment of length at most $3$.
    
    Moreover, any other geodesic between $x$ and $y$ is Hausdorff distance at most $4$ from this geodesic.
\end{lemma}

By the construction of the combinatorial horoballs, it is easy to see that the length of those vertical segments are closely related to the distance between two endpoints in the base graph.

\begin{lemma}
\label{lem:length-of-geodesic-in-horoball}
    Let $\Hcal(\G)$ be a combinatorial horoball.    Let $p,q \in \G$ and let $\g$ be a geodesic of the shape in \cref{lem:shape-of-geodesic-horoball} in the horoball. Let $L$ be the length of vertical segment, then:
    \[
        2^L \leq d_\G(p,q) \leq 2^{L+2}.
    \]
\end{lemma}

For all $r\in \N \cup \{0\} $, we construct a coned-off horoball $\Hcal_r(\G)$ as follows: we add one vertex $v$ and edges connecting $v$ to all vertices in $\mathcal{D}_r$. We call $v$ the \emph{apex}. The additional edges are called the \emph{cone edges} of $\Hcal_r(\G)$. 

\begin{remark}
    \label{rmk:coned-off-horoball-connected}
    Notice that when $\G$ is disconnected,  $\Hcal(\G)$ is also disconnected. However, $\Hcal_r(\G)$ is always connected.
\end{remark}

\begin{definition}
    \label{def:penetrate}
    Let $c$ be a path in a cusped Cayley graph. Let $Q$ be a horoball. We say $c$ \emph{penetrates the horoball $Q$ to depth $d$} if the maximal depth of vertices in $c\cap Q$ is $d$. We say $c$ \emph{penetrates the horoball trivially} if $c$ doesn't intersect the horoball or intersects the horoball only at its base graph (and thus have depth $0$).
    
    In a cusped horoball, we declare the apices to be at depth $r+1$.
\end{definition}

\begin{lemma}
\label{lem:shape-geodesic-cusped-horoball}
    Let $\Hcal_r (\G) $ be a cusped combinatorial horoball. Suppose that $x, y \in \Hcal_r (\G)$ are distinct vertices. Then there is a geodesic $\g(x, y) = \g(y, x)$ between $x$ and $y$ which consists of at most two vertical segments and either a single horizontal segment of length at most $3$, or at most two cone edges.
    
    Moreover, any other geodesic between $x$ and $y$ is Hausdorff distance at most $5$ from this geodesic.
\end{lemma}
\begin{proof}
    Let $\Hcal(\G)\subseteq \Hcal_r(\G)$ be the combinatorial horoball and let $a$ be the apex.    Let $\g_0$ be the geodesic $\Hcal(\G)$ from \cref{lem:shape-of-geodesic-horoball} that connects $x,y$. 
    
    If $\g_0$ penetrates $\Hcal(\G)$ to depth less than $r$, then it remains a geodesic in $\Hcal_r(\G)$. 
    Suppose now $\g_0$ penetrates $\Hcal(\G)$ to depth equal to or larger than $r$. 
    Let $u$ (resp. $v$) be the vertex on $\g_0$ closes to $x$ (resp. $y$) with depth equal to or larger than $r$. Notice that it is possible that $x=u$ or $y=v$. 
    
    If the subsegment of $\g_0$ between $u,v$ has length one, then $\g_0$ remains a geodesic in $ \Hcal_r(\G)$. 
    Suppose now  the subsegment of $\g_0$ between $u,v$ has length larger than one. 
    Let $\g$ be the concatenation of $[x,u]$, $\{u,a\}$, $\{a,v\}$ and $[v,y]$. Then $l(\g)\leq l(\g_0)$.
    On the other hand, observe that if $\g$ is not a geodesic, then there exists some path $c$ shorter than $\g$. But then $c$ can not contain the apex, hence is a path in $\Hcal(\G)$, contradicting the fact that $\g_0$ is a geodesic in $\Hcal(\G)$. Therefore $\g$ is a geodesic.

    Consider now any other geodesic $\g'$ connecting $x,y$. 

    If $\g'$ does not contain the apex, then $\g'$ is a path in $\Hcal(\G)$. Therefore we must have $l(\g') = l(\g) = l(\g_0)$. In this case, $\g_0$ and $\g$ has Hausdorff distance $1$. It follows that $\g'$ is Hausdorff distance at most $5$ away from $\g$.

    Suppose now $\g'$ contains the apex.  Then $\g'$ consists of a (possible trivial) vertical segment $\g'_x$ starting at $x$, two cone edges, and a (possible trivial) vertical segment $\g'_y$ ending at $y$. Observe that $\g'_x, \g'_y$ must satisfy the following inequalities:
    \begin{align*}
        l(\g'_x)\geq & d[x,u] ,\\
        l(\g'_y) \geq & d[v,y] ,\\
        l(\g'_x ) + l(\g'_y) = & d[x,u] + d[v,y],
    \end{align*}
    It follows that $\g'$ has to be $\g$.
\end{proof}

Let $G$ be a group and let $\Pcal = \{P_i\}_{i\in I}$ be a collection of infinite subgroups of $G$. Suppose $\Pcal \he (G,X)$ for some relative generating set $X\subseteq G$.
The group $G$ can be regarded as a quotient of the free product $F=(*_{i\in I} P_i)*F(X)$ where $F(X)$ is the free group with the basis $X$.  
Let $\Rcal \subseteq F$ be a set whose normal closure is the kernel of the quotient map $F\to G$. For every $i\in I$, let $S_i$ be the set of all words over the alphabet $P_i$ that represents the identity in $P_i$. Let $\Scal = \cup_{i\in I} \Scal_i$. Then $G$ has a \emph{relative presentation}:  
\begin{align}
\label{eq:g-pre}
    G = \langle X, \Pcal \mid \Scal \cup \Rcal \rangle .
\end{align}

A relative presentation is \emph{strongly bounded} if it is bounded (that is, words in $\Rcal$ have bounded length), and for every $i\in I$, the set of letters from $P_i$ that appears in relators in $\Rcal$ is finite. By \cite[Theorem 4.24]{dahmaniHyperbolicallyEmbeddedSubgroups2017}, $\Pcal \he (G,X)$ if and only if there exists a strongly bounded relative presentation of $G$ with respect to $X$ and $\Pcal$  with linear relative isoperimetric function. 

Let $\langle X, \Pcal \mid \Scal \cup \Rcal \rangle$ be a strongly bounded presentation of $G$. Let $X_i\subseteq P_i$ be the set of all letters from $P_i \backslash \{1\}$ that appear in words from $\Rcal$.  Fix also $r\in \mathbb{N} \cup \{0\}$. To those data we associate a graph $\K_r = \K(G,X,\cup_{i} X_i,\{ P_i \})$. 

\begin{definition}[{\cite[Definition 6.44]{dahmaniHyperbolicallyEmbeddedSubgroups2017}}]
    \label{def:kr}
    Let $\Xcal = (\cup_{i} X_i) \cup X$. Let $\G(G, \Xcal)$ be the Cayley graph of $G$ with respect to the set $\Xcal$. Let $\G(P_i, X_i)$ be the Cayley graph of $P_i$ with respect to $X_i$.

    In what follows, $g\G (P_i,X_i)$ denotes the image of $\G (P_i,X_i)$ under the left action of $G$ on $\G (G,\Xcal)$. For each $i$ we fix a set of representatives $T_i$ of the left cosets of $P_i$ on $G$. Let 
\begin{align}
    \label{def:base-graph}
    \Qcal = \{ g\G (P_i,X_i)\mid 1\leq i\leq m, g\in T_i \}.
\end{align}
    Let $\K_r(G,X,\cup_{i} X_i, \{P_i\})$ be the graph obtained from $\G(G,\Xcal)$ by attaching $\Hcal_r(Q)$ to every $Q\in \Qcal$ via the attaching map $(q,0)\mapsto q$, $q\in Q$. 
    We call it \emph{the cusped Cayley graph coned-off at depth $r$}, or simply \emph{the coned-off cusped Cayley graph} when $r$ is clear.
\end{definition}

\begin{definition}
    \label{def:regular-geodesic}
    We say a geodesic in $\Hcal_r (\G) $ is \emph{regular} if it has the shape given by \cref{lem:shape-geodesic-cusped-horoball}. We say a geodesic in a coned-off cusped Cayley graph is \emph{regular} if whenever it intersects a horoball, the maximal subsegment in the horoball is regular.  
\end{definition}

In the rest of the paper, we always assume that geodesics in any coned-off cusped Cayley graph are regular.

\begin{remark}
\label{rmk:base-graph-of-Kr-connect}
    Note that $\G(G,\Xcal)$ is not necessarily connected. Indeed it is connected if and only if $X$ is a relative generating set of $G$ with respect to the subgroups $\langle X_i \rangle$, which is not always the case. However, recall that by \cite[Corollary 4.27]{dahmaniHyperbolicallyEmbeddedSubgroups2017}, if $\Pcal\he (G,X)$, then  $\Pcal$ is hyperbolically embedded in $G$ with respect to  the union of $X$ and  any  finite subset of $G$. Therefore if $G$ is finitely generated, by adding finitely many elements to $X$ if necessary, we may assume that $X$ generates $G$ and thus $\G(G,\Xcal)$ is connected. 
\end{remark}

\begin{lemma}[{\cite[Lemma 6.45]{dahmaniHyperbolicallyEmbeddedSubgroups2017}}]
     \label{lem:uniform-hyperbolicity-for-coned-off-cusped-space-Kr}
     There exists $\d >0$ such that for every $r\in \N \cup \{0\}$, the graph $ \K_r(G,X,\cup_{i} X_i, \{P_i\})$ is $\d$--hyperbolic.
\end{lemma}

\begin{lemma}
    \label{lem:coned-off-horoball-is-convex}
    Suppose $G$ is a finitely generated group and $\Pcal \he (G,X)$. Suppose further that $X$ generates $G$.
    Let $\d$ be the hyperbolicity constant such that for all $r \in \N \cup \{ 0 \}$,  $\Gbb_r = \K_r(G,X,\cup_{i} X_i, \{P_i\})$ is $\d$-hyperbolic. For any $L>\d$,  any coned-off $L$-horoball is convex in $ \Gbb_r$.
\end{lemma}
\begin{proof}
    Let $H$ be a coned-off horoball and let $\G(P_i, X_i)$ be its base graph. Let $H_1$ be the coned-off $1$-horoball contained in $H$, and let $H_L$ be the coned-off $L$-horoball contained in $H_1$. We observe that $H_L$ is convex in $H_1$. Thus if $H_L$ fails to be convex in $\Gbb_r$, then two points in $H_L$ are connected by a geodesic which contains vertices at depth $0$.  As a result this geodesic has length at least $2L+1$.

    Let $p,q$ be two such points in $H_L$ with minimal distance in the path metric on $H_1$. Notice that $d_{H_1}(p,q) \geq d_{\Gbb_r}(p,q) \geq 2L$. Notice also both $p,q$ have depth $L$. 

    \begin{claim*}
        There exists $z\in H_L$ so that $\max\{ d_{H_1}(p,z), d_{H_1}(z,q)  \}< d_{H_1}(p,q)$.
    \end{claim*}
        Let $p_0, q_0$ be the vertices in $\G(P_i, X_i)$ that are connected by a vertical segment of length $L$ to $p, q$ respectively.
        Let $\g_0$ be the geodesic connecting $p_0,q_0$ and let $\g_H$ be the geodesic in $H_L$ connecting $p,q$. Then $\g_0$, $\g_H$, and the two vertical segments form a loop $c$ in $\Gbb_r$.    
        The proof of \cite[Lemma 6.45]{dahmaniHyperbolicallyEmbeddedSubgroups2017} then implies that there exists a path $c'$ connecting $p_0, q_0$ in $\G(P_i, X_i)$. 

        Consider the concatenation of two vertical segments and $\g_H$. This is a geodesic between $p_0, q_0$. Indeed, consider a geodesic $\g'$ between $p_0, q_0$. If the penetration depth is at most $L-1$, then there exists a path of at most depth $2$ between $p$ and $q$, contradicting the assumption that $\g_H$ has length larger than $2$. So $\g'$ penetrates the horoball to depth at least $L$, but that means $\g'$ has to pass through $p,q$. 
        
        Let $c_0$ be the shortest path in $\G(P_i,X_i)$ that connects $p_0, q_0$, we then have 
        \[
         2^{L-1} d_{H_1}(p,q) \leq  l(c_0) \leq 2^{L+1} d_{H_1}(p,q)
        \]

        So there exists at least one vertex $z_0$ in the interior of $c_0$. Let $z$ be the vertex in $H_L$ that is connected to $z_0$ by a vertical segment of length $L$, then $z$ has the desired property.
    
        There are then geodesics $[p,z]$ and $[q,z]$ that lie entirely in $H_L$ by the choice of $p,q$. Consider the triangle formed by these geodesics and $[p,q]$ and consider the vertex on $[p,q]$ that has depth $0$. It follows that $L<\d$, contradicting the assumption of $L$. 
\end{proof}

\subsection{Dehn Fillings}
\label{sec:Dehn-filling-prepare}

Let $G$ be a finitely generated group. Let $\Pcal = \{ P_1,\dots,P_n \}$ be a set of infinite subgroup of $G$.
Given a collection of subgroups $\mathfrak{N} = \{ N_i \vartriangleleft P_i \mid i=1,\dots,n  \}$, the \emph{filling} of $G$  is the quotient group $G(N_1,\dots, N_n) = G/\langle \langle \cup_{i} N_i \rangle \rangle ^G$. We sometimes denote $G(N_1,\dots, N_n)$ as $\bar G$ when the filling is clear. The elements of $\mathfrak{N}$ are called \emph{filling kernels}.

The following theorem is part of \cite[Theorem 7.19]{dahmaniHyperbolicallyEmbeddedSubgroups2017}.

\begin{theorem}
    \label{thm:DGO-dehn-filling-of-HEG}
    Let $G$ be a group, $X$ a subset of $G$, $\Pcal = \{ P_1,\dots,P_n\}$ a collection of subgroups of $G$. Suppose that $\Pcal \he (G,X)$. Then for any finite subset $Z\subseteq G$ there exists a family of finite subsets $F_i \subseteq H_i \backslash \{1\}$ such that for every collection $\mathfrak{N} = \{ N_i \vartriangleleft P_i \mid i=1,\dots,n  \}$ satisfying $N_i \cap F_i = \emptyset$ the following holds.

    \begin{enumerate}
        \item The natural map from $P_i / N_i$ to $\bar G$ is injective for every $i =1,\dots,n$.
        \item $\{P_i / N_i \}_{i=1,\dots,n} \he (\bar G,\bar X)$, where $\bar X$ is the image of $X$ in $\bar G$.
        \item The natural epimorphism $\epsilon: G\to \bar G$ is injective on $Z$.
    \end{enumerate}
\end{theorem}

Let $H$ be a group and $\Dcal = \{ D_1,\dots, D_m \}$ a finite collection of hyperbolically embedded subgroups of $H$. Let $\phi: H\rightarrow G$ be a homomorphism. We say that the map $\phi$ \emph{respects the peripheral structure on $H$} if every $\phi(D_i) \in \Dcal$ conjugates in $G$ into some $P_j \in \Pcal$. In this case, any filling of $G$ induces a filling of $H$ as follows.

\begin{definition}
    \label{def:InducedFilling}
    For each $D_i$, there exists $c_i \in G$ and $P_{j_i}\in \Pcal$ such that $\phi (D_i)\subseteq c_i P_{j_i} c_i^{-1} $. The induced filling kernels $K_i\vartriangleleft D_i$ are given by
    \[
        K_i = \phi^{-1} (c_i N_{j_i} c_i^{-1} ) \cap D_i.
    \]
    The induced filling is $\bar H = H(K_1,\dots, K_m)$. The map $\phi$ induces a homomorphism $\Bar{\phi} : H(K_1,\dots,K_m) \rightarrow G(N_1,\dots,N_n) $.
\end{definition}

\begin{definition}
\label{def:hfilling}
Suppose $H<G$ and $\Dcal = \{D_1,\dots,D_m\}\he H$. Suppose also the inclusion of $H$ into $G$ respects the peripheral structure.

A filling $G\to \bar G$ is an \emph{$H$--filling} if for all $g\in G$, $| H\cap P_i^g| = \infty$ implies $N_i^g\subseteq D_j^s \subseteq H$ for some $s\in H$ and $D_j\in \Dcal$. 
\end{definition}

A property $\mathsf{P}$ holds \emph{for all sufficiently long fillings} of $\gp$ if there is a finite set $F \subseteq G$ so that $\mathsf{P}$ holds whenever $\Ncal \cap F = \emptyset$.  
Let $\mathsf{A}$ be a property of fillings. We say \emph{$\mathsf{P}$ holds for all sufficiently long $\mathsf{A}$--fillings} if for all sufficiently long fillings, $\mathsf{P}$ holds only when $\mathsf{A}$ holds.

Recall \cref{lem:uniform-hyperbolicity-for-coned-off-cusped-space-Kr}. A combination of \cite[Corollary 6.36]{dahmaniHyperbolicallyEmbeddedSubgroups2017} and \cite[Proposition 5.28]{dahmaniHyperbolicallyEmbeddedSubgroups2017} gives the following property.

\begin{proposition}
    \label{prop:quotient-space-hyperbolic}
    There exists $r_\d>0$ and $\d>0$ such that for all $r>r_\d$, and for all sufficiently long fillings with filling kernel $N$, the quotient space $\bar{\K}_r = \K_r/N$ is $\d$-hyperbolic.
\end{proposition}


\section{Quasiconvexity}
\label{sec:qc}

We extend the concept of relatively quasiconvexity to the setting of groups with hyperbolically embedded subgroups. In the context of relatively hyperbolic group, there are many equivalent definitions of relatively quasiconvex subgroups, see for example \cite{hruskaRelativeHyperbolicityRelative2010}. \cite[Definition 1.2]{martinez-pedrozaRelativeQuasiconvexityUsing2011} gives a characterization of relatively quasiconvex subgroup using fine graphs.
The following definition generalizes \cite[Definition 1.2]{martinez-pedrozaRelativeQuasiconvexityUsing2011}.

\begin{definition}[$\gp$--quasiconvex]
    \label{def:qc_finegraph}
    Let $G$ be a finitely generated group, $\Pcal$ a finite family of infinite subgroups which is hyperbolically embedded into $G$. Let $H$ be a subgroup of $G$. 
    
    We say that $H$ is \emph{$\gp$--quasiconvex} if there exists a $\gp$--graph $K$, and a nonempty connected,  $H$-invariant and quasi-isometrically embedded subgraph $L$ of $K$ so that $L$ has finitely many $H$-orbits of vertices.

    A subgroup $H$ is \emph{full} $\gp$--quasiconvex if $H$ is $\gp$--quasiconvex, and for every $P\in \Pcal$ that intersects $H$ infinitely, $H\cap P$ is a finite index subgroup of $P$.
\end{definition}

\begin{assumption}
    \label{assume:fg-heg}
    We assume $G$ is a finitely generated group, $\Pcal = \{ P_i \}_{1\leq i \leq n} $ a finite collection of subgroups of $G$, and  $X$ a subset of $G$ such that $\Pcal\he (G,X)$. We assume further that $X$ generates $G$. Assume also $H$ is a finitely generated subgroup of $G$ and is $\gp$--quasiconvex.
\end{assumption}

\begin{theorem}
    \label{thm:subgraph-is-acy}
    Make \cref{assume:fg-heg}. There is a finite collection $\Dcal$ of parabolic subgroups of $H$ such that $\Dcal\he H$. Moreover, every $D\in \Dcal$ can be conjugated into some $P\in \Pcal$.
\end{theorem}
\begin{proof}
    Since $H$ is $\gp$--quasiconvex, there is a $\gp$--graph $K$ and a nontrivial connected and quasi-isometrically embedded subgraph $L\subseteq K$ which is $H$-invariant and has finitely many $H$-orbits of vertices. Being a quasi-isometrically embedded subspace of a hyperbolic space, $L$ is also hyperbolic. Since  $V_\infty(L) \subseteq V_\infty(K)$, $L$ is fine at $V_\infty(L) $. 
    
    Since $G$-stabilizers of edges of $K$ are finite, $H$-stabilizers of edges of $L$ are also finite. Let $\Dcal$ be the collection of infinite $H$-stabilizers of vertices of $L$. Since $L$ has finitely many $H$-orbits of vertices, $\Dcal$ is a finite collection. 

    The above analysis shows that $L$ is a $\hd$--graph as in \cref{def:GPgraph}. Hence by \cref{def:heg_finegraph} $\Dcal\he H$. In particular, since $L$ is embedded into $K$, every infinite valence vertex in $L$ is an infinite valence vertex in $K$. It follows that every $D\in \Dcal$ can be conjugated into some $P\in \Pcal$.
\end{proof}

\subsection{Quasiconvexity with Coned-off Cayley graph}

It is convenient to work in the setting of the coned-off Cayley graph. In this section we prove a characterization of $\gp$-quasiconvexity using the coned-off Cayley graph.

\begin{definition}
    \label{def:extension-of-map-to-coned-off}
    
Suppose $\Dcal \he (H,Y)$ and $\Pcal \he (G,X)$. Let $\phi :H \to G$ be a homomorphism. If every $\phi(D) \in \Dcal$  is conjugate in $G$ into some $P\in \Pcal$, we say that the map $\phi$ \emph{respects the peripheral structure on $H$}.

In that case, $\phi$ extends to an $H$-equivariant map $\check{\phi} : \GH \to \GG$ given by 
\begin{itemize}
    \item For each $h\in H$,   $ h \mapsto \phi(h)$.
    \item For each pair $h,g\in H$, the edge $\{h,g\}$  is sent to $\{ \phi(h),\phi(g) \}$.
    
    \item For each $s\in H$ and $D\in \Dcal$, $D$ is contained in $P^{c_D}$ for some $c_D\in G$ and $P\in \Pcal$. Then $sD \mapsto \phi(s)c_D P$.
    \item For each $c_D$, let $\nu_D$ be a geodesic connecting $1$ and $c_D$ in $\GG$.  Any cone edge $\{h, hD\}$ is sent to the concatenation of $\phi(h).\nu_D$ and $\{\phi(h)c_D, \phi(h)c_DP\}$.  
\end{itemize}

\end{definition}

By \cref{thm:subgraph-is-acy}, if $H$ is $\gp$--quasiconvex, then the inclusion $\iota: H \to G$ respects the peripheral structure on $H$. By construction, $\check \iota$ is injective on vertices. Since $\Dcal$ is a finite collection, there exists $\l=\l(\phi)>0$ such that $l(\nu_D) <\l$. As a result $\check\iota $ is $\l$-Lipschitz.
We prove below that $\check \iota$ is a quasi-isometric embedding.

We first promote $L$ and $K$ to thick graphs (see \cref{def:thick_graph}).

\begin{lemma}
\label{lem:qc-thick-graph}
    Make \cref{assume:fg-heg}. There exists a thick $\gp$--graph $K$ and a thick $\hd$--graph $L$ that satisfies all conditions in \cref{def:qc_finegraph}. 

    Moreover, one can choose $K$ and $L$ so that the vertex in $L$ with trivial $H$--stabilizer is the vertex $u_0\in K$ with trivial $G$--stabilizer (See \cref{def:thick_graph} (1)).
\end{lemma}
\begin{proof}
    Since $H$ is $\gp$--quasiconvex, there exists $L_0\subseteq K_0$ that realizes the $\gp$--quasiconvexity of $H$. By \cite[Remark 5.3]{martinez-pedrozaNoteHyperbolicallyEmbedded2022}, we can assume without loss of generality that there exists $a_0 \in L_0 \subseteq K_0 $ such that $G_{a_0} = \{1\}$. Choose $a_0$ as $u_0$. 

    Assume now $L_0\subseteq K_0$ and that $a_0=u_0$.  As in the proof of \cite[Proposition 5.2]{martinez-pedrozaNoteHyperbolicallyEmbedded2022}, by attaching a finite sequence of $H$-orbits of edges to $L_0$ if necessary, we obtain a  thick $\hd$--graph $L$. For each $H$-orbit of $L_0$ attached, we attach a $G$-orbit of edges to $K_0$, and thus obtain a $\gp$--graph $K_1$ that contains $L$ as a subgraph. Now by attaching a finite sequence of $G$-orbits of edges to $K_1$ if necessary, we obtain a thick $\gp$--graph $K$ that contains $L$ as a subgraph. 

    Observe that $L\subseteq K$ is nonempty connected, $H$--invariant and has finitely many $H$--orbits of vertices. By \cite[Corollary 4.3]{martinez-pedrozaNoteHyperbolicallyEmbedded2022}, $L$ and $K$ are quasi-isometric to $L_0$ and $K_0$ respectively. It follows that $L$ is quasi-isometrically embedded in $K$.  Hence $L$ and $K$ realize the $\gp$--quasiconvexity of $H$.
\end{proof}

\cref{prop:qi_to_conedoff} then gives a quasi-isometric embedding between coned-off Cayley graphs, which is exactly the extension of inclusion.

\begin{lemma}
\label{lem:qi-embed-injective}
    Make \cref{assume:fg-heg}. There exists $Y\subseteq H$ and $\Dcal$ a collection of infinite subgroups of $H$ such that $\GH$ is an $\hd$--graph. 
 
    Moreover, $\check \iota : \GH \to \GG$ is a quasi-isometric embedding.
\end{lemma}
\begin{proof}

    By \cref{def:qc_finegraph}, there exists a $\gp$--graph $K$ and a nontrivial connected and quasi-isometrically embedded subgraph $L\subseteq K$ which is $H$-invariant and has finitely many $H$-orbits of vertices.
    By \cref{lem:qc-thick-graph}, we may assume that $L$ is a thick $\hd$--graph and $K$ is a thick $\gp$--graph.
    By \cref{thm:subgraph-is-acy}, there exists a collection $\Dcal$ of subgroups of $H$ such that $L$ is an $\hd$--graph.
    
    By \cref{prop:qi_to_conedoff}, there exists $X\subseteq G$ and $Y\subseteq H$ such that $L,K$ are quasi-isometric to $\GH$, $\GG$ respectively. Then the composition of the quasi-isometry $\GH\to L$ from \cref{prop:qi_to_conedoff}, the quasi-isometric embedding $L\to K$, and the quasi-isometry $K\to  \hat\G(G,\Pcal,X)$ from \cref{cons:inverse_qi_from_conedoff} 
    is a quasi-isometric embedding $\hat{\G}(H,\Dcal, Y)\to \hat\G(G,\Pcal,X)$. We denote this quasi-isometric embedding as $\phi$.
        
    \begin{center}
        \begin{tikzcd}
    	\GH & \GG \\
    	L & K
    	\arrow["{\phi}", hook, from=1-1, to=1-2]
    	\arrow[tail reversed, from=1-1, to=2-1]
    	\arrow[tail reversed, from=1-2, to=2-2]
    	\arrow["{\subseteq}", hook, from=2-1, to=2-2]
    \end{tikzcd}
    \end{center}

    It remains to show that $\phi$ coincides with $\check\iota$ at all vertices.

    Let $\{u_0, u_1,\dots, u_l \}$, $\{v_1,\dots,v_n\}$  (resp. $\{a_0, a_1,\dots,a_k \}$, $\{b_1,\dots,b_m \}$) be vertices in $K$ (resp. $L$) that satisfy the properties in \cref{def:thick_graph}. By \cref{lem:qc-thick-graph}, we can assume that $a_0$ is $u_0$. By the construction of the quasi-isometry in \cref{prop:qi_to_conedoff}, any element $h\in H$ is mapped to $h.a_0$, then to $h.u_0$, then to $h \in G$ as a vertex of $\GG$. So $\phi(h) = h$.

    As for the apices, let $sD_i$ be an apex in $\GH$. It is sent to $s.b_i \in L$.  
    Since $L$ is a subgraph of $K$, the image of $s.b_i$ in $K$ is a translation of $v_j$ for some $j$. There exists $c_{D_i} \in G$ such that the image of $s.b_i$ is $sc_{D_i}.v_j$.  
    Since the stabilizer of $v_j$ is $P_j$, the image of $sc_{D_i}.v_j$ is $sc_{D_i}P_j$. Therefore $\phi(sD_i) = sc_{D_i}P_j$.
\end{proof}

    Notice that in the setting of \cref{lem:qi-embed-injective}, $Y$ is not necessarily a subset of $X$.  \cref{prop:conedoff-embed-injective-subset-generator} shows that we can choose $X,Y$ such that $Y\subseteq X$. As a result, there is a natural $H$--equivariant inclusion of $\G(H,\Dcal, Y)$ as a subgraph of $ \G(G,\Pcal , X)$. We remark that this does not mean that $\GH$ is a subgraph of $\GG$.

\begin{proposition}
    \label{prop:conedoff-embed-injective-subset-generator}
    Make \cref{assume:fg-heg}. Then there exists $X\subseteq G$ and $Y\subseteq H\cap X$ so that:
    \begin{enumerate}
        \item $\Pcal\he (G,X)$,
        \item $\Dcal \he (H,Y)$, and
        \item the extension of inclusion $\check \iota: \GH \to \GG$ is a quasi-isometric embedding.
    \end{enumerate}
\end{proposition}
\begin{proof}
    As shown in the proof of \cref{lem:qi-embed-injective}, there exists a thick $\gp$--graph $K$ and a thick $\hd$--graph $L$ so that $L$ is a quasi-isometrically embedded subgraph of $K$.
    
    Let $\{u_0, u_1,\dots, u_l \}$, $\{v_1,\dots,v_n\}$  (resp. $\{a_0, a_1,\dots,a_k \}$, $\{b_1,\dots,b_m \}$) be as in the proof of \cref{lem:qi-embed-injective}. Let $Y$ (resp. $Z$) be the generating set for $H$ (resp. $G$) produced by \cref{prop:qi_to_conedoff}.   
    Then $\GH$ is an $\hd$--graph while $\GGx{Z}$ is a $\gp$--graph. By \cref{lem:qi-embed-injective}, $\GH$ is quasi-isometrically embedded in $\GGx{Z}$. 
    
    Let $X = Z\cup Y$.     Let $q_Z:\GGx{Z} \to K$ be the quasi-isometry constructed in \cref{prop:qi_to_conedoff}.
    Consider now $q_X:\GG\to K$ given by $g\mapsto g.u_0$ and $gP_i \mapsto g.v_i$.

    \begin{center}
    \begin{tikzcd}
	\GH & \GGx{Z} &\GG \\
	L & K &
	\arrow[hook, from=1-1, to=1-2]
    \arrow["{\subseteq}", from=1-2, to=1-3]
	\arrow[tail reversed, from=1-1, to=2-1]
	\arrow["{q_Z}", tail reversed, from=1-2, to=2-2]
    \arrow["{q_X}", tail reversed, from=1-3, to=2-2]
	\arrow["{\subseteq}", hook, from=2-1, to=2-2]
    \end{tikzcd}
    \end{center}

    \begin{claim*}
        The map $q_X$ is a quasi-isometry.
    \end{claim*}
    For simplicity, we will write $d_Z$ for distance in $\GGx{Z}$ and $d_X$ for distance in $\GG$. As proved in the proof of \cref{prop:qi_to_conedoff}, for any $a,b\in V(\GGx{Z})$, \[
    d_K(q_Z(a), q_Z(b) ) \leq 3 d_Z(a,b), \quad d_Z(a,b) \leq 2 d_K(q_Z(a), q_Z(b) ).
    \]

    Since $Z\subseteq X$, $\GGx{Z}\subseteq \GG$. Therefore, for any $a,b\in V(\GGx{Z})$, $d_X(a,b)\leq d_Z(a,b)$.  Notice that $\GG$ and $\GGx{Z}$ have same vertices. Moreover, for any vertex $a\in V(\GG) = V(\GGx{Z})$, $q_Z(a) = q_X(a)$. We immediately have
    \[d_X(a,b) \leq d_Z(a,b) \leq 2d_K(q_Z(a),q_Z(b) ) = 2d_K(q_X(a),q_X(b) )  \]  for any two vertices $a,b$ in $V(\GG)$.

    On the other hand, since $L\subseteq K$, each $a_i\in L$ can be identified as $g_i.a_i = u_{i'}$ for some $u_{i'}\in K$ and some $g_i\in G$. Similarly, each $b_s\in L$ can be identified as $h_s.b_s = v_{s'}$ for some $v_{s'}\in K$ and $h_s\in G$. Notice that the choice of $g_i$ (resp. $h_s$) is not unique. We choose one representative for each $a_i$ and each $b_s$. Let $A$ be the collection of those representatives. In particular, $A$ is a finite set. 
    
    Suppose $\{f,g\}$ is an edge in $\GG \backslash \GGx{Z}$. Then it must be labeled by some $y\in X\backslash Z\subseteq Y$. By the construction of $Y$, either $d_L(a_i, y.a_j) =1 $ for some $0\leq i,j\leq k$, or $d_L(a_i, y.b_s) =1$ for some $0\leq i\leq k$, $1\leq s\leq m$. In the first case, since $L$ is a subgraph of $K$, we have  $d_K(g_i.u_{i'},yg_j.u_{j'})=1$. It follows that  $g_i^{-1}yg_j \in Z$.
    Analogously, in the second case we have  $d_K(g_i.u_{i'},yh_s.v_{s'})=1$  for some $g_i, h_s\in A$. It follows that $g_i^{-1}y h_s \in Z$. 
    Either way, there exists $a,a'\in A$ and $z\in Z$ such that $y=aza'$.

    Since $\GGx{Z}$ is connected,  for each $a\in A$ there exists a path of finite length in $\GGx{Z}$ that represents $a$. Let $l_A$ be the maximal length of all such paths.   
    Thus there exists a path of length $2l_A+1$ connecting $f,g$ in $\GGx{Z}$.  Therefore for any $a,b\in V(\GG)$, \[
    d_Z(a,b)\leq (2l_A+1)d_X(a,b),
    \] which implies
    
    \[
    d_K(q_X(a), q_X(b)) = d_K(q_Z(a),q_Z(b)  \leq 3 d_Z(a,b) \leq 3(2l_A+1)d_X(a,b).
    \]

    As a result, $\GG$ is quasi-isometric to $\GGx{Z}$. 
    
\begin{claim*}
    The graph $\GG$ is fine at cone vertices.
\end{claim*}
    Let $Z'=Z\cup A$. Since $A$ is finite, by \cite[Corollary 4.27]{dahmaniHyperbolicallyEmbeddedSubgroups2017} $\Pcal \he (G,Z)$ if and only if $\Pcal\he (G,Z')$. In particular, $\GGx{Z'}$ is fine at cone vertices.

    Every edge in $\GG$ of the form $\{g,gx\}$ with $g\in G$ and $x\in Z$ corresponds to an edge $\{g,gx\}$ in $\GGx{Z'}$. Every cone edge in $\GG$ corresponds to a cone edge in $\GGx{Z'}$.

    Moreover, since for any $y\in X\backslash Z$, there exists $a,a'\in A$ and $z\in Z$ such that $y=aza'$, every edge in $\GG$ of the form $\{g,gy\}$ with $g\in G$ and $y\in X\backslash Z$ corresponds to a path of length at most three in $\GGx{Z'}$ between $g$ and $gy$. In particular, this path does not contain cone vertices. 

    In this way, every path $\a\in \GG$ from $e$ to $g$ in $\GG$ corresponds to a path $\a'$ from $e$ to $g$ in $\GGx{Z'}$. Moreover, $\a$ pass through a cone vertex $fP_i \in \GG$ if and only if $\a'$ pass through a cone vertex $fP_i \in \GGx{Z'}$. We also have $| \a'| \leq 3 | \a |$.

    Suppose $\GG$ is not fine at some cone vertex $a$, then there exists a positive number $n$ such that there exists infinitely many paths of length at most $n$ connecting two vertices in $T_a {\GG}$ (Recall that in \cref{def:angle} we define this notation as the set of vertices adjacent to $a$). But this implies there exists infinitely many paths of length at most  $3 n$  connecting vertices in  $T_a{ \GGx{Z}}$, which contradicts the assumption that $\GGx{Z}$ is fine at cone vertices.
    Hence $\GG$ is fine at cone vertices.

    Therefore, $\Pcal\he (G,X)$. In particular, $\GG$ is a $\gp$--graph.

\begin{claim*}
    The map  $\GH \hookrightarrow \GG$ is a quasi-isometric embedding and coincides with $\check\iota$ at all vertices.
\end{claim*}
    Since $\GG$ is quasi-isometric to $\GGx{Z}$, $\GH$ is quasi-isometrically embedded in $\GG$. By \cref{lem:qi-embed-injective}, $\GH\hookrightarrow \GGx{Z}$ maps $h\in H$ to $h\in G$ and $sD_i$ to $sc_{D_i}P_j$. Since $\GG$ and $\GGx{Z}$ has the same vertex set, the composition map coincides with the extension of inclusion.
\end{proof}

We now give a characterization of $\gp$--quasiconvexity using coned-off Cayley graphs.
\begin{proposition}
\label{prop:equiv-gp-qc-by-coned-off-cayley-graph}
    Let $G$ be a finitely generated group, $\Pcal$ a finite collection of subgroups of $G$, and $X$ a subset of $G$ such that $G=\langle X \rangle$ and $\Pcal \he (G,X)$. 
    Let $H$ be a finitely generated subgroup of $G$. Then $H$ is $\gp$--quasiconvex if and only if there exists a finite collection $\Dcal$ of subgroups of $H$  and a subset $Y \subseteq H\cap X$ such that 
    \begin{enumerate}
        \item  $\Dcal \he (H,Y)$,
        \item  the inclusion $\iota:H\to G$ respects the peripheral structure on $H$, and
        \item the extension of inclusion $\check \iota: \GH \to \GG$ is a quasi-isometric embedding.
    \end{enumerate}
\end{proposition}
\begin{proof}
    \cref{prop:conedoff-embed-injective-subset-generator} proves one direction. Consider now the other direction and assume that there exists $\Dcal$ and $Y$ with the properties described above. By \cref{prop:qi_to_conedoff}, $\GG$ is a $\gp$--graph. 
    Let $L = \check \iota(\GH)$. 
    Since $\GH$ is nonempty and connected, so is $L$. 
    Recall that $\check \iota$ is an $H$-equivariant map. Since $\GH$ is $H$-invariant, so is $L$. Notice also that $\GH$ has finitely many $H$-orbits of vertices, so does $L$.
    It follows that $L$ is a nonempty, connected, $H$-invariant quasi-isometrically embedded subgraph of $\GG$ that has finitely many $H$-orbits of vertices. Thus by \cref{def:qc_finegraph}, $H$ is $\gp$--quasiconvex.
\end{proof}

\section{Angles}
\label{sec:prepare-uniform-QC}
    
    This section records several technical lemmas used in the proof of \cref{thm:uniform_qc}.  We construct the spaces we are working on in \cref{sec:basic-constructions}. In \cref{subsec:big-angle}, we prove some properties about the angle metric in coned-off Cayley graphs. In \cref{subsec:interplay}, we build maps between the coned-off Cayley graphs and the coned-off cusped Cayley graph, and inspect the behaviors of such maps.

\subsection{Basic Constructions}
\label{sec:basic-constructions}
       
\begin{construction}
    \label{cons:generating-set-in-actual-use}
    Let $A$ be a finitely generated group, $\Ocal$ a finite collection of hyperbolically embedded subgroups in $A$. Let $B$ be an $(A,\Ocal)$--quasiconvex subgroup.  
       
    By \cref{prop:conedoff-embed-injective-subset-generator} there exists a finite collection $\Ecal$ of subgroups of $B$, a set $X_0\subseteq A$ and a set $Y\subseteq B\cap X_0$ so that $\Bcal_0 = \hat \G(B,\Ecal,Y)$ is a $(B,\Ecal)$--graph and $\Acal_0 = \hat \G(A,\Ocal, X_0)$  is an $(A,\Ocal)$--graph.    
    Since both $A$ and $B$ are finitely generated, we may assume $X_0$ and $Y$ are generating sets of $A$ and $B$ respectively. Moreover, there exists a quasi-isometric embedding $\phi_0:\Bcal_0\to \Acal_0$ that is injective on vertices.
       
    Recall the construction of coned-off cusped Cayley graph (see \cref{def:kr}). Let $ \langle Y, \Ecal | \Scal, \Rcal  \rangle $ be a strongly bounded relative presentation of $B$.
    For $E_i\in \Ecal$, let $Y_i$ be the set of all elements in $E_i$ that appear in $\Rcal$ and let $\mb{B}_r = \K_r(B, Y, \cup_i Y_i, \Ocal)$ be the associated coned-off cusped Cayley graph. We remark that for each $i$, $Y_i$ is finite.
    Let $\Ycal= \cup_i Y_i \cup Y_0$, and let $\Bcal = \hat \G(B, \Ecal,\Ycal)$. Let $\G_B = \G(B,\Ycal)$.
    
    Since $\Ecal$ is a finite set, $\cup_i Y_i$ is finite. Let $ X =  \cup_i Y_i \cup X_0$. Then $\hat \G(A, \Ocal, X)$ is a thick $(A, \Ocal)$--graph and is quasi-isometric to $\hat \G(A, \Ocal, X_0)$. We can therefore find a strongly bounded relative presentation of $A$ with respect to $ X $ and construct  $\Xcal= \cup_j X_j \cup X$ similarly. We denote $\G(A,\Xcal)$ by $\G_A$ and construct $\Acal = \hat \G (A, \Ocal,\Xcal)$ and $\A_r = \K_r(A,X,\cup_j X_j ,\Ocal)$ with base graph $\G_A$.
\end{construction}
    
    By \cref{lem:coned-off-cayley-graph-add-edge-in-P}, $\Bcal$ is a thick $(B,\Ecal)$-graph. Furthermore, $\Bcal$ is quasi-isometric to $\Bcal_0$. In fact, there exists a quasi-isometry between $\Bcal$ and $\Bcal_0$ that restricts to the identity on vertices.
    
    By \cref{rmk:base-graph-of-Kr-connect} we may assume that $\G_A$ and $\G_B$ are connected.      Observe that $\Ycal \subseteq \Xcal$. As a result, $\G_B\subseteq \G_A$.

    \begin{construction}
        \label{cons:qc-induce-map-coned-off-graph}
        We construct a map $\phi: \Bcal \to \Acal$ in the following way. For each $E_i\subseteq \Ecal$, there exists $c_{E_i}\in A$ and $O_j\in \Ocal$ such that $E_i\subseteq P_j^{c_{E_i}}$. For each $c_{E_i}$, let $\nu_i$ be a geodesic connecting $1$ and $c_{E_i}$ in $\G_A$. 
        
        \begin{itemize}
            \item Any vertex or edge in $\G_B \subseteq \Bcal$ is sent to the corresponding vertex or edge in $\Acal$. 
            \item Any apex $sE_i$ is sent to the apex $sc_{E_i}O_j$.
            \item Any cone edge $\{h, sE_i\}$ is sent to the concatenation of $h.\nu_i$ and the cone edge $\{hc_{E_i}, sc_{E_i} O_j\}$. We don't specify the parametrization of the path here.
        \end{itemize}
        Notice that $\phi$ coincides with the extension of inclusion of $B\subseteq A$ at all vertices (See \cref{def:extension-of-map-to-coned-off}).
    \end{construction}
    
    \begin{lemma}
    \label{lem:qc-induce-map-is-qi}
        There exists $k_\phi,l_\phi>0$ such that for any parametrization of cone edges, the map $\phi:\Bcal \to \Acal$ is a $(k_\phi,l_\phi)$-quasi-isometric embedding that is injective on vertices and maps any path in $\Bcal$ to a path in $\Acal$. Furthermore, the set of apices that $\phi(c)$ passes through consists of images of apices in $c$. 
        
        In particular, there exists $\l_\phi >0$ such that for any parametrization of $\phi$, $l(\phi(c)) \leq \l_\phi l(c) + \l_\phi$ for every path $c$ in $\Bcal$. 
    \end{lemma}
    \begin{proof}
        By construction, there exists a quasi-isometry $f:\Acal_0 \to \Acal$ and a quasi-isometry $g:\Bcal_0 \to \Bcal$. Both $f$ and $g$ restrict to identity on vertices. Observe also that $\phi$ coincides with $\phi_0$ on vertices. Recall that by \cref{prop:conedoff-embed-injective-subset-generator} $\phi_0$ is injective on vertices. It follows that $\phi$ is a quasi-isometric embedding that is injective on vertices.

        By construction $\phi$ maps any edge in $\Bcal$ to an edge path. Thus for any path $c$, $\phi(c)$ is a path. It is also clear from the construction that $\phi(c)$ passes through an apex if and only if $c$ passes through one.
    
        Since $\Ecal$ is a finite set, there exists $\l_\phi$ such that $l(\nu_i)\leq \l_\phi-1$ for all $i$. Notice that for any subsegment $e$ of a cone edge, we have $\phi(e) \leq \l_\phi$. It follows that $l(\phi(c)) \leq \l_\phi l(c) + \l_\phi$. 
        
  \end{proof}

    \begin{construction}
        \label{construction:qc-extend-map-coned-off-cusped-space}
        For all $r>0$, we construct a map $\phi_r:\mb{B}_r \to \mb{A}_r$ in a manner similar to \cref{cons:qc-induce-map-coned-off-graph}. Since both $\mb{A}_r$ and $\Acal$ contains a copy of $\G_A$, we may use the same $c_{E_i}$ and $\nu_{E_i}$ for each $i$.
        \begin{itemize}
            \item Any vertex or edge in $\G_B \subseteq \mb{B}_r$ is sent to the corresponding vertex or edge in $\mb{A}_r$. 
            \item Any apex $sE_i$ is sent to the apex $sc_{E_i}O_j$.
            \item Any vertex in a horoball is determined by a quadruple $(h,s,E_i, d)$, where $h\in sE_i$, and $d \in \N$. We define $\phi_r(h,s,E_i,d) = (hc_{E_i}, sc_{E_i}, O_j, d)$.
            \item Any edge at depth $\geq 1$ in the coned-off horoball over $sE_i$ is sent to the corresponding edge in the coned-off horoball over $sc_{E_i}O_j$.
            \item Any vertical edge $\{ h, (h,s,E_i, 1) \}$ is sent to the concatenation of $\nu_{E_i}$ and $\{ hc_{E_i}, (hc_{E_i}, sc_{E_i}, O_j, 1) \}$. Again, we don't specify the parametrization for now.
        \end{itemize}
    \end{construction}

    A similar argument of \cref{lem:qc-induce-map-is-qi} gives the following property.
    \begin{lemma}
        \label{lem:qc-induce-map-cusp-is-Lipschitz}
        Let $\l_\phi$ be the constant from \cref{lem:qc-induce-map-is-qi}. For all $r\in \N \cup \{0\}$, and for any parametrization of vertical edges,  $\phi_r$ is an $H$-equivariant quasi-isometry such that $l(\phi(c)) \leq \l_\phi l(c) + \l_\phi$ for any path $c$ in $\mathbb{B}_r$.
    \end{lemma}

In the rest of \cref{sec:prepare-uniform-QC}, we fix the notation $A, B, \Ocal, \Ecal, \Acal, \Bcal, \mathbb{A}_r, \mb{B}_r, \G_A, \G_B$ as constructed in \cref{cons:generating-set-in-actual-use} and  $\phi:\Acal \to \Bcal$ from \cref{cons:qc-induce-map-coned-off-graph}. We fix also $\delta >0$ to be a positive integer larger than the hyperbolicity constant from  \cref{lem:uniform-hyperbolicity-for-coned-off-cusped-space-Kr} such that for every $r\in \N \cup \{0\}$, $\A_r$, $\mb{B}_r$, $\Acal$, and $\Bcal$ are $\delta$-hyperbolic spaces.

Recall the definition of $T_v\G$ (see \cref{def:angle}) and the angle metric (see \cref{def:angle_of_paths}). Suppose $b$ is an apex in $\Bcal$ and $u,v \in T_b\Bcal$. By a slight abuse of notation, we denote $\angle_{\phi(b)}(\phi(\{ u,b\}),\phi(\{ b,v \}))$ as $\angle_{\phi(b)}(\phi(u),\phi(v))$.

\subsection{The Big Angle Property of Coned-off Space}
\label{subsec:big-angle}
We first examine some geometric properties of the coned-off Cayley graph.

\begin{lemma}
    \label{lem:AngleUsingQuasiGeod}
    For all $k,l>0$, there exists  $\lambda = \lambda(\delta, k, l)$ satisfying the following: 
    
    Suppose that  $a \in \Acal$ is an apex, $c_1, c_2$ a pair of $(k,l)$-quasi-geodesics  that share an endpoint $a$ in $\Acal$ and that do not go through $a$ elsewhere, and  $\g_1, \g_2$ a pair of geodesics  that share the same endpoints with $c_1,c_2$ respectively. Then the following inequalities hold: 
    \[
    \angle_a(c_1,c_2) - \lambda \leq \angle_a(\g_1,\g_2) \leq \angle_a(c_1,c_2) + \lambda  . 
    \]

    In particular, $ \angle_a(c_1,c_2) = \infty$ if and only if  $\angle_a(\g_1,\g_2) = \infty$.
\end{lemma}

\begin{proof}
     Let $i\in \{1,2\}$. By \cite[Theorem III.H.1.7]{bridsonMetricSpacesNonPositive1999}, there exists $\lambda_0 = \lambda_0(\delta, k, l)$ such that the Hausdorff distance between $c_i$ and $\g_i$ is bounded above by $\lambda_0$ for $i=1,2$.  Let $u_i$ be the vertex adjacent to $a$ on $\g_i$, and $v_i$ be the vertex adjacent to $a$ on $c_i$ respectively.  

     Suppose first $\angle_a(\g_1,\g_2)<\infty$. 
     Let $c'_i$ (resp. $\g'_i$) be the longest subsegment of $c_i$ (resp. $\g_i$) between $v_i$ (resp. $u_i$) and $a$ for each $i$.
     Let $c$ be a path whose length realizes $\angle_a(\g_1,\g_2)$. 
     Then $c'_1\cup \g'_1\cup c\cup \g'_2 \cup c'_2$ is a path connecting $v_1,v_2$ that does not pass through $a$. It follows that $\angle_a(c_1,c_2 ) < \infty$. 
     Similarly, if $\angle_a(\g_1,\g_2)<\infty$, then $\angle_a(c_1,c_2) < \infty$. 
     Therefore  $\angle_a(c_1,c_2) = \infty$ if and only if  $\angle_a(\g_1,\g_2) = \infty$.

     Suppose now that $ \angle_a(c_1,c_2) < \infty$ and  $ \angle_a(\g_1,\g_2) < \infty$.

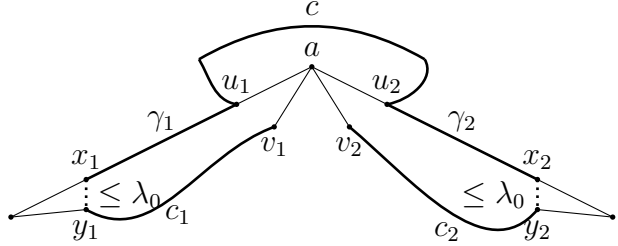
\begin{figure}[h!]
\centering
    \begin{tikzpicture}
        \fill[] (0,0) circle (1pt)  (4,2) circle (1pt) node[above]{$a$} (8,0) circle (1pt);
        \fill[] (1,0.5)  circle (1pt) node[above]{$x_1$} (3,1.5)  circle (1pt) node[above]{$u_1$} (5,1.5)  circle (1pt) node[above]{$u_2$} (7,0.5) circle (1pt) node[above]{$x_2$};
        \fill[] (1,0.1) circle (1pt)  node[below]{$y_1$} (3.5,1.2) circle (1pt)  node[below]{$v_1$} (4.5,1.2) circle (1pt)  node[below]{$v_2$} (7,0.1) circle (1pt)  node[below]{$y_2$};
        \draw (0,0) -- (4,2) node[midway, above] {$\g_1$};
        \draw[line width=1pt] (1,0.5) -- (3, 1.5);
        \draw (4,2) -- (8,0) node[midway, above] {$\g_2$};
        \draw[line width=1pt] (5,1.5) --(7,0.5);
        \draw[dotted, line width=1pt] (1,0.5) -- (1,0.1) node[midway, right] {$\leq \l_0$};
        \draw[dotted,line width=1pt] (7,0.5) -- (7,0.1) node[midway, left] {$\leq \l_0$} ; 
        \draw (0,0) -- (1,0.1);
        \draw[line width=1pt] (1,0.1) to[out=-30, in=-160] node[midway, below] {$c_1$} (3.5,1.2);
        \draw (3.5,1.2) -- (4,2);
        \draw (4,2) -- (4.5,1.2);
        \draw[line width=1pt] (4.5,1.2) to[out=-40, in=230] node[midway, below] {$c_2$} (7,0.1);
        \draw (7,0.1) -- (8,0);
        \draw[line width=1pt] (3,1.5) to[out=170, in=300] (2.5,2.1);
        \draw[line width=1pt] (2.5,2.1) to[out=30, in=150]  node[midway, above] {$c$} (5.5,2.1);
        \draw[line width=1pt]  (5.5,2.1) to[out = -60, in=10] (5,1.5);   
    \end{tikzpicture}
\caption{A path  between $v_1$ and $v_2$ that does not pass through $a$}
\label{fig:image-of-quasigeodesic-angle}
\end{figure}
     
     If $l(\g_i)>\lambda_0$, then there exists a point $x_i$ on $\g_i$ that is $\lambda_0+1$ away from $a$, and $y_i$ on $c_i$ such that $d(x_i,y_i) \leq \lambda_0$. Notice that $[x_i, y_i]$ does not contain $a$ because $\g_i$ is geodesic. In particular, $y_i$ cannot be $a$. 
     
     Let $c$ be a path whose length realizes $\angle_a(\g_1,\g_2)$, and consider the path: 
     \[
     c' = [v_1,y_1]\cup [y_1,x_1]\cup[x_1, u_1]\cup c \cup [u_2,x_2] \cup [x_2,y_2]\cup [y_2,v_2].     \]
     Then $c'$ is a path connecting $v_1,v_2$ that doesn't contain $a$. We then have:
     \[
     \angle_a(c_1,c_2) \leq l(c') \leq d(v_1,y_1) +4\lambda_0 + d(v_2,y_2) + l(c).
     \]

     Notice also that a geodesic $[a,y_i]$ has length at most $2\lambda_0 +1$. Since $c_i$ is a $(k,l)$-quasi-geodesic, so is a segment of $c_i$. We then have $d(v_i,y_i)\leq k(2\lambda_0 +1) + l -1 $.    Let $\lambda = 4\lambda_0 +  2(k(2\lambda_0 +1) + l -1)$, it follows that:
     \[
     \angle_a(c_1,c_2) \leq \angle_a(\g_1,\g_2) + \lambda.
     \]
     If $l(\g_i)\leq \lambda_0$, then let $x_i$ be the endpoint other than $a$, and let $y_i = x_i$, and we have the same conclusion.
     Similarly, considering a path whose length defined $\angle_a(c_1,c_2)$ and concatenating with the same segments as above (in a slightly different order), we have:
     \[
     \angle_a(\g_1,\g_2) \leq \angle_a(c_1,c_2) + \lambda.
     \]
\end{proof}
 
\begin{lemma}
    \label{lem:GeodpassBigAngle}
    There exists $D=D(\delta,k,l) >16 \d$ so that for every apex $a\in \Acal$ and every $(k,l)$-quasi-geodesic $c_1$ (resp. $c_2$) starting at $a$ and ending at $x$ (resp. $y$) for which $\angle_a(c_1,c_2)> D$,  any geodesic $[x,y]$ must pass through $a$.
\end{lemma}
\begin{proof}
    Let $u$ (resp. $v$) be the point on $c_1$ (resp. $c_2$) connected to $a$ by a cone edge. 
    
    First notice that,  by \cref{lem:AngleUsingQuasiGeod},  the angle of any geodesic $[x,a]$ and $[a,y]$ is bounded below by $\angle_a(u,v)-  \lambda$, where $\lambda = \lambda (\delta, k,l)$. We can thus assume $c_1$ and $c_2$ are geodesics. Without loss of generality, we can assume $\lambda >  \delta \geq 1$ are positive integers. Let $D= 16 \lambda$ and suppose $\angle_a (u,v) > D$. Since $D>2$, any geodesic $[u,v]$ must pass through $a$.

    Suppose there exists a geodesic $[x,y]$ that does not contain $a$. Then by relabeling $x, y$  if needed, there exists a pair of vertices $p \neq x \in c_1, q\in c_2$ satisfying the following properties:
    \begin{enumerate}
        \item Any geodesic connecting $p,q$ passes through $a$. 
        \item Let $p'\in [p,x]$ be the vertex next to $p$ on $c_1$.  There exists a geodesic $\g$ connecting $p'$ and $q$ that does not pass through $a$. 
    \end{enumerate}
    
    Consider the triangle $p'pq$. Since $d(p,p')=1$, $a$ lies in the $2\delta$ neighborhood of $p'q$. Notice that the length of $p'q$ is either $d(p,q)+1$ or $d(p,q)$. It cannot be shorter than $d(p,q)$, otherwise $[pp']\cup[p'q]$ is a geodesic avoiding $a$, contradicting the assumption. So the perimeter of the triangle $p'pq$ is either $2d(p,q)+1$ or $2d(p,q)+2$. Now, consider the path $c = [u,p']\cup[p',q]\cup[q,v]$. It does not contain $a$ because $c_1, c_2$ are geodesics. By assumption $l(c)>16\lambda >16\delta$. On the other hand, $c$ has two edges fewer than the triangle $p'pq$, hence  $l(c) \leq 2d(p,q)$. It follows that $d(p,q)>8\delta$ and one of $[a,p'], [a,q]$ must be longer than $4\delta$.

    Consider now the triangle $p'aq$ and calculate its Gromov product:
    \begin{align*}
        (p'\mid q)_a = & \frac 1 2\left[ d(a,p') +  d(a,q) - d(p',q) \right] \\
        =& \frac{1}{2}\left[ 1 + d(p,q) - d(p',q) \right] \\
        = &  \frac{1}{2} \text{ or } 0.
    \end{align*}
   
    Either way, the image of $u$ (resp. $v$) on the tripod lands on $[O,v_{p'}]$ (resp. $[O,v_q]$), where $O$ is the center point of the tripod. (Notice that $O$ is the image of $a$ in the case $(p'\mid q)_a = 0$.)

    Suppose $d(a,p'), d(a,q) \geq 2\delta$, then we can find a point $x$ (resp. $y$) on $[a,p']$ (resp. $[a,q]$) such that $d(a,x) = 2\delta$ (resp. $d(a,y)= 2\delta$). Let $x'$ (resp. $y'$) be the other preimage of the image of $x$ (resp. $y$) on the tripod. Since $d(x,x')$ (resp. $d(y,y')$) is bounded by $\delta$, we have $[x,x']$ and $[y,y']$ won't contain $a$. It follows that $c= [u,x]\cup [x,x'] \cup [x',y'] \cup [y'y] \cup [yv]$ is a path connecting $u,v$ and avoiding $a$.  Consider $[x', y']$ which is a subsegment of $[p',q]$. It is a geodesic, hence its length is bounded above by the length of path $[x',x]\cup [x,a]\cup [a,y] \cup [y,y']\leq 6\delta $. It follows that  $l(c)\leq 12\delta -2 $, which implies $\angle_a(u,v) \leq 12\delta -2 <D$, a contradiction.

    In the case where $[a,p']$ or $[a,q]$ is smaller than $2\delta$,  we take $x=p'$ or $y=q$. The rest of the argument still works. 

    It follows that $[p'q]$ must contain $a$, but this contradicts the choice of $p'$. Hence when $\angle_a(c_1,c_2)>D$ any geodesic connecting$[x,y]$ must pass through $a$.
\end{proof}
\begin{remark}
    \label{rmk:GeodpassBigAngle-can-be-arb-large}
    We remark that the results of \cref{lem:GeodpassBigAngle} remain true if we  replace $D$ by a larger number.
\end{remark}

Recall that we fix $B$ to be an $(A,\Ocal)$--quasiconvex subgroup and $\phi: \Bcal \to \Acal$ to be the quasi-isometric embedding from \cref{cons:qc-induce-map-coned-off-graph}.

\begin{lemma}
    \label{lem:AngleUpperBoundForHatSpace}
    There exists $M_A>0$ such that for any apex $b=sE_i$ in $\Bcal$ and any vertices $u,v \in   \Bcal$ adjacent to $b$,

     \[\angle_{\phi(b)}(\phi(u),\phi(v)) \leq M_A \angle_b(u,v).\]
\end{lemma}
\begin{proof} 
    The inequality holds trivially when $\angle_b(u,v) =\infty$. Suppose now $\angle_b(u,v)$ is finite.  Let $c$ be a path connecting $u,v$ that realizes $\angle_b(u,v)$. Then by the construction of $\phi$, $\phi (c)$ is a path connecting $\phi(u), \phi(v)$ that does not contain $ \phi (b)$. 
    By \cref{lem:qc-induce-map-is-qi}, there exists $\l_\phi >0$ such that $l(\phi(c)) \leq \l_\phi l(c)+\l_\phi$.
    
    By the construction, $\phi(\{u,b\})$ is the concatenation of $u.\nu_{E_i}$ and $\{uc_{E_i}, \phi(b) \}$. Similarly, $\phi(\{b,v\})$ is the concatenation of $v.\nu_{E_i}$ and $\{vc_{E_i},\phi(b)\}$.
    Recall that $\nu_{E_i}$ is a path in $\G_A$ of length at most $\l_\phi$. 
    It follows that the concatenation of $u.\nu_{E_i}$, $\phi(c)$, and $v.\nu_{E_i}$ is a path connecting $uc_{E_i}$ and $vc_{E_i}$ that does not go through $\phi(b)$ and has length at most $\l_\phi l(c) + 3\l_\phi$. 

    Let $M_A = 4\l_\phi$. Since $\angle_b(u,v)$ is an edge path and has length at least $1$, we have:
    \[ 
        \angle_{\phi(b)}(\phi(u),\phi(v)) \leq \l_\phi (\angle_b(u,v)) + 3\l_\phi \leq M_A 
        \angle_b(u,v).
    \]    
\end{proof}

\begin{corollary}
\label{cor:hat-space-path-angle-bound}
    Let $c$ be a path in $\Bcal$. Suppose for some $\a\geq 1$, $c$ does not pass through any apex with angle larger than or equal to $\a$. Then $\phi(c)$ is a path in $\Acal$ that does not pass through any apex with angle larger than or equal to $M_A\a$.
\end{corollary}
\begin{proof}
    By \cref{lem:qc-induce-map-is-qi}, if $\phi(c)$ passes through an apex, then that apex is an image of an apex in $c$. 
    By \cref{lem:AngleUpperBoundForHatSpace} if $c$ passes through an apex $b$ with angle smaller than $\a$, then $\phi(c)$ passes through $\phi(b)$ with angle smaller than $M_A \a$ by \cref{lem:AngleUpperBoundForHatSpace}. 
\end{proof}

We observe that $\angle_a(u,v)$ also controls $\angle_{\phi(a)}(\phi(u),\phi(v))$ from below, which may be of independent interest. 

\begin{lemma}
\label{lem:AngleLowerBoundForHatSpace}
    Fix an apex $b\in \Bcal$. For each $M>0$ there exists $\bar{M}=\bar{M}(M)>0$ such that for any $u,v \in \Bcal$ adjacent to $b$, if $\angle_b(u,v)>\bar{M}$, then $\angle_{\phi(b)}(\phi (u), \phi(v)) > M$. 
\end{lemma}
\begin{proof} 
    Indeed, if not, then there exists $M>0$ such that, for each $\bar M>0$, there exists a pair of points $u,v\in T_b{\Bcal}$ with $\angle_b(u,v)>\bar M$ and  $\angle_{\phi(b)}(\phi (u), \phi(v))<M$. However, by fineness at cone vertices of $\Acal$, there are only finitely many orbits of pairs of vertices in $T_{\phi(b)}\Acal$ with an angle smaller than $M$. Therefore $\phi$ cannot be injective on vertices, contradicting \cref{lem:qc-induce-map-is-qi}.
\end{proof}

\subsection{Interplay between Different Graphs}
\label{subsec:interplay}

In this subsection we build maps between $\Acal$ and $\A_r$. We define a notion of angle on $\A_r$ in analogue to \cref{def:angle}. We then  investigate how those maps restrict angle metrics on both space. 

\subsubsection{Back and Forth}
We now build a map between the coned-off Cayley graph and the coned-off cusped Cayley graph of the same group with respect to the same generators. 

Recall that we can describe a vertex in a horoball built on $h\G_{O_i}$ as $(g, h, O_i, d)$, where $g$ is the corresponding group element in $A$ and $d$ is the depth. When the horoball is fixed, we can omit $h, O_i$ and refer to a vertex using only $(g, d)$.

\begin{convention}
    An edge in $\A_r$ or $\Acal$ contains both of its end points.  
    Given a partial edge, its closure is the edge that contains it. An \emph{edge path} in $\A_r$ or $\Acal$ is a path that consists of edges. 
    A path is the concatenation of one (possibly trivial) partial edge, one edge path and another (possibly trivial) partial edge.
    We allow the existence of backtrack in a path.
    The \emph{closure} of a path $c$ is the shortest edge path that contains $c$. We can obtain the closure of a path by replacing any partial edge with its closure.
    
    Recall that we assume geodesics in a coned-off cusped Cayley graph to be regular (as in \cref{lem:shape-geodesic-cusped-horoball}).
\end{convention}

\begin{definition}
    \label{def:shallow-path}
    We say $c$ is a \emph{$d$-shallow} path if $c$ penetrates any horoball to depth at most $d$. When $d$ is clear, we may drop $d$ and simply call $c$ shallow.
\end{definition}

In \cref{construction:pi_Kr_to_hatK}, we construct a map from $\A_r$ to $\Acal$ for every $r\in\N$. In \cref{cor:pi_shallow_geod_is_good}, we prove  that this map sends shallow geodesics in $\A_r$ to quasi-geodesics in $\Acal$ with constants independent of $r$.


\begin{construction}
    \label{construction:pi_Kr_to_hatK}
    For any $r\in \N $, and for any integer $q \in [1, r-1]$, we construct a map $\pi_{q}^r: \A_r  \to  \Acal$ in the following way: 
    \begin{itemize}
        
        \item Any vertex or edge in $\G_A\subseteq \A_r$ is sent to the corresponding vertex or edge in $\Acal$.
        \item The apex of a coned-off horoball $\Hcal_r(hO_i)$, where $h\in A$ and $O_i \in \Ocal$, is sent to the apex $hO_i$ in $\Acal$.
        \item Any vertex or edge at depth greater than or equal to $ q$ in $\Hcal_r(hO_i)$ is sent to the apex of the corresponding horoball.
        \item Any vertical path starting at some $g\in A$ and ending at depth $q$ in $\Hcal_r(hO_i)$ is sent to the cone edge of $hO_i$ based at $g$. 
        For each $1\leq d < q$, the vertex $(g,d)$ is sent to the point on the cone edge with distance $d/q$ from $g\in \Acal$. 
        \item  The interior of any horizontal edge with depth $0< d< q$  is mapped to the image of its starting point.
    \end{itemize}
\end{construction}

    Notice that under this map, the image of a continuous path $ c \subseteq \Acal$ might not be a continuous path. This is inconvenient when we compare lengths of images. To fix this problem we introduce the following definition.

\begin{definition}
    \label{def:push-forward_under_pi}
    Let $c$ be a continuous path in $\A_r$ that starts at $x$ and ends at $y$.  Let $u$ (resp. $v$) be the first (resp. last) vertex on $c$ that has depth $0$. In the case where $c$ contains no vertex at depth $0$, that is, $c$ is contained in a single horoball, let $u= v=y$. 
    
    Let $c_x$ be the subsegment of $c$ that starts at $x$ and ends at $u$, then $c_x$ is contained in a single horoball, say $\Hcal_r(h_xO_x)$. Similarly, let $c_y$ be the subsegment of $c$ that starts at $v$ and ends at $y$. Then $c_y$ is also contained in a single horoball, say $\Hcal_r(h_yO_y)$. Let $c_0$ be the subsegment of $c$ that starts at $u$ and ends at $v$. Notice that one or two of $c_x, c_0, c_y$ could be trivial. In particular, when $c$ is contained in a single horoball, $c_x = c$ and $c_0,c_y$ are both trivial.

    We construct $\hat c_x$ as follows:
    \begin{itemize}
        \item If $x=u$, then $c_x$ is just the vertex $x$. Let $\hat c_x$ be $\pi_q^r(x)$.
        \item If  $x\neq u$, and $\pi_q^r(u)$ lies in the same cone edge as $\pi_q^r(x)$, then $\hat c_x$ is the (partial) cone edge $\{\pi_q^r(x),\pi_q^r(u)\}$. 
        \item If $\pi_q^r(x)$ and $\pi_q^r(u)$ lie in different cone edges, then $\hat c_x$ consists of the (partial) cone edges $\{ \pi_q^r(x),a_x \}$ and $\{ a_x, \pi_q^r(u) \}$, where $a_x$ is the apex of the horoball.        
    \end{itemize}

    We construct $\hat c_y$ in the same way.  We then build $\hat c_0$ based on $\pi_q^r(c_0)$ as follows:
    \begin{itemize}
        \item  Any partial cone edge in $\pi_q^r(c_0)$ that connects to $\G_A$ is replaced by its closure.
        \item Any partial cone edge in $\pi_q^r(c_0)$ that does not connect to $\G_A$ is dropped.
        \item Any isolated vertex is dropped.
    \end{itemize}

    We call the concatenation of $\hat c_x$, $\hat c_0$ and $\hat c_y$ as the  \emph{push-forward} of $c$ under $\pi_q^r$, and denote it as $\hat c$.
\end{definition}

\begin{lemma}
    \label{lem:push-forward-continuous}
    For any continuous path $c\subseteq \A_r$ with endpoints $x,y$, its push-forward $\hat c$ under $\pi_q^r$ is a continuous path connecting $\pi_q^r(x)$ and $\pi_q^r(y)$. Moreover,  \[l(\hat c)\leq 2 l(c).\]

    When $c$ does not contain partial horizontal edges, we have \[l(\hat c)\leq  l(c).\]
\end{lemma}

\begin{proof}
    Let $c_x,c_0,c_y$ be as in \cref{def:push-forward_under_pi}. Since $u,v$ are contained in $\G_A$, $\hat c_0$ contains $\pi_q^r(u)$ and $\pi_q^r(v)$. It suffices to show  $\hat c_0$ is a continuous path.

    A maximal subsegment of $c_0$ contained in $\G_A$ is called a \emph{$\G_A$ component} of $c_0$. Similarly, a maximal subsegment of $c_0$ that contains in a single horoball and contains no edges in $\G_A$ is called a \emph{horoball component} of $c_0$. Then $c_0$ can be divided into $\G_A$ components and horoball components, and $\hat c_0$ is the union of push-forwards of components.
    
    Images of $\G_A$ components of $c_0$ are continuous edge paths. We do not change anything in $\G_A$ in the construction of $\hat c_0$, so $\hat c_0$ contains all the images of $\G_A$ components of $c_0$.

    If there is more than one $\G_A$ component in $c_0$, then two consecutive  $\G_A$ components are joined by one horoball component, say $\t$. Any horoball component starts at a vertical edge and ends at a vertical edge. Therefore, $\pi_q^r(\t)$ contains at least two (partial) cone edges that are connected to $\G_A$. When constructing $\hat c_0$, we replace these two cone edges by their closure, and delete everything else in $\pi_q^r(\t)$. 
    Therefore $\hat c_0$ is a continuous path that starts at $\pi_q^r(u)$ and ends at $\pi_q^r(v)$.

    Consider now the length of $\hat c$. In $c_0$, each horoball component has length at least $3$, whereas the corresponding component in $\hat c_0$ has length $2$. So $l(\hat c_0)\leq l(c_0)$
    
    If $\hat c_x$ consists of one (partial) cone edge, then $l(\hat  c_x) \leq l(c_x)/q \leq l(c_x)$. If $\hat c_x$ consists of two (partial) cone edges, then $c_x$  contains at least one (partial) horizontal edge and one vertical edge. Thus $l(\hat c_x)\leq 2 l(c_x)$. If there's no partial horizontal edge in $c$, then $l(\hat c_x) \leq l(c_x)$. The same analysis applies to $\hat c_y$. 
\end{proof}


Similarly, we construct a map $\iota_q^r:\Acal \to \A_r$. 

\begin{construction}
    \label{construction:iota_hatK_to_Kr}
    For any $r\in \N $, and for any integer $q \in [1, r-1]$, we construct a map $\iota_q^r : \Acal \to \A_r$ in the following way:
\begin{itemize}
    \item Any vertex or edge in $\G_A \subseteq \Acal$ is sent to the corresponding vertex or edge in  $\G_A\subseteq  \A_r$, 
    \item The interior of any cone edge of $hO_i$ based at $g$ is sent to the interior of the vertical geodesic connecting $(g,h,O_i,0)$ and  $(g,h,O_i, q)$ proportionally. That is, a point with distance $d$ from $g$ is sent to the point $(g,h,O_i,dq)$. And
    \item The apex $hO_i$ is sent to the apex of $\Hcal_r(hO_i)$.
\end{itemize}
    
\end{construction}

\begin{definition}
    \label{def:pullback_under_pi}
    Let $c = [x,y]$ be a path in $\Acal$. Observe that the interior of each cone edge is mapped to the interior of a vertical path. By taking the limit, we can recover the endpoints of such vertical paths.

    Given $\iota_q^r(c)$, we construct the \emph{pullback} $\widetilde c$ in $\A_r$ of $c$ as follows: 
    \begin{itemize}
        \item For any (partial) cone edge that contains an apex in $c$, recover the endpoint at depth $q$ of its image as described above.
        \item Drop any isolated apices in $\iota_q^r( c)$.
        \item For each pair of consecutive components, join them with a geodesic in the corresponding coned horoball in $\A_r$.
    \end{itemize}

    Clearly the pullback is a continuous path in $\A_r$.
\end{definition}

\begin{lemma}
\label{lem:length-path-pullback}
    For any $r\in \N $, for any integer $q \in [1, r-1]$, and for any edge path $c\subseteq \Acal$, the following inequality holds:   
    \[l(c)  \leq l(\tilde c) \leq (r+1) l(c) .\] 
\end{lemma}
\begin{proof}
    Consider the construction of $\tilde c$. Each edge in $c$ is either kept the same,  or mapped to a path of length $q$ where $q\geq 1$. Then we add additional geodesics to $\iota_q^r(c)$ and dropped only vertices. Therefore $l(c) \leq l(\tilde c)$.

    On the other hand, each attached geodesic penetrates the horoball to depth at most $r+1$. Therefore each pair of consecutive cone edges in $c$ corresponds to a path of length at most $2(r+1)$. If $c$ is started or ended with a cone edge, that cone edge corresponds to a path of length $q$. Therefore the second inequality follows.
\end{proof}

\begin{lemma}
    \label{lem:pi_is_r+1_qi}
    For any $r\in \N $, and for any integer $q \in [1, r-1]$, the map $\pi_q^r: \A_r \to \Acal$ given by \cref{construction:pi_Kr_to_hatK} is a $(2(r+1),0)$--quasi-isometry.
    
\end{lemma}
\begin{proof}
    Let $\g = [x,y]$ be a geodesic in $\A_r$. By \cref{lem:push-forward-continuous} we have:
    \[
    d_{\Acal}( \pi_q^r(x), \pi_q^r(y) ) \leq l(\hat \g) \leq l(\g) = d_{\A_r}(x,y) .
    \]

    On the other hand, let $\eta = [\pi_q^r(x), \pi_q^r(y)]$ be a geodesic in $\Acal$ and let $\tilde \eta$ be its pullback in $\A_r$.  By \cref{lem:length-path-pullback} we have:
    \[
     d_{\A_r}(x,y) \leq l(\tilde \eta) \leq  2(r+1) d_{\Acal}( \pi_q^r(x), \pi_q^r(y) ) .
    \]    

    Furthermore, for every $x\in \Acal$, if $x\in\G_A$, then $x\in \pi_q^r(\A_r)$. If $x\notin \G_A$, then $x$ is an apex and therefore is the image of an apex in $\A_r$. Therefore $\Acal \subseteq N_1(\pi_q^r(\A_r))$ and the conclusion follows.
\end{proof}

\begin{lemma}
    \label{lem:iota_is_r+1_qi}
    For any $r\in \N $, and for any integer $q \in [1, r-1]$,  the map $\iota_q^r: \Acal \to \A_r$ given by \cref{construction:iota_hatK_to_Kr} is an $(r+1,0)$--quasi-isometry.
\end{lemma}
\begin{proof}
    Let $\g = [x,y]$ be a geodesic in $\Acal$. Let $\tilde\g$ be its pullback in $\A_r$. By \cref{lem:length-path-pullback}, we have
    \[
    d_{\A_r}( \iota_q^r(x), \iota_q^r(y) ) \leq l(\tilde \g) \leq  (r+1) d_{\Acal}(x,y) .
    \]

    On the other hand, let $\eta = [\iota_q^r(x), \iota_q^r(y)]$ be a geodesic in $\A_r$. Notice that $\eta$ contains no partial horizontal edge. Let $\hat \eta$ be its push-forward, then by \cref{lem:push-forward-continuous} we have:
    \[
     d_{\Acal}(x,y) \leq l (\hat \eta) \leq  d_{\A_r}( \iota_q^r(x), \iota_q^r(y) ) .
    \]
    
    Furthermore, let $x\in \A_r$ be a vertex. If $x\in \G_A$, then $x\in \iota_q^r(\Acal)$. If $x$ lies in a horoball, then it's at most $r+2$ away from $\G_A$. Therefore $\A_r \subseteq N_{r+2}( \iota_q^r(\Acal))$. Since $r+2 < 2(r+1)$, the conclusion follows.
\end{proof}

\begin{lemma}
\label{lem:composition_pi_iota}
    The composition $\iota_q^r \circ \pi_q^r$ coincides with the identity map of $\A_r$ on vertices at depth less than $q$.

    On the other hand, $\pi_q^r \circ \iota_q^r$ coincides with the identity map on $\Acal$.
\end{lemma}
\begin{proof}
    Both statements follow from the constructions.
\end{proof}

Let $x,y\in \A_r$ and let $c\in \Acal$ be a path connecting $\pi_q^r (x), \pi_q^r (y)$. In general, the endpoints of $\tilde c \in \A_r$ doesn't have to be $x,y$. We identify below a special case where the endpoints of $\tilde c$ are $x,y$.

\begin{lemma}
    \label{lem:pullback-shallow-ends}
    Let $x,y$ be two points at depth  less than $q$ in $\A_r$. Let $c\in \Acal$ be a path connecting $\pi_q^r (x), \pi_q^r (y)$. Then $\tilde c$ is a continuous path connecting $x,y$.
\end{lemma}
\begin{proof}
     Since the pullback drops only images of apices,  the endpoints of $\eta$ are $\iota_q^r\circ \pi_q^r(x), \iota_q^r \circ \pi_q^r(y)$.      By \cref{lem:composition_pi_iota}, $\iota_q^r\circ \pi_q^r(x ) = x$, $\iota_q^r\circ \pi_q^r(y ) = y$.
\end{proof}

\subsubsection{Angle in Coned-off Cusped Cayley Graph}
We define the notation of angle in $\A_r$ in a way similar to the angle in $\Acal$.

\begin{definition}[Angle in Coned-off Cusped Space]
\label{def:angle_in_Kr}

Let $\Hcal_r(hO_i)$ be the coned-off horoball built on $h\G_{O_i}\subseteq \G_A$. Let $a$ be the apex and ${ u}, { v}\in h\G_{O_i}$, we define $\angle^{\Hcal}_{a}( u, v)$ to be the length of the shortest path connecting $u,v$ that penetrates $\Hcal_r(hO_i)$ trivially. If no such path exists then we say $\angle_a^{\Hcal}(u,v) = \infty$. 

Let $x,y\in A$. Let $c_1 = [a,x]$,  $c_2=[a,y]$ be paths in $\A_r$. Let $ u \in c_1\cap \Hcal_r(hO_i)$ be the first such vertex on $c_1$ after $x$ and let  $v \in c_2 \cap \Hcal_r(hO_i)$ be the last such vertex on $c_2$ before $y$. We define $\angle^{\Hcal}_a(c_1,c_2) := \angle^{\Hcal}_a(u,v)$. 
Notice that if $[u,v]$ penetrates the horoball to depth $d\leq r$, then $d_{h\G_{O_i}}(u,v)\leq 2^{d+1}$.
    
\end{definition}

\begin{lemma}
    \label{lem:Shallow_to_SmallAngle}
        For all $r\in \N \cup \{0\}$, $\A_r$ satisfies the following property. Let $a$ be the apex of $\Hcal_r(hO_i)$ and let $\g $ be a geodesic connecting $u,v \in hO_i$.  For all integer $q\in [1,r-1]$, if $\g$ penetrates $\Hcal_r(hO_i)$ to depth at most $q$, then $\angle_{\pi_q^r (a)} (\pi_q^r(u),\pi_q^r( v)) \leq 2^{q+1}$.
\end{lemma}
\begin{proof}
    Let  $c$ be a path that penetrates $\Hcal_r(hO_i)$ trivially and whose length realizes $\angle_a^\Hcal(u,v)$. Since $c$ contains no vertical edges in $\Hcal_r(hO_i)$ nor the apex $a$, its push-forward $\hat c$ doesn't contain $\pi_q^r(a)$. Then by \cref{lem:push-forward-continuous} we have
    
    \[
        \angle_{\pi_q^r (a)} (\pi_q^r(u),\pi_q^r( v)) \leq l(\hat c) \leq l(c) =  \angle_{ a}^\Hcal (  u,  v ) .
    \]  

    On the other hand, since $\g$ penetrates $\Hcal_r(hO_i)$ to depth at most $q$, $d_{h\G{O_i}}(u,v)\leq 2^{q+1}$. It follows immediately that
    \[
    \angle_{\pi_q^r (a)} (\pi_q^r(u),\pi_q^r( v))  \leq  \angle_a^\Hcal (u,v) \leq 2^{q+1}.
    \]
\end{proof}

\begin{lemma}
   \label{lem:DeepPen_to_BigAngle}
   For all $D>0$, there exists $r_D=r_D(D)\geq D$ such that, for all $r>r_D$, for all $q\in [r_D, r)$, and for any $u,v\in hO_i \subset \A_r$, if $[u,v]$ penetrates $\Hcal_r (hO_i)$ to depth at least $r_D$, then $\angle_{\pi_q^r (a)}(\pi_q^r (u), \pi_q^r (v)) > D$.
   
\end{lemma}
\begin{proof}
    For simplicity of notation, let $\hat a = \pi_q^r(a), \hat u = \pi_q^r (u), \hat v = \pi_q^r(v)$. 

    Suppose first $\angle_a^\Hcal (u,v) < \infty$, then there exists a path $c$ connecting $u,v$ that penetrates $\Hcal_r (hO_i)$ trivially. The push-forward of $c$ is an edge path in $\Acal$ that connects $\hat u, \hat v$ and does not go through $\hat a$, hence $\angle_{\hat a}(\hat u, \hat v)< \infty.$  
    Notice that there are finitely many orbits of pairs of points $u,v$ such that $\angle_{\hat a}(\hat u, \hat v) \leq D$. Consider now the geodesics connecting those pairs of $u,v$ in $\A_r$. There exists an upper bound $r_D = r_D(D)$ of their maximal depths. Notice that this bound does not depend on $r$ so long as $r>r_D$.  

    Suppose now $\angle_a^\Hcal(u,v) = \infty$. In this case, the geodesic $[u,v]$ goes through the apex $a$. Hence $[u,v]$ has penetration depth $r+1 >r_D$. 
    If $\angle_{\hat{a}} (\hat u, \hat v)<\infty$, then there exists a path $\eta$ connecting $\hat u, \hat v$ that does not contain $\hat a$. In particular, it does not contain any cone edge of $hO_i$.  
    Hence the pullback of $\eta$ is a path connecting $u,v$ that penetrates $\Hcal_r(hO_i)$ trivially, contradicting the assumption that $\angle_a^\Hcal(u,v) = \infty$. 
    Therefore, $\angle_{\hat{a}} (\hat u, \hat v) = \infty >D$.
       
    Without loss of generality we may assume $r_D>D$. 
\end{proof}

\begin{corollary}
    \label{cor:small-angle-shallow-pullback}
    Let $c$ be any path in $\Acal$. For all $D>0$, let $r_D=r_D(D)$ be the constant from \cref{lem:DeepPen_to_BigAngle}. Then for all $r>r_D$, and for all $q\in [r_D, r)$, if $c$ does not pass any apex with angle greater than $D$, then $\tilde c$ is a $q$-shallow path. Moreover, $l(\tilde c) \leq (q+1) l(c) $. 
\end{corollary}
\begin{proof}
    Suppose $c$ passes some apex $a$ and let $u,v\in c$ be the vertices adjacent to $a$.  If the geodesic $[\iota_q^r(u), \iota_q^r(v)]$ has penetration depth greater than or equal to $r_D$, then $\angle_a(u,v)\geq D$, contradicting the assumption. Therefore, there exists a horizontal edge connecting $(u,q)$ and $(v,q)$ in the horoball. 

    It follows that $l(\tilde c)$ is $q$-shallow and has length at most $(q+1)l( c)$.
\end{proof}

A straightforward combination of \cref{lem:GeodpassBigAngle,lem:DeepPen_to_BigAngle} gives the following corollary.

\begin{corollary}
\label{cor:Deep-to-Big-Angle-quasigeodesic}
    For all $k,l \geq 1$, let $D=D(\d,k,l)$ be as in \cref{lem:GeodpassBigAngle} and let $r_D = r_D (D)$ be as in \cref{lem:DeepPen_to_BigAngle}. 
    Let $hO_i$ be a coset of hyperbolically embedded subgroup in $\Acal$. 
    Let $c_1$ (resp. $c_2$) be a $(k,l)$-quasi-geodesic starting at $a$, the apex of $hO_i$ and ending at $x$ (resp. $y$). 
    Let $u$ (resp. $v$) be the vertex on $c_1$ (resp. $c_2$) adjacent to $a$. 
    For all  $r>r_D$ and all $q\in [r_D, r)$, if $[\iota_q^r(u), \iota_q^r(v)]$ penetrates $\Hcal_r(hO_i)$ to depth at least $r_D$, then any geodesic $[x,y]$ must pass through $a$.
\end{corollary}

\subsubsection{Extended pullback}

In this section we construct a class of paths in $\mathbb{A}_r$ that are close to the geodesics. This class of paths plays an important role in the proof of \cref{thm:uniform_qc}.

Recall that we fix $\delta>0$ such that for all $r>0$, $\A_r$ and $\Acal$ are $\delta$-hyperbolic in the sense that all triangles are $\delta$-thin hence $\delta$-slim (See \cref{sec:basic-constructions}). Recall also we assume geodesics are regular (See).

\begin{construction}
    \label{cons:extended-pullback}
    For any $r\in\N \cup \{0\}$, any integer $q\in [1,r-1]$, and any pair of vertices $x,y \in \A_r$, we construct a path $\eta$ in $\A_r$ as follows:
        \begin{enumerate}
        \item If both $x,y$ are at depth  less than $q$, then $\eta$ is the pullback of some geodesic $[\pi_q^r (x), \pi_q^r (y)]$ as defined in \cref{def:pullback_under_pi}. By \cref{lem:pullback-shallow-ends} $\eta$ connects $x,y$.
        
        \item If both $x,y$ are at depth greater than or equal to $ q$ and are contained in the same horoball, then $\eta $ is some geodesic connecting $x,y$.
        
        \item If both $x,y$ are at depth greater than or equal to $ q$ but are contained in different horoballs, then let $[x,y]$ be a geodesic. Let $u\in[x,y]$ (resp. $v\in[x,y]$) be the last (resp. first) vertex that has depth $q$ and is in the same horoball of $x$ (resp. $y$). Let $\eta_0 = [\pi_q^r (u), \pi_q^r (v)]$ be a geodesic in $\Acal$, and let $\tilde{\eta}_0$ be its pullback. Let $u',v'$ be the endpoints of $\tilde \eta_0$, then $u',v'$ are at depth $q$ in the same horoball of $u,v$ respectively. Let $[x,u]$ and $[v,y]$ be subsegments of $[x,y]$, and $[u,u'],[v',v]$ be geodesics. Let $\eta$ be the concatenation of $[x,u]$, $[u,u']$, $\tilde \eta_0$, $[v',v]$ and $[v,y]$.

        \item If only one of $x,y$ is at depth greater than or equal to $ q$,  without loss of generality let  $x$ be the one deep in the horoball. Let $[x,y] $ be a geodesic. Let $u\in[x,y]$ be the last vertex at depth $q$ in the same horoball of $x$. Let $\eta_0$ be a geodesic $[\pi_q^r (u), \pi_q^r (y)]$, then $\tilde \eta_0$ connects $u'=\iota_q^r\circ\pi_q^r(u)$ and $y$. Let $\eta$ be the concatenation of the subsegment $[x,u]\subseteq [x,y]$, a geodesic $[u,u']$ and $\tilde \eta_0$.
    \end{enumerate}

    Such a path is called an \emph{extended pullback between $x,y$ at depth $q$}.  Since neither $\mathbb{A}_r$ nor $\Acal$ needs to be uniquely geodesic, the construction of $\eta$ may not be unique. We denote the class of all extended pullback between $x,y$ at depth $q$ as $\eta_q(x,y)$.
\end{construction}

In general a subsegment of an extended pullback might not be an extended pullback. We now prove that this is true in a special case.

\begin{lemma}
    \label{pullback-is-extended-pullback}
    Let $p$ be a path in $\Acal$. Then there exists a pair of points $x, y \in\A_r$ and $\eta \in \eta_q(x,y) $ such that $\pi_q^r( \eta) = p$. 
    
    Moreover, if $p'$ shares endpoints with $p$, then there exists $ \eta' \in \eta_q(x,y)$ such that $\pi_q^r(\eta)=p, \pi_q^r(\eta') = p'$.
\end{lemma}
\begin{proof}
    Let $u,v$ be the endpoints of $p$. Let $\eta = \tilde p$ be the pullback of $p$ and let $x,y$ be the endpoints of $\eta$. It's clear that $\eta\in\eta_q(x,y)$. By \cref{lem:composition_pi_iota} $p = \pi_q^r(\eta)$.

   Let $\eta'_0$ be the pullback of $p'$ and let $x', y'$ be the endpoints of $\eta'$.  If neither of $u,v$ is apex, then $x' = x, y'=y$. Let $\eta'=\eta'_0$ and we are done. Suppose now $u$ is an apex, then  $x'$ is in the same horoball as $x$, and they are both at depth $q$.  Let $\eta ' $ be the concatenation of $[x,x'],\eta'_0,[y',y]$. Then $\eta' \in \eta(x,y)$.  Notice that any point in $[x,x']$ or $[y,y']$  has depth greater than or equal to $q$. Therefore $\pi_q^r([x,x']) = u$, $\pi_q^r(y,y') = v$. It follows that $\pi_q^r(\eta ' ) = p'$.
\end{proof}

We now prove that extended pullbacks of the same endpoints are actually close to each other. As a result, we may treat $\eta_q(x,y)$ as one extended pullback between $x,y$ at depth $q$. 

\begin{lemma}
    \label{lem:extended-pullback-are-close-bigon}
    Let $x,y\in \A_r$. Let $D=D(\d,1,0)$ be as in \cref{lem:GeodpassBigAngle} and let $r_D = r_D(D)$ be as in \cref{lem:DeepPen_to_BigAngle}. Then for all integer  $r>r_D$ and all integer $q\in [r_D, r-1]$,  the Hausdorff distance in $\Acal$ between $\eta_1,\eta_2 \in \eta_q(x,y)$ is bounded above by $2q\delta$.
\end{lemma}
\begin{proof}
       We discuss the four cases of extended pullback one by one. For simplicity, we denote the concatenation of $n-1$ geodesics $[x_1, x_2], [x_2,x_3],\dots,[x_{n-1}, x_n]$ as $[x_1,x_2,x_3,\dots,x_n]$.
       
    \begin{case}
        Both $x,y$ are at depth less than $q$. 
    \end{case}        
        In this case, $\eta_1$ (resp. $ \eta_2$) is the pullback of geodesics $\hat\g_1$ (resp. $\hat\g_2$)  connecting $\hat x = \pi_q^r(x), \hat y =  \pi_q^r(y) \in \Acal$. Without loss of generality suppose $\hat x, \hat y $ are different. 

        \begin{claim*}
        \label{claim:intersect-at-small-angle}
            Suppose both $\hat x, \hat y$ are vertices. Suppose also $\hat \g_1,\hat \g_2$ has no intersection other than $\hat x, \hat y$. Then 
            \[ 
                \angle_{\hat x} (\hat \g_1,\hat \g_2) <8\d <D, \qquad  \angle_{\hat y} (\hat \g_1,\hat \g_2) <8\d <D.
            \]
        \end{claim*}

        \begin{proof}
                
        By \cref{lem:GeodpassBigAngle}, $\hat \g_1$ and $\hat \g_2$ do not pass any apex with angle greater than or equal to $ D$, otherwise they must intersect at that apex. 
     
        Notice that $l(\hat \g_1)\geq 2$, otherwise $\g_1$ and $\g_2$ must be the same edge. Then there are two different points $\hat u_1 \in \hat \g_1, \hat u_2\in \hat \g_2$ that are adjacent to  $\hat x$.  If $l(\hat \g_1) \leq 3\delta$, then $[\hat u_1,\hat y,\hat u_2]$ is a path of length at most $6\delta$ that avoids $\hat x$ because both segments are geodesics. 

        If  $l(\hat \g_1)>3\delta$, then take $\hat v_1 \in \hat \g_1$ that is $3\delta$ away from $\hat x$. By hyperbolicity of $\hat L$, there exists $\hat v_2 \in \hat \g_2$ such that $d_{\hat L}(\hat v_1,\hat v_2)\leq \delta$. Moreover, $[\hat v_1,\hat v_2]$ cannot contain $\hat x$, otherwise $d_\Acal(\hat x, \hat v_1)\leq \d < 3\d$, contradicting the fact that  $\hat \g_1$ is a geodesic. It follows that $[\hat u_1,\hat v_1,\hat v_2,\hat u_2]$  is a path of length at most $8\delta-2$ that does not contain $\hat x$. See \cref{fig:small-angle}. Either way, $\angle_{\hat x} (\hat \g_1,\hat \g_2)< 8\delta \leq D$. The same argument goes for $\angle_{\hat x} (\hat \g_1,\hat \g_2)$.
  \end{proof}
    
        \begin{figure}[htb]
            \centering
            \includegraphics[width=0.5\linewidth]{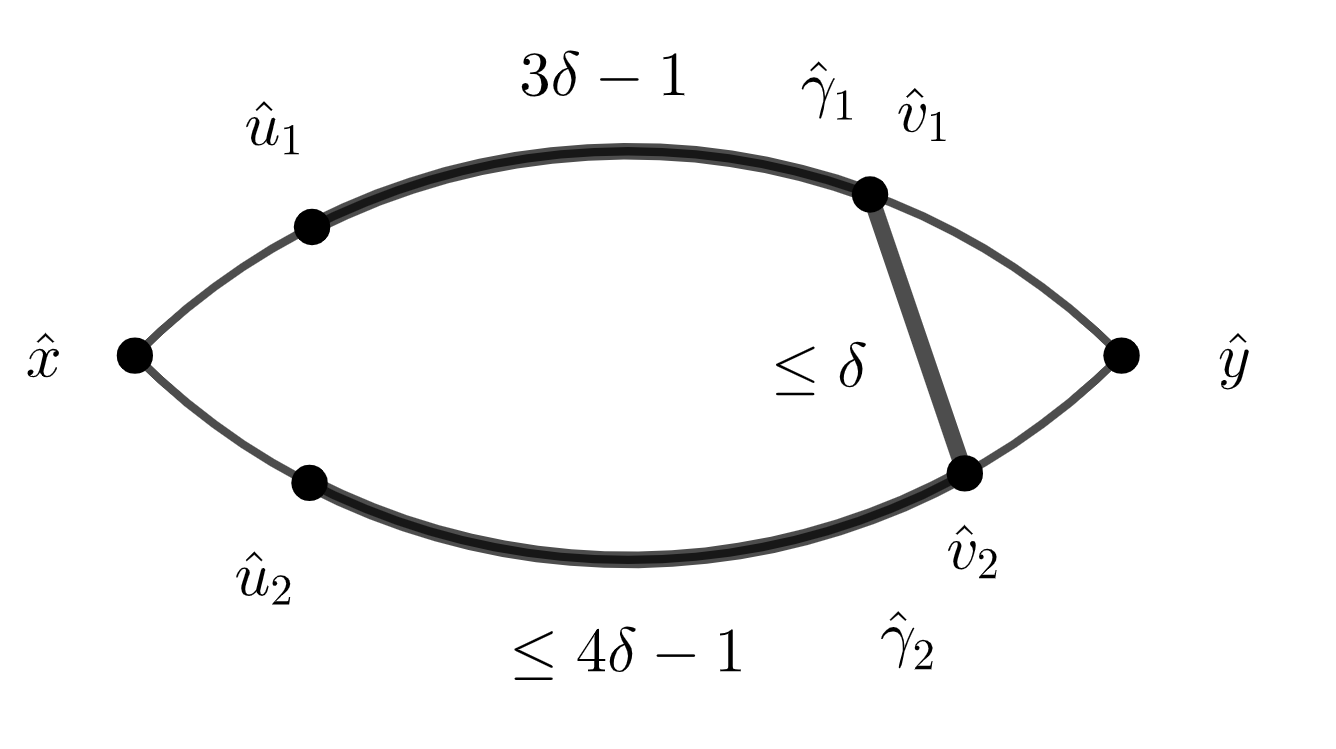}
            \caption{When $l(\hat\g_1)\geq 3\delta$, $\angle_{\hat x}(\hat \g_1, \hat \g_2) < 8\delta$}
            \label{fig:small-angle}
        \end{figure}

    Notice that $\hat \g_1, \hat \g_2$ might have intersection other than $\hat x, \hat y$. If they intersect at an apex $\hat a$, then the pullback of $\hat a$ might be different vertices in $\eta_1, \eta_2$. However, the claim above guarantees that the two pullbacks of $\hat a$ are at distance $1$ from each other.     We can therefore assume that $\hat \g_1, \hat \g_2$ have no intersection other than $\hat x, \hat y$. 
    
    Consider now any point $p\in \eta_1 \backslash \eta_2$. Let $\hat p = \pi_q^r(p)$ and let $\hat q\in \hat\g_2$ be the closest vertex to $\hat p$. Since $\hat\g_1$ and $\hat\g_2$ are geodesics sharing both endpoints, $d_\Acal(\hat p,\hat q)\leq \delta$. 

    \begin{claim*}
        \label{claim:extended-pullback-shallow}
        Under this condition,  $[\hat p, \hat q]$ does not pass through any apex with angle greater than or equal to $ D$.
    \end{claim*}
    \begin{proof}
        
    In order to obtain a contradiction, suppose $[\hat p, \hat q]$ passes some apex $\hat a$ with angle greater than or equal to $D$. Let $\hat a_1,\hat a_2$ be the two adjacent vertices of $\hat a$ on $\hat c$.

     First suppose $\hat p$ is at least $3\delta $ away from $\hat x$ and $\hat y$. Then we can find a point $\hat p_1\in \hat \g_1$ (resp. $\hat p_2\in \hat \g_1$) which is $3\delta$ away from $\hat p$ on $[\hat x, \hat p]$ (resp. $[\hat p, \hat y]$). Let $\hat q_1$ (resp. $\hat q_2$) be the closest point of $\hat p_1$ (resp. $\hat p_2$) on $\hat \g_2$. 
     \begin{figure}[htb]
        \centering
        \includegraphics[width=0.5\linewidth]{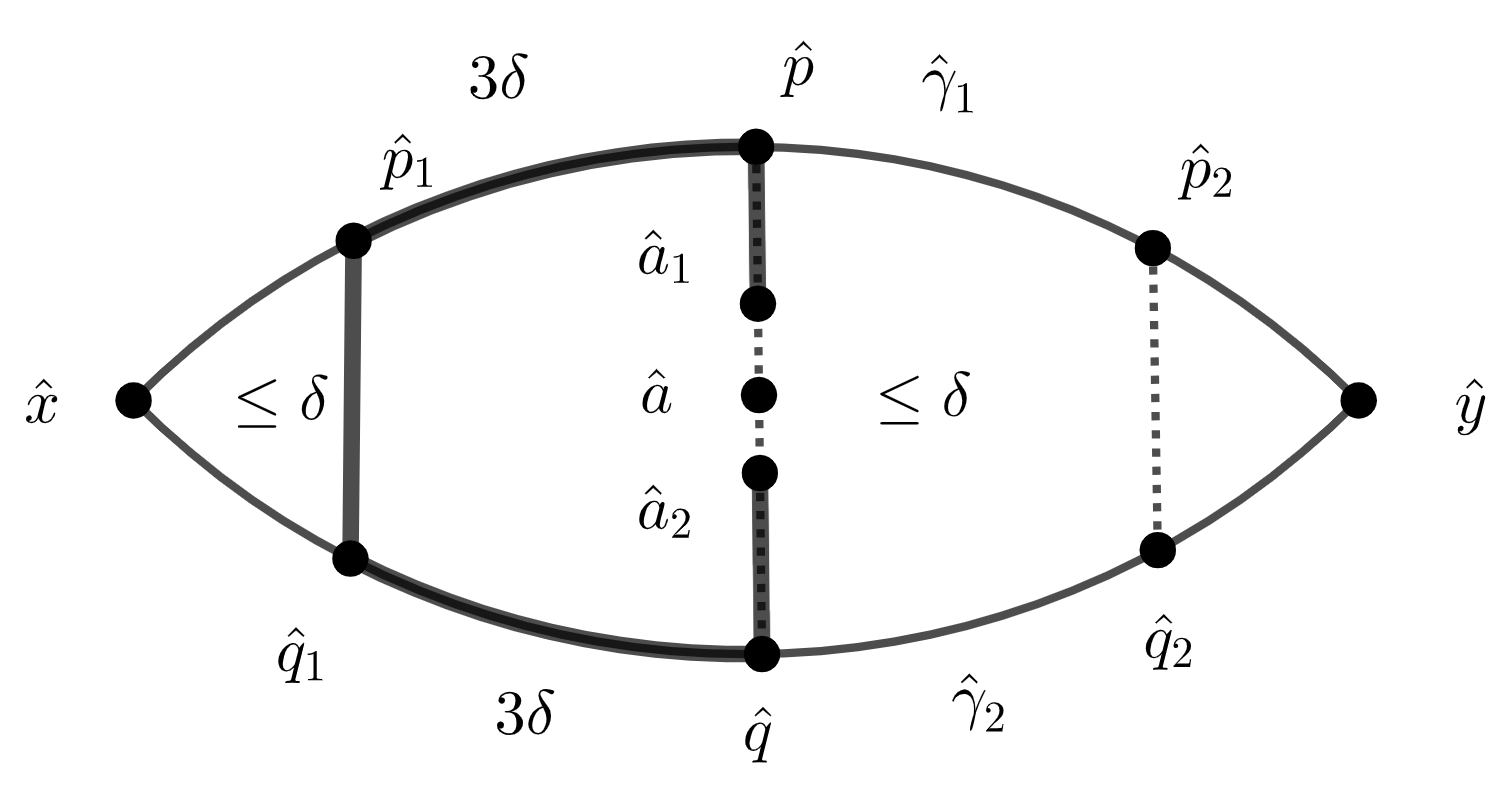}
        \caption {The path that shows $\angle_{\hat a }(\hat a_1, \hat a_2) \leq 8 \d$. }
        \label{fig:bigon-close}
    \end{figure}

    Since $[\hat p, \hat q]$ is a geodesic, $\hat a$ appears once on $[\hat p, \hat q]$. In particular, $\hat a$ is not contained in $[\hat p, \hat a_1]$ nor $[\hat a_2, \hat q]$. If $\hat a$ is contained in $\hat \g_2$, then $\hat a$ is closer to $\hat p$ than $\hat q$, contradicting the choice of $\hat q$. So $\hat a$ is not contained in $[\hat q, \hat q_1]$ nor $[\hat q, \hat q_2]$.
     
    Notice that $ [\hat p,\hat q]$ cannot intersect $ [\hat p_1,\hat q_1]$ (resp.  $[\hat p_2,\hat q_2]$), otherwise there exists a path $[\hat p, \hat p_1]$ (resp.  $[\hat p,\hat p_2]$) with length shorter than $2\d$. Therefore, $\hat a$ is not contained in $[\hat p_1,\hat q_1]$ nor $[\hat p_2, \hat q_2]$. 

    Moreover, $\hat \g_1$ is a geodesic, so $\hat a$ could appear on $\hat \g_1$ at most once. Without loss of generality suppose $\hat a \notin [\hat p, \hat x]$. But then $[\hat a_1, \hat p, \hat p_1, \hat q_1, \hat q, \hat a_2]$ is a path of length at most $10\delta <D$ excluding $\hat a$, a contradiction. See \cref{fig:bigon-close}.

    Suppose now $\hat p$ lies in the $3\delta$ neighborhood of $\hat x$. If $\hat a\notin [\hat x, \hat p]$, then  $[\hat a_1, \hat p, \hat x, \hat q, \hat a_2]$ is a path of length at most $8\delta$ excluding $\hat a$. So $\hat a \in [\hat x, \hat p]$, which implies $\hat a \notin [\hat y, \hat p] $. The same argument applies to this side. We obtain again a contradiction. 
  \end{proof}
     
    By \cref{cor:small-angle-shallow-pullback}, the pullback of $[\hat p, \hat q]$ has length at most $(q+1)\d$. Hence $p$ is in the $(q+1)\d\leq 2q\d$ neighborhood of $\eta_2$. It follows that $d_{\A_r}(p, \iota_q^r(\hat q)) \leq 2q \d$ as desired.

    If $\hat x$ or $\hat y$ is not a vertex, then it is contained in a cone edge, and $\hat \g_1, \hat \g_2$ forms either a geodesic triangle where one side is a cone edge, or a geodesic quadrilateral where two opposite side are cone edges. Take $p\in \eta_1\backslash\eta_2$ and let $\hat p = \pi_q^r(p)$. If $\hat p$ lands on the same cone edge as $\hat x$, then $ p$ is at most $q$ away from $\eta_2$. If $\hat p$ is not contained in the cone edge, then a similar argument shows that $d_{\A_r}(p,\eta_2)$ is at most $2q\d$. 
     
    It follows that in Case 1, the Hausdorff distance between $\eta_1,\eta_2$ is at most $2q\d$.

    \begin{case}
        Both $x,y$ are  at depth greater than or equal to $q$, and they are contained in the same horoball.
    \end{case}
        In this case, $\eta_1$ and $\eta_2$ are two different geodesics between $x,y$. Since $\A_r$ is $\delta$ hyperbolic, the Hausdorff distance between geodesics sharing the same endpoints is at most $\delta$. 

    \begin{case}
        Both $x,y$ are at depth greater than or equal to $ q$, and they are contained in different horoballs.
    \end{case}

        In this case, there exists two geodesics $g_1,g_2$ connecting $x,y$ with $u_1$ (resp. $u_2$, $v_1$, $v_2$) at depth $q$ in the same horoball of $x$ (resp. $x$, $y$, $y$), such that  $\eta_1(x,y)$ (resp. $\eta_2(x,y)$) is the concatenation of $[x,u_1], \eta(u_1,v_1), [v_1,y]$ (resp. $[x,u_2],\eta(u_2,v_2), [v_2,y]$). 

        Let $\hat x = \pi_q^r(x)$ and $\hat y = \pi_q^r(y)$.    Observe that $\pi_q^r(x) = \pi_q^r(u_1) = \pi_q^r(u_2)$ and $\pi_q^r(y) = \pi_q^r(v_1) = \pi_q^r(v_2)$. Therefore, there exists two different geodesic $\hat \g_1,\hat\g_2$ connecting $\hat x, \hat y$  such that $\eta(u_1,v_1)$ (resp. $\eta(u_2,v_2)$) is the pullback of $\hat \g_1$ (resp. $\hat \g_2$). The same argument as Case 1 shows that the Hausdorff distance between $\eta(u_1,v_1)$ and $\eta(u_2,v_2)$ is bounded above by $2q \d$. 
        
        Moreover, since $\angle_{\hat x}(\hat \g_1,\hat g_2)\leq D$, $d_{\A_r}(u_1, u_2) = 1$. Similarly $d_{\A_r}(v_1,v_2) = 1$. Recall that the Hausdorff distance between $\g_1,\g_2$ is bounded above by $\delta<2r_D\d$. It follows that the Hausdorff distance between $\eta_1,\eta_2$ is bounded above by $2q\d$.

    \begin{case}
        Only one of $x,y$ is at depth greater than or equal to  $ q$. 
    \end{case}
        The same argument as in Case 3 shows the Hausdorff distance between $\eta_1,\eta_2$ is bounded above by $2q\d$.
\end{proof}

We use \cite[Proposition 3.1 ``Guessing Geodesics Lemma"]{bowditchUniformHyperbolicityCurve2014} to prove that extended pullbacks are close to geodesics. 

\begin{proposition}[{\cite[Proposition 3.1 ``Guessing Geodesics Lemma"]{bowditchUniformHyperbolicityCurve2014}}]
    \label{prop:guess-geodesic}
    Given $h\geq 0$, there is some $k\geq 0$ with the following property. 
    Suppose that $G$ is a connected graph, and that for each $x,y\in V(G)$, we have associated a connected subgraph, $\ms{L}(x,y) \subseteq G$, with $x,y\in \ms{L}(x,y)$. Suppose that:
    \begin{enumerate}
        \item For all $x,y,z \in V(G)$,
        \[
            \ms{L} (x,y) \subseteq N_h(\ms L (x,z) \cup \ms L (z,y) ).
        \]
         \item For any $x,y\in V(G)$ with $d(x,y)\leq 1$, the diameter of $\ms L (x,y)$ in $G$ is at most $h$.
    \end{enumerate}
    Then $G$ is $k$-hyperbolic. In fact, we can take any $k\geq (3m-10h)/2$, where $m$ is any positive real number satisfying $2h(6+\log_2(m+2)\leq m$. Note that the condition on $m$ is monotonic: if it holds for $m$, it holds strictly for any $m'>m$.
    Moreover, for all $x,y\in V(G)$, the Hausdorff distance between $\ms L (x,y)$ and any geodesic from $x$ to $y$ is bounded above by $m-4h$.
\end{proposition}

\begin{lemma}
    \label{lem:ext-pullback_close_to_geodesic}
    Let $D = D(\d,1,0)$ be as in \cref{lem:GeodpassBigAngle} and let $r_D= r_D(D)>0$ be as in \cref{lem:DeepPen_to_BigAngle}. Then $\forall r>r_D$,  $\forall q\in [r_D, r-1]$, there exists $N = N(q,\delta)$ satisfying the following property. Let $x,y\in \A_r$. The Hausdorff distance between any extended pullback between $x,y$ at depth $q$ and any geodesic $[x,y]\in L$ is bounded above by $N$.

    In particular, $N$ does not depend on $r$, so long as $r>r_D$. 
\end{lemma}
\begin{proof}

    Without loss of generality we may assume $D>32\d$ (See \cref{rmk:GeodpassBigAngle-can-be-arb-large}). For each pair of vertices $x,y$ in $\A_r$, we associate one extended pullback as in \cref{cons:extended-pullback}. By the construction, the extended pullback is connected and contains $x,y$. 
    
    When $d_{\A_r}(x,y)\leq 1$, there are three cases. If both $x,y$ are at depth greater or equal to $q$, then $\eta(x,y)$ is a geodesic of length $1$. If both $x,y$ are at depth smaller than $ q$, then by \cref{lem:composition_pi_iota}  $\eta(x,y)$ is a geodesic connecting $[x,y]$, and hence has length $1$. In the case where only one of $x,y$ lays at depth greater than or equal to $q$, without loss of generality let $x$ be the one at depth less than $q$. Let $z\in [x,y]$ be the closet point to $x$ that has depth $q$. Then $\eta(x,z)$ and $\eta(z,y)$ both has diameter at most $1$. It follows that $\eta(x,y)$ has length at most $2$. 
 
    \begin{claim*}
        For all $x,y,z \in L$, \[\eta(x,y) \subseteq N_{2q\delta}(\eta(x,z) \cup \eta(y,z)). \]
    \end{claim*}


    Consider any three different points $x,y,z\in \A_r$. Let $\hat x =\pi_q^r(x), \hat y = \pi_q^r(y), \hat z = \pi_q^r(z) $.
    Let $[\hat x, \hat y]$, $[\hat y, \hat z]$, and $[\hat x, \hat z]$ be the geodesic in $\Acal$ whose pullbacks are contained in $\eta(x,y)$, $\eta(y,z)$, and $\eta(x,z)$ respectively. Consider now the geodesic triangle formed by $[\hat x, \hat y]$, $[\hat y, \hat z]$, and $[\hat x, \hat z]$.
    In the degenerate case where the triangle is a geodesic bigon, \cref{lem:extended-pullback-are-close-bigon} gives the conclusion.
     
    Suppose now we have a non-degenerate geodesic triangle. If two sides intersect other than at the endpoints, then the respective subsegments between one endpoint and the intersection vertex form a bigon. By \cref{pullback-is-extended-pullback} they correspond to two extended pullbacks with same endpoints respectively, hence by  \cref{lem:extended-pullback-are-close-bigon}, the extended pullbacks are within $2q\d$ of each other. 

    We may therefore assume each pair of sides only intersect at endpoints. Notice that if any two sides meet with angle greater than or equal to $ D$, then by \cref{lem:GeodpassBigAngle} the third side has to pass through that apex, which contradicts the assumption that the triangle is non-degenerate. Hence no two sides meet at angle greater than or equal to $ D$. 

    Let $p$ be a vertex on $\eta(x,y)$ and let $\hat p = \pi_q^r (p)$. If $p$ is at depth $>q$ and is in the same horoball of $x$ (resp. $y$), then $p$ is $ \delta$ close to $\eta (x,z)$ (resp. $\eta(y,z)$). Suppose now $\hat p$ lies in the interior of $[\hat x,\hat y]$. Let $\hat q \in [\hat x, \hat z]\cup [\hat y,\hat z]$ be the closest vertex to $\hat p$. Then $d_\Acal(\hat p,\hat q)<\d$. Let $\hat c =[\hat p, \hat q]$. It suffices to show that $\hat c$ does not pass any apex angle greater than or equal to $ D$.

    Without loss of generality we assume $\hat q$  lands on $[\hat x, \hat z]$. Suppose $\hat c$ passes $\hat a$ with angle greater than or equal to $ D$, and let $\hat a_1, \hat a_2$ be the two adjacent vertices. It is straightforward to see that $\hat a$ does not appear elsewhere on $\hat c$ and that $\hat a$ does not land in $[\hat x, \hat z]$ or $[\hat y, \hat z]$. 
    
    Suppose $\hat a$ lands in $[\hat p, \hat y]$. Then $\hat a \notin [\hat x,\hat p]$. As in \cref{lem:extended-pullback-are-close-bigon}, if $d_\Acal(\hat x, \hat p)\leq3\d$, then $[\hat a_1, \hat p, \hat x, \hat  q, \hat a_2]$ forms a path of length $\leq 8\d -2 < D$ that excludes $\hat a$. If $d_\Acal(\hat x, \hat p)>3\d$, then take $\hat u\in [\hat p, \hat x]$ the vertex $3\d$ away from $\hat p$, and $\hat v \in [\hat x, \hat q]$ the vertex closes to $\hat u$. It follows that $[\hat a_1,\hat p, \hat u, \hat v, \hat q, \hat a_2]$ has length at most $10\d <D$, and does not contain $\hat a$. Therefore $\hat a$ does not land in $[\hat p, \hat y]$.
    
    \begin{figure}[h]
    \centering
    \includegraphics[width=0.5\linewidth]{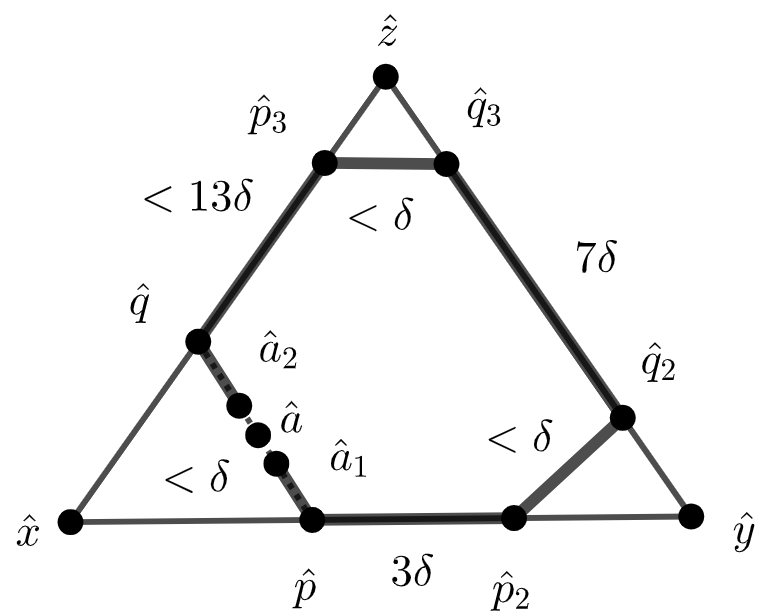}
    \caption{The geodesic triangle in the proof of \cref{lem:ext-pullback_close_to_geodesic}}
    \label{fig:slim-triangle}
\end{figure}

    Let $\hat p_2$ be the vertices $3\delta$  away from $\hat p$ in $[\hat p, \hat y]$ (or $\hat y$ if $l([\hat p,\hat y])\leq 3\delta $) and let $\hat q_2 \in [\hat x,\hat z]\cup [\hat y,\hat z]$ be the closest point to $\hat p_2$. Notice that $\hat a$ can not land on $[\hat p_2, \hat q_2]$, otherwise $d_\Acal(\hat p,\hat p_2)\leq 2\d$, contradicting the fact that $[\hat x,\hat y]$ is a geodesic.  If $\hat q_2 \in [\hat x,\hat z]$, then $[\hat a_1, \hat p, \hat p_2, \hat q_2, \hat q, \hat a_2]$ does not contain $a$ and has length at most $10\d$. Hence $\hat q_2\in [\hat y, \hat z]$.

    Suppose now $\hat a \notin[\hat y, \hat z]$, let $\hat q_3\in [\hat q_2, \hat z]$ be the vertex $7\delta$ away from $\hat q_2$ (or $\hat z$ if $l([\hat q_2, \hat z])< 7\d$). Let $\hat p_3 \in [\hat x, \hat y] \cup [\hat x, \hat z]$ be the vertex closest to $\hat q_3$. Notice that, if $\hat p_3 \in [\hat x,\hat y]$, then $\hat p_3$ must land in $[\hat x, \hat p]$ because $d_\Acal(\hat q_2, \hat q_3)= 6\d$. But then $\hat q$ would land in $[\hat y, \hat z]$, contradicting our assumption. Therefore $\hat p_3 \in [\hat x, \hat z]$.

    We claim that $\hat a \notin [\hat p_3,\hat q_3]$. 
    Otherwise, $[\hat q_3, \hat a, \hat p, \hat p_2, \hat q_2]$ has length at most $6\d < 7\d$, contradicting the fact that $[\hat y, \hat z]$ is a geodesic. But then, the path $[\hat a_1, \hat p, \hat p_2, \hat q_2, \hat q_3, \hat p_3, \hat q, \hat a_2]$ has length at most $26\delta$ and does not contain $\hat a$, contradicting the assumption that $\angle_{\hat a}(\hat a_1,\hat a_2)\geq D$.  

    Therefore, $[\hat p, \hat q]$ can not pass any apex at angle $>D$. 
    It follows that $p$ is in the $2q\delta$ neighborhood of $\eta(x,z)\cup \eta(y,z)$. By \cref{prop:guess-geodesic}, there exists $N = N(q,\delta)\geq q$ such that for any $x,y\in L$, the Hausdorff distance between $\eta(x,y)$ and any geodesic from $x$ to $y$ is bounded above by $N$. We remark that in the case where we choose $q=r_D$, $N = N(D,r_D,\d)$. 

    Notice that this argument works for any choice of $\eta(x,y)$. The conclusion follows.
\end{proof}

\begin{corollary}
\label{cor:pi_shallow_geod_is_good}
    Let $r_D, r, q, N$ be as in \cref{lem:ext-pullback_close_to_geodesic}. For all $0< d\leq q$, let $\g=[x,y]$ be a $d$-shallow geodesic in $\A_r$. The push-forward $\hat \g\subseteq \Acal$ is a $(d+N,0)$-quasi-geodesic.
\end{corollary}

\begin{proof}
     Let $\eta = \eta_q(x,y)\subseteq \A_r$ be an extended pullback between $x,y$. Let $\hat \eta$ be the push-forward of $\eta$, as constructed in \cref{def:push-forward_under_pi}. By the construction of extended pullback, $\hat \eta$ is a geodesic connecting $\hat x, \hat y$. By \cref{lem:ext-pullback_close_to_geodesic} the Hausdorff distance between $\eta$ and $\g$ is bounded above by $N$. Since $\g$ is $d$-shallow, $\eta$ is $(d+N)$-shallow. It follows that
    \[\label{eq:pi_shallow_geod_good}
    l(\hat \g) \leq l(\g)  \leq l( \eta) \leq (d+N) l(\hat \eta).
\]  
    
    Notice that any subsegment of $\hat \g$ also satisfies this inequation. The conclusion follows.
\end{proof}

\begin{corollary}
    \label{cor:iota_shallow_geod_is_good}
    Let $r_D$ be as in \cref{lem:ext-pullback_close_to_geodesic}. 
    For some $t>0$, let $\hat\g$ be a $(k,l)$-quasi-geodesic in $\Acal$ such that $\hat\g$ passes any apex with angle $\leq t$. 
    Apply  \cref{lem:DeepPen_to_BigAngle} with $D=t$ and let $r_t = r_D(t)$.  
    Then for all integer $r>\max{(r_D, r_t)}$ and all integer $q \in [\max{(r_D, r_t)} ,r-1]$, the pullback $\g \subseteq \A_r$ is a $q$-shallow, $((q+1)k, (q+1)l)$-quasi-geodesic.
\end{corollary}
\begin{proof}
    By \cref{cor:small-angle-shallow-pullback}, $\g$ is $q$-shallow. Let $\hat x, \hat y$ be the endpoints of $\hat \g$  and let  $x,y$ be the endpoints of $\g$. Let $\eta = [x,y]$ be a geodesic in $\A_r$ and let $\hat \eta$ be its push-forward in $\Acal$ under $\pi_q^r$.  We then have:
    \begin{align*}
    l(\g)   & \leq (q+1) l(\hat \g) \\
            & \leq (q+1) (kd_{\Acal}(\hat x, \hat y) + l) \\
            & \leq (q+1) (kl( \hat \eta) + l) \\
            & \leq (q+1) (k l(\eta) + l) \\
            & =  k(q+1)d_{\A_r}( x ,  y) + l(q +1).
    \end{align*}

    The above analysis is true for all subsegments of $\g$, hence the conclusion holds.
\end{proof}

\section{Uniform Quasiconvexity}
\label{sec:uniformqc}

It is easy to see that $\K_r(G,X, \cup_i X_i ,\Pcal)$ is quasi-isometric to $\hat\G(G,X\cup (\cup_i X_i))$. Therefore, the $\gp$--quasiconvexity of $H<G$ leads to a quasi-isometric embedding $\K_r(H, Y, \cup_j Y_j, \Dcal)\to \K_r(G,X, \cup_i X_i, \Pcal)$ for every $r>0$. Specifically, the map from \cref{construction:qc-extend-map-coned-off-cusped-space} is a quasi-isometric embedding.

\begin{construction}
\label{cons:reparametrization-of-edge}
    In \cref{cons:qc-induce-map-coned-off-graph}, the parametrization of the image of any cone edge under $\phi$ is not specified, as it does not affect the conclusion of \cref{lem:qc-induce-map-is-qi}. In particular, for any positive integer $q$, we could parametrize the image of any cone edge $\{h, sE_i\}$ as follows: 
    Let $u$ be the point on $\{h, sE_i\}$ that is $1/q$ away from $h$, and let $v$ be the point on $\{h, u\}$ that is $1/(q l(\nu_i)+q)$ away from $u$.
    \begin{enumerate}
        \item The subsegment $\{u, sE_i\}$ is sent to the subsegment $\{ uc_{E_i}, sc_{E_i} P_j \}$ proportionally.
        \item The subsegment $\{v,u\}$ is sent to the subsegment  $\{ hc_{E_i}, uc_{E_i} \}$ proportionally.
        \item The subsegment $\{ h, v\}$ is sent to $h.\nu_i$ proportionally.
    \end{enumerate}
\end{construction}

\begin{lemma}
    \label{lem:qc-extend-cusped-map-is-QI}
    Make \cref{assume:fg-heg}. Construct $\hat L = \hat \G (H,\Dcal, \Ycal)$, $\hat K = \hat \G(G, \Pcal, \Xcal)$, $\Hbb_r = \K_r(H, Y, \cup_j Y_j, \Dcal)$, and $\Gbb_r =\K_r(G,X, \cup_i X_i, \Pcal)$ as in \cref{cons:generating-set-in-actual-use}.
    Let $\phi_r :\Hbb_r\to \Gbb_r $ be constructed as in \cref{construction:qc-extend-map-coned-off-cusped-space}.  Then $\phi_r$ is a quasi-isometric embedding and is $\l_\phi$-Lipschitz.
\end{lemma}
\begin{proof}
    Let $\pi_{r-1}^r : \Hbb_r\to \GH$ be the quasi-isometry constructed as in \cref{construction:pi_Kr_to_hatK}, 
    and let $ \iota_{r-1}^r : \GG \to\Gbb_r$ be the one constructed as in \cref{construction:iota_hatK_to_Kr}. 
    Let $\hat \phi : \GH \to \GG$ be the quasi-isometrical embedding constructed as in \cref{cons:qc-induce-map-coned-off-graph}. 
    We denote the composition $\iota_{r-1}^r \circ \hat \phi \circ \pi_{r-1}^r$ as $\Phi_r$. Clearly $\Phi_r$ is a quasi-isometric embedding.
    
    By \cref{cons:reparametrization-of-edge}, we may construct $\hat\phi$ such that $\phi_r$ coincides with $\Phi_r$ on vertices with depth smaller or equal to $r$. In particular, a vertical edge $\{(h,0), (h,1)\}$ is sent proportionally to the concatenation of $h.\nu_i$ and $\{ ( hc_{E_i},0), ( hc_{E_i}, 1) \}$. All other edges in $\mathbb{H}_r$ are sent to some edge in $\mathbb{G}_r$. It follows that $\phi_r$ is $\l_\phi$-Lipschitz.
    
    By \cref{lem:qc-induce-map-cusp-is-Lipschitz}, $l(\phi_r(c)) \leq \l_\phi l( c) + \l_\phi$ for all path $ c$, so it suffices to show that there exists $\l, \mu >0 $ such that for any $x,y \in \Hbb_r$,
\[
d_{\Gbb_r}(\phi_r(x), \phi_r(y) ) \geq \l d_{\Hbb_r}(x,y) + \mu.
\]

    If $x,y$ both have depth smaller than $r$, then $d_{\Gbb_r}(\phi_r(x), \phi_r(y) ) = d_{\Gbb_r}(\Phi_r(x), \Phi_r(y) )$ and the conclusion follows. If they both have depth larger than or equal to $r$, then $d_{\Hbb_r}(x,y) \leq 2$. Recall that $\phi_r$ does not change the depth of vertices, so $d_{\Gbb_r}(\phi_r(x), \phi_r(y) ) \leq 2 $ and the conclusion follows.

    Suppose now $x$ has depth smaller than $r$ and $y$ has depth larger than or equal to $r$.
    Let $\g$ be a geodesic connecting $ x,y$, and let $u$ be the vertex on $\g$ that has depth $r$ and is closest to $ x$. Let $\phi_r(x) , \phi_r(y), \phi_r(u)$ be the image of $x, y, u$ under $\phi_r$ respectively. We then have:
    \begin{align*}
        d_{\Gbb_r}(\phi_r(x), \phi_r(y)) & \geq d_{\Gbb_r}(\phi_r(x) , \phi_r(u)) - d_{\Gbb_r}(\phi_r(u), \phi_r(y)) \\
        & \geq d_{\Gbb_r}(\phi_r(x) , \phi_r(u)) - 2 \\
        &  = d_{\Gbb_r}(\Phi_r (x) , \Phi_r (u)) -2.
    \end{align*}
    Again, since $\Phi_r$ is a quasi-isometric embedding, the conclusion follows.
\end{proof}

More importantly, we prove in \cref{thm:uniform_qc} that for large enough $r$, the image of $\K_r(H, Y, \cup_j Y_j, \Dcal)$ is always $\l$--quasiconvex, where $\l$ is a constant independent of $r$.

\UniformQC*

\begin{proof}
    Let $X,Y$ be the generating set constructed as in \cref{cons:generating-set-in-actual-use} and construct $\hat L = \hat \G (H,\Dcal, \Ycal), \hat K = \hat \G(G, \Pcal, \Xcal), \Hbb_r = \K_r(H, Y, \cup_j Y_j, \Dcal), \Gbb_r =\K_r(G,X, \cup_i X_i, \Pcal)$ as in \cref{cons:generating-set-in-actual-use}. 
    Let  $\hat \phi :\hat L \to \hat K = \GH $ be the quasi-isometric embedding from \cref{cons:qc-induce-map-coned-off-graph} for some $X\subseteq G$. 

    We fix the following constants for the proof below. 
    \begin{itemize}
        \item Let $k,l$ be the quasi-isometric constants of $\hat \phi:\hat L \to \hat K$. 
        \item Let $\d>0$ be the hyperbolicity constants from \cref{lem:uniform-hyperbolicity-for-coned-off-cusped-space-Kr} such that for all $r>0$,  $\Hbb_r $, $\Gbb_r$, $\hat L$ and $\hat K$ are all $\d$-hyperbolic. Without loss of generality we may assume $\d$ is a positive integer.
        \item  Let  $D_L = D_L(1,0,\d)>32\d$ be the constant for $\hat L$ obtained from \cref{lem:GeodpassBigAngle}. 
        \item  Let $r_L = r_L (D_L)$ be the constant from \cref{lem:DeepPen_to_BigAngle}. 
        \item Let $N=N(q,\d)$ be the constant from \cref{lem:ext-pullback_close_to_geodesic}. 
        \item Let $M_K$ be the upper bound constant from \cref{lem:AngleUpperBoundForHatSpace}. 
        
        \item Let $r_1 = r_1(M_K( 2^{q+1}))$ be the constant from \cref{lem:DeepPen_to_BigAngle}.
        \item  Let $D_K =  D_K(k(q+N),l,\d) >32\d$ be the constant for $\hat K$ obtained from \cref{lem:GeodpassBigAngle}. Let $r_K = r_K(D_K)$ be the constant from \cref{lem:DeepPen_to_BigAngle}.
        \item Let $q =  \max (r_L, r_K,r_1)$. Set $r_\l=2q$ and let $r > r_\l $.
    \end{itemize}
   
    Let $\phi_r:\mb{H}_r \to \Gbb_r$ be the map from \cref{construction:qc-extend-map-coned-off-cusped-space}. By \cref{lem:qc-extend-cusped-map-is-QI} $\phi_r$ is a quasi-isometric embedding.
    Let $\pi^r = \pi_q^r : \Hbb_r\to \hat L$ be constructed as in \cref{construction:pi_Kr_to_hatK}, and let $\bar \iota^r = \iota_q^r : \hat K \to\Gbb_r$ be constructed as in \cref{construction:iota_hatK_to_Kr}. 
    Observe that by \cref{cons:reparametrization-of-edge} we could construct $\hat \phi$ so that $\phi_r$ coincides with the map $\Phi_r = \bar \iota^r \circ \hat \phi \circ \pi^r$ on vertices at depth  less than $q$. We can thus trace the image of a geodesic under $\Phi_r$ and draw conclusions for $\phi_r$.

    \[\begin{tikzcd}
    \label{fig:diagram-uniform-qc}
        {\g_i\subseteq \Hbb_r} & &{\Gbb_r \supseteq \bar{\g}_i} \\
        {\hat{\g}_i \subseteq \hat L} & & { \hat K \supseteq \hat{\bar \g}_i}
        \arrow["{\Phi_r = \bar \iota^r \circ \hat \phi \circ \pi^r}", dashed, hook, from=1-1, to=1-3]
        \arrow["{\pi^r}"', tail reversed, no head, from=2-1, to=1-1]
        \arrow["{\hat \phi}", hook, from=2-1, to=2-3]
        \arrow["{\bar{\iota}^r}"', from=2-3, to=1-3]
    \end{tikzcd}\]

    Given a geodesic $\g = [x,y]$ in $\Hbb_r$, it can be divided into:
    \[
        (\epsilon_0,) \g_0, \epsilon_1,\g_1,\dots, \epsilon_m, \g_m(, \epsilon_{m+1}),
    \]  
    where $\g_i = [x_i,y_i] $'s are $q$-shallow geodesics, and $\epsilon_j$'s are geodesics that lie in a single horoball at depth greater than or equal to $ q$.  

    
    Let $\hat \g_i$ be the push-forward of $\g_i$ in $\hat L$. Let $\hat{\bar \g}_i$ be the image of $\hat \g_i$ in $\hat K$ for all $i$. Let  $\bar \g_i$ be the pullback of $\hat{\bar \g}_i$ in $\Gbb_r$. Observe that $\Phi_r(\g_i)$ lies in the $2$-neighborhood of $\bar \g_i$, whereas $\bar \g_i$ lies in the $q$-neighborhood of $\Phi_r(\g_i)$. Notice also $\bar \g_i$ and $\Phi_r(\g_i)$ share endpoints. Let $\bar \g$ be the concatenation of $\bar \g_i$'s and $\bar \epsilon_j = \Phi_r(\epsilon_j)$'s. Then $\bar \g$ lies in the $q$-neighborhood of $\Phi_r(\g)$.
    
    By \cref{cor:pi_shallow_geod_is_good}, push-forwards of $\g_i$'s in $\hat L$ are $(q+N,0)$ quasi-geodesics. 
    By \cref{lem:Shallow_to_SmallAngle}, none of them pass through any apex with angle greater than or equal to $ 2^{q+1}$. 
    Thus, by \cref{cor:hat-space-path-angle-bound}, $\hat{\bar \g}_i$'s are $((q+N)k,l)$-quasi-geodesics that don't pass through any apex with an angle larger than $M_K( 2^{q+1})$.
    Then by \cref{cor:iota_shallow_geod_is_good}, $\bar \g_i$'s are $q$-shallow $(k',l')$-quasi-geodesics, where $k'=(q+1)(q+N)k, l'=(q+1)l$. In particular, both $k', l'$ do not depend on $r$.

    Let $x', y'$ be the starting point of $\g_0$ and the ending point of $\g_m$ respectively. Denote the image of $x',y'$ in $K$ (resp. $\hat L$, $\hat K$) as $\bar x', \bar y'$ (resp. $\hat{x}'$ and $ \hat{ y}'$, $\hat{\bar x}'$ and $\hat{\bar y}'$) respectively.

    By the construction of $\pi^r$, the concatenation of $\hat \g_i$'s is a continuous path in $\hat L$. Hence the concatenation of $\hat{\bar \g}_i$'s is a continuous path in $\hat K$. Let $\eta$ be a geodesic connecting $\hat{\bar x}'$  and $\hat{\bar y}'$. By \cref{lem:GeodpassBigAngle} $\eta$ passes all the endpoints of $\hat {\bar\g}_i$'s, Hence $\eta$ can be written as the concatenation of geodesic segments $\eta_i$'s, where $\eta_i$ shares the endpoints of $\hat{\bar \g}_i$ for each $i$. See \cref{fig:image-of-geodesic-in-hat-K}.

    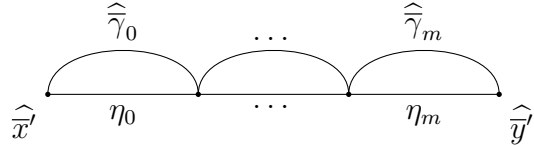
\begin{figure}[h!]
    \centering
        \begin{tikzpicture}
            \fill[] (0,0) circle (1pt) node[below left]{$\hat{\bar x'}$} (2,0) circle (1pt) (4,0) circle (1pt) (6,0) circle (1pt) node[below right]{$\hat{\bar y'}$};
            \draw (0,0) -- (2,0) node[midway, below] {$\eta_0$};
            \draw (2,0) -- (4,0) node[midway, below] {$\dots$};
            \draw (4,0) -- (6,0) node[midway, below] {$\eta_m$};
            \draw (0,0) to[out=90, in=90] node[midway, above] {$\hat{\bar \g}_0$} (2,0);
            \draw (2,0) to[out=90, in=90] node[midway, above] {$\dots$} (4,0);
            \draw (4,0) to[out=90, in=90] node[midway, above] {$\hat{\bar \g}_m$} (6,0);
        \end{tikzpicture}
    \caption{A geodesic between $\hat{\bar x}'$ and $\hat{\bar y}'$ in $\hat K$}
    \label{fig:image-of-geodesic-in-hat-K}
    \end{figure}

    Consider now the extended pullback $\bar \eta$ given by $\eta$ in $\Gbb_r$. Recall \cref{cons:extended-pullback}, $\bar\eta$ is a concatenation of:
    \[
    \bar\eta_0, \tau_1,\dots, \tau_m, \bar \eta_m,
    \]
    where $\bar\eta_i$'s are the pullbacks of $\eta_i$'s respectively, and $\tau_j$'s are the geodesics connecting endpoints of $\bar\eta_{j-1}$ and $\bar\eta_j$.

    Let $\zeta$ be the geodesic connecting $\bar x'$ and $\bar y'$, and $\xi_i$'s the geodesics connecting $\bar x_i$'s to $\bar y_i$'s respectively. 
    By \cref{lem:ext-pullback_close_to_geodesic}, $\zeta$ is $N$-close to $\bar \eta$. 

    \begin{figure}[h]
        \centering
            \begin{tikzpicture}[relative]
                \coordinate (a) at (0,0);
                \fill[] (a) circle (1pt) node[below left] {$\bar x'$};
                \coordinate (b) at (8,0);
                \fill[] (b) circle (1pt) node[below right] {$\bar y'$};
                \draw (a) -- (b) node[midway, below] {$\zeta$};
                \coordinate (c) at (1.5,1.5);
                \fill[] (c) circle (1pt);
                \draw (a) to[out= 90, in= 90] node[midway, above left] {${\bar \eta}_0$} (c);       
                \draw (a) -- (c) node[midway, above left]{$\xi_0$};
                \draw (a) to[out= -30, in= -150] node[midway, below right] {${\bar \g}_0$} (c);
                \coordinate (d) at (3.5, 2);
                \fill[] (d) circle (1pt);
                \draw (c) to[out= -60, in= -120] node[midway, below right] {${\bar \epsilon}_1$} (d);       
                \draw (c) -- (d) node[midway, above left]{$\tau_1$};
                \coordinate (e) at (4.5,2);
                \fill[] (e) circle (1pt);
                \path (d) -- node[auto=false]{\ldots} (e);
                \coordinate (f) at (6.5,1.5);
                \fill[] (f) circle (1pt);
                \draw (e) to[out= -60, in= -120] node[midway, below left] {${\bar \epsilon}_m$} (f);       
                \draw (e) -- (f) node[midway, above right]{$\tau_m$};
                \draw (f) to[out= 90, in= 90] node[midway, above right] {${\bar \eta}_m$} (b);       
                \draw (f) -- (b) node[midway, above right]{$\xi_m$};
                \draw (f) to[out= -30, in= -150] node[midway, below left] {${\bar \g}_m$} (b);

            \end{tikzpicture}
        \caption{A geodesic between $\bar x'$ and $\bar y'$ in $\Gbb_r$}
        \label{fig:enter-label}
    \end{figure}
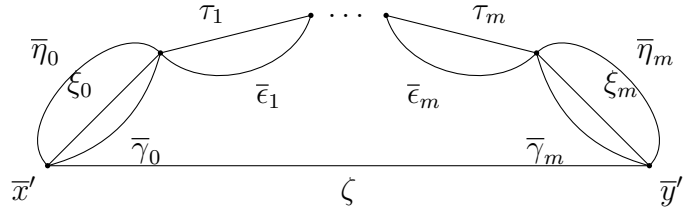

    \begin{claim*}
        The extended pullback $\bar \eta$ lies in a bounded neighborhood of $\bar \g$.
    \end{claim*}

    By \cref{lem:ext-pullback_close_to_geodesic}, $\xi_i$'s are $N$-close to the pullback $\bar\eta_i$'s respectively. On the other hand, $\bar\g_i$'s are $(k',l')$ quasi-geodesics, hence there exists $\mu= \mu(k','l',\delta)$ such that $\bar \g_i$ lies in the $\mu$ neighborhood of $\xi_i$. It follows that $\bar \eta_i$'s lie in the $\mu + N$ neighborhood of $\bar \g_i$'s.

    Consider now the distance between $\tau_j$ and $\bar \epsilon_j$ for each $j$. Notice that $\bar \epsilon_j$ consists of two vertical subsegments and one horizontal subsegment, whereas $\tau_j$ consists of zero or two vertical subsegments and one horizontal subsegment of length at most $3$. Since $\bar \epsilon_j$ shares endpoints with $\tau_j$, and two vertical segments initiating at the same vertex have to lie in the same ray, it is straightforward that $\tau_j$ lies in the $2$ neighborhood of $\bar \epsilon_j$.

    It follows that $\zeta$ lies in the $\mu + 2N + 2$ neighborhood of the subsegment of $\bar \g$ between $\bar x'$ and $\bar y'$. If $\epsilon_0$ and $\epsilon_{m+1}$ don't exist, then we are done. Suppose now $\epsilon_0$ and $\epsilon_{m+1}$ both exist. Then the geodesic $[\bar x, \bar y]$ lies in the $2\delta$ neighborhood of $\zeta \cup \bar \epsilon_0 \cup \bar \epsilon_{m+1}$. In the case where only one of those segments exists, the same bound works. 

    Set $\l = \mu + 2N + 2 + 2\delta + q$. We have $\zeta$ lies in the $\l$-neighborhood of $\Phi_r(\g)\subseteq \Phi_r(\Hbb_r)$.  It follows that $\Phi_r(L)$ is $\l$-quasiconvex. In particular, notice that none of $\mu, N, \delta,q$ depends on the choice of $r$.
\end{proof}

\begin{proposition}
    \label{thm:finite_index_qc}
    Suppose $H$ is $\gp$--quasiconvex. Let $H'$ be a finite index subgroup of $H$, then $H'$ is also $\gp$--quasiconvex.
\end{proposition}

\begin{proof}

    Since $H$ is $\gp$--quasiconvex, it acts on a subgraph $\mathcal{L}$ of some $\gp$--graph. Consider the action of $H'$ on $\mathcal{L}$. Clearly $\mathcal{L}$ is $H'$-invariant. Since $H'$ is a finite index subgroup of $H$, $\mathcal{L}$ has finitely many $H'$-orbits of vertices. It follows that $H'$ is also $\gp$--quasiconvex.
\end{proof} 

\section{Application}
\label{sec:main}

As an application of \cref{thm:uniform_qc}, we generalize \cite[Corollary 6.6]{grovesHyperbolicGroupsActing2023} below. Recall the definition of full $\gp$-quasiconvex subgroup as in \cref{def:qc_finegraph}.

\MainCube*

The proof for \cite[Corollary 6.6]{grovesHyperbolicGroupsActing2023} relies on a series of technical lemmas. We start by generalizing preliminary results in \cref{sec:geods,sec:meta}. We remark that in most cases, the generalization requires few modifications. \cref{thm:uniform_qc} is used in the proof of \cref{lem:qc-deep-pen-infinite-intersect}.  We give the proof of \cref{thm:main_cube}  in \cref{subsec:main-proof}.

In \cref{sec:image-of-H-under-dehn-filling} we also prove some results that may be of independent interest.

\subsection{Projections of Geodesics}
\label{sec:geods}
In this section, we generalize \cite[Proposition 4.5]{agolResidualFinitenessQCERF2009} to the setting of finitely generated groups with hyperbolically embedded subgroups. As in the original proof, we first prove two technical lemmas which correspond to \cite[Lemma 4.1, Lemma 4.2]{agolResidualFinitenessQCERF2009}, respectively.

\begin{assumption}
    \label{assume:gp-qc}
    We assume $G$ is a finitely generated group, $\Pcal = \{ P_i \}_{1\leq i \leq n} $ a finite collection of subgroups of $G$, and $H$ a $\gp$--quasiconvex subgroup. We assume further that $X, Y$ are the generating set of $G,H$ from \cref{thm:uniform_qc} respectively. 
    
    We write $\Gbb_r$ for $\K_r(G,X,\cup_i X_i,\Pcal)$ and $\Hbb_r$ for $\K_r(H,Y,\cup_j Y_j,\Dcal)$.
\end{assumption}

We write $|\cdot|_{\Gbb_r}$ for $d_{\Gbb_r}(\cdot,1)$. Given a filling $\pi:G\to \bar G$ with kernel $K$, we write the  $\bar \Gbb_r = \Gbb_r/K$, and denote the quotient map as $\pi_r$. Recall that for all $r>0$, $\bar \Gbb_r$ contains a copy of the Cayley graph of $\bar G$. Therefore, by a slight abuse of notation, for all $r>0$ and for all $g\in G$, we denote $\pi_r(g)$ as $\pi(g)$.

\begin{lemma}[Generalization of {\cite[Lemma 4.1]{agolResidualFinitenessQCERF2009}}]
    \label{thm:GendGreendlinger}
    Make \cref{assume:fg-heg}. Let $r_\d>0 $ and $\d>0$ be the constants from \cref{prop:quotient-space-hyperbolic} such that for all $r>r_\d$, for any sufficiently long filling with kernel $K$, the coned-off cusped space $\Gbb_r=\K_r(G,X, \cup_i P_i, \Pcal)$ and $\bar \Gbb_r=\Gbb_r /K$ are both $\d$-hyperbolic.

    Let $L\geq 10\delta$ and let $D\geq 3L$. Let $F=\{g\in G \mid \forall r >0, |g|_{\Gbb_r} \leq 2D\}$, and let $\pi :G\to \bar G = G(N_1, \dots, N_m)$ be a filling which is injective on $F$ and so that for all $r>r_\d$, $\bar \Gbb_r$ is $\delta$-hyperbolic. For all $r>\max(r_\d , D)$, let $\g$ be a geodesic in $\Gbb_r$ joining two elements in $G$. Then one of the following occurs:

\begin{enumerate}
    \item There is a $10\delta$-local geodesic with the same endpoints as $\pi_r(\g)$ which is contained in a $2$-neighborhood of $\pi_r(\g)$ and coincides with $\pi_r(\g)$ everywhere in an $L$-neighborhood of the Cayley graph of $\Bar{G}$. Moreover, $\pi_r(\g)$ is a quasi-geodesic.
    \item There is a coset $tP_i$ whose corresponding horoball intersects $\g$ in a subsegment $[g_1,g_2]$ of length at least $2D-20\delta-4$, but there is some $n\in N_i$ with $d_{\Gbb_r}(g_1,g_2 n) \leq 2L+3$.  
\end{enumerate}
\end{lemma}

\begin{proof}
    Recall the shape of regular geodesic (See \cref{def:regular-geodesic}). In horoballs of $\Gbb_r$ the geodesics consist of either two vertical subsegments of length smaller than $r$ and one horizontal subsegment of length 2 or 3, or two vertical subsegments of length $r$ and two cone edges. With this in mind, the proof of \cite[Lemma 4.1]{agolResidualFinitenessQCERF2009} works with the obvious modification.

    Moreover, if conclusion (1) holds, then the proof of \cite[Lemma 4.1]{agolResidualFinitenessQCERF2009} constructs a $10\d$-local geodesic $\g'$ by replacing $\pi_r(\g \cap B)$ with a geodesic of same endpoints if necessary, where $B$ is some horoball that $\g$ penetrates to depth greater than $D-10\d -2$. Recall that $\g$ penetrates $B$ to depth smaller than $r+1$. Observe also that the constructed segment penetrates $B$ to depth greater than $L$.  By the construction of $\g'$, we have:
    \[
       \frac{L}{r} l(\pi_r(\g))  \leq l(\g') \leq  l(\pi_r(\g)).
    \]

    By \cite[III.H.1.13]{bridsonMetricSpacesNonPositive1999}, any $10\d$-local geodesic in $X'$ is a $(7/3, 2\d)$-quasi-geodesic. It follows that $\pi_r(\g)$ is a $(7r/3L, 10r \d /L )$-quasi-geodesic.
\end{proof}

\begin{lemma}[Generalization of {\cite[Lemma 4.2]{agolResidualFinitenessQCERF2009}}]
    \label{thm:GendGreendlinger-qc}
    Make \cref{assume:gp-qc}. Let $r_\l >0 , \l>0$ be the depth constant and quasiconvex constant from \cref{thm:uniform_qc}. 
    Let $r_\d,\d, L\geq 10\delta$ be as in the hypothesis of \cref{thm:GendGreendlinger}.
    Let $\mu$ be the Lipschitz constant $\l_\phi$ of $\phi_r$ from \cref{lem:qc-extend-cusped-map-is-QI}. 
    Let $D \geq 3L + 100 \lambda + 4 \mu $, $r>\max(r_\d, r_\l, D)$
    and let $F = \{g\in G \mid \forall r >0, |g|_{\Gbb_r} \leq 2D\}$. Suppose $\pi: G\to \bar G =  G(N_1,\dots, N_m)$ is an $H$-filling with kernel $K$ which is injective on $F$ and so that $\bar \Gbb_r = \Gbb_r/K$ is $\delta$-hyperbolic. Let $K_H < K \cap H$ be the kernel of the induced filling of $H$ (as in \cref{def:InducedFilling}). Finally, suppose that $\g$ is a geodesic joining $1$ to $h$ for some $h\in H$ in $\Gbb_r$.
    
    If conclusion (2) of \cref{thm:GendGreendlinger} holds, then for some $k\in K_H$, we have $|kh|_{\Gbb_r}< |h|_{\Gbb_r}$.
\end{lemma}
\begin{proof}

    Let $g_1, g_2 \in tP_i$ and $n\in N_i$ be as in the conclusion of \cref{thm:GendGreendlinger}, and let $Q$ be a horoball in $\Gbb_r$ which contains $tP_i$. We have
    \begin{equation}
        \label{eq-1}
        d_{\Gbb_r}(g_1,g_2) \geq 2D-20\d - 4 \geq 6L + 200 \l + 8 \mu - 20\d -4, \text{but,}
    \end{equation}
    \begin{equation}
        \label{eq-2}
        d_{\Gbb_r}(g_1, g_2 n) \leq 2L+3.
    \end{equation}

    If $g\in tP_i$ and $d\in \N$, then we write $(g,d)$ for the unique vertex of $B$ that is connected to $g$ by a vertical geodesic of length $d$.

    By Inequality (\ref{eq-1}), the geodesic $\g$ penetrates the horoball $B$ to depth at least $D-10\d -4 > 2\l$; in particular, since $r>D>100\l$,  $\g$ passes through $(g_1, \l+1)$ and $(g_2, \l+1)$. 
    
    Since $H$ is $\l$-$\gp$--quasiconvex, there exist points 
    \[z_1 = \phi_r((s_1 D_{j_1},h_1,d_1)),\quad z_2 = \phi_r((s_2 D_{j_2},h_2,d_2)) \in B\] such that $z_1, z_2$ are within $\l$ of $(g_1,\l+1)$ and $(g_2,\l+1)$, respectively. Notice that both $d_1,d_2$ have to be smaller than $2\l+2 <r$. 
        
    The argument from \cite[Lemma 4.2]{agolResidualFinitenessQCERF2009} then shows that for some $k\in K_H$, $|kh|_{\Gbb_r} < |h|_{\Gbb_r}$.
\end{proof}

\begin{proposition}[Generalization of {\cite[Proposition 4.5]{agolResidualFinitenessQCERF2009}}]
\label{thm:seperate}
Make \cref{assume:gp-qc} and let $I>0$. There is some $F= F(I)$ so that if $G(N_1,\dots, N_n)$ is an $H$-filling of $G$ which is injective on $F$, and $g\in G\backslash H$ satisfies $\forall r>0,|g|_{\Gbb_r} < I$, then $\pi(g) \notin \pi(H)$.
\end{proposition}
\begin{proof}
    Let $r_\d, \d$ be the constants from \cref{prop:quotient-space-hyperbolic} such that for all $r>r_\d$, both $\Gbb_r$ and $\bar \Gbb_r$ are $\d$-hyperbolic. Let $r_\l$, $\l$ be depth constant and the constant of uniform-quasiconvexity from \cref{thm:uniform_qc} for $H$ respectively. 

    Choose $L=2I+10\delta$ and $D=100\l + 100\delta + 6I$. Suppose $r>  \max(r_\l, r_\d , D)$. Let $F=\{ g\in G\mid \forall r>0,  |g|_{\Gbb_r} \leq 2D\}$, and let $G(N_1,\dots, N_n)$ be an $H$-filling of $G$ which is injective on $F$ and so that the associated cusped coned-off Cayley graph is $\d$-hyperbolic.

    If the proposition does not hold, then $\exists g\in G\backslash H$ such that $|g|_{\Gbb_r} <I$ for all $r>0$, and $\exists h\in H$ such that $\pi(g) = \pi(h)$. Without loss of generality, choose $h$ so that $|h|_{\Gbb_r}$ is minimal. Let $\g$ be a geodesic joining $1$ to $h$.

    Apply  \cref{thm:GendGreendlinger}. If conclusion 2 holds, then by \cref{thm:GendGreendlinger-qc} we can find $k\in K$ such that $|kh|_{\Gbb_r}< |h|_{\Gbb_r}$ and $\pi(kh) = \pi(g)$, contradicting the choice of $h$, therefore conclusion 1 must hold. Then there exists a $10\delta$-local geodesic $\g'$ connecting $1$ and $\pi(g)$  which lies in a $2$-neighborhood of $\pi_r(\g)$ and coincides with $\pi_r(\g)$ in a $(2I+10\delta)$-neighborhood of both $1$ and $\pi(g)$. Then $l(\g')\geq 4I+20\delta$. But since $\g'$ is a $10\delta$-local geodesic, it must be a $(7/3,2\delta)$-quasi-geodesic, so the distance between $1$ and $\pi(g)$ is at least
    \[
    \frac37(4I+20\delta) - 2\delta > I.
    \]
    It follows that $|g|_{\Gbb_r}>I$, a contradiction.
\end{proof}

\subsection{The image of H under Dehn Filling}
\label{sec:image-of-H-under-dehn-filling}
The following results are not required for the proof of \cref{thm:main_cube}, but may be of independent interest.

\begin{proposition}[Generalization of {\cite[Proposition 4.4]{agolResidualFinitenessQCERF2009} }]
    \label{prop:induce-filling-map-injective}
    Make \cref{assume:gp-qc}. For any sufficiently long $H$--filling $\bar G$ of $G$, the induced map from the induced filling $\bar H$ into $\bar G$ is injective. 
\end{proposition}
\begin{proof}
    By \cref{lem:uniform-hyperbolicity-for-coned-off-cusped-space-Kr,prop:quotient-space-hyperbolic}, there exists $r_\d >0$ and  $\d >0$ such that for all $r>r_\d$,  $\Gbb_r$ and $\bar \Gbb_r$ are both $\d$-hyperbolic. 
    Let $r_\l,\l$ be the constants from \cref{thm:uniform_qc}.
    
    We apply \cref{thm:GendGreendlinger,thm:GendGreendlinger-qc} with $L= 10 \d$ and $D = 100 \l + 100 \d$. 
    So ``sufficiently long'' means that the filling is injective on $F=\{g\in G \mid \forall r > 0, |g|_{\Gbb_r} \leq 2D\}$ and so that $\bar \Gbb_r$ is also $\d$-hyperbolic.

    Let $\pi: G\to \bar G$ be such a filling. Let $h\in \ker (\pi) \cap H $ be nontrivial. We must show that $h\in K_H$, the kernel of the induced filling of $H$.  
    Note that $X$ is a generating set of $G$, so $\G(G,X)$ is connected. It follows that there exists $r_h>0$ such that when $r>r_h$, any such geodesics in $\Gbb_r$ that connects $1$ to $h$ does not pass through any apex. That is, for all $r$, $|h|_{\Gbb_r}$ are the same. We denote that length as $|h|_G$.
    We fix $r>\max(|h|_G, r_h, D, r_\d, r_\l)$ in the rest of the proof.
    
    The argument from the proof of \cite[Proposition 4.4]{agolResidualFinitenessQCERF2009} then shows that there is a $k\in K_H$ so that $|kh|_{\Gbb_r} < |h|_G$. 
    We remark that this observation means any geodesic between $1$ and $kh$ does not pass through any apex.
    If $kh\in K_H$, then we are done. If not, we repeat the argument for $kh$. Since $|h|_G$ is finite, the induction on the length of $h$ terminates after finite steps. It follows that $h\in K_H$ as required.
\end{proof}

\cref{prop:image-of-H-convex} is a more detailed version of \cref{prop:preserve-QC-rephrase}.
\begin{theorem}[Generalization of{ \cite[Proposition 4.3]{agolResidualFinitenessQCERF2009}}]
    \label{prop:image-of-H-convex}
    Make \cref{assume:gp-qc}. There exists $\beta >0, \l'>0$ such that for all sufficiently long $H$-fillings $\bar G$ of $G$ with kernel $K$, for all $r>\beta$, the image of the induced filling $\bar \Hbb_r $ in $\bar \Gbb_r = \Gbb_r/K$ is $\l'$--quasiconvex.
    
    Moreover, $\bar H$ is $(\bar G, \bar {\Pcal})$--quasiconvex.
\end{theorem}
\setcounter{case}{0}

\begin{proof}
     As in the proof of \cref{prop:induce-filling-map-injective}, let $r_\d >0$ and  $\d >0$ such that for all $r>r_\d$,  $\Gbb_r$ and $\bar \Gbb_r$ are both $\d$-hyperbolic. Let $r_\l,\l$ be the constants from \cref{thm:uniform_qc}. 
     Without loss of generality, we may assume that both $\d$ and $\l$ are positive integers.
     Once again, we apply \cref{thm:GendGreendlinger,thm:GendGreendlinger-qc} with $L= 10\d$ and $D= 100\d +100 \l$. Let $\beta = \max(r_\l, r_\d, D)$ and suppose $r>\beta$.
      ``Sufficiently long'' again means that the filling is injective on $F=\{g\in G \mid \forall r >0, d_{\Gbb_r}(g,1) \leq 2D\}$ and that for all $r>r_\d$, $\bar \Gbb_r$ is $\d$-hyperbolic.

    By \cite[III.H.1.13]{bridsonMetricSpacesNonPositive1999}, any $10\d$-local geodesic in a $\d$--hyperbolic space is a $(7/3, 2\d)$-quasi-geodesic. Let $R$ be the constant of quasi-geodesic stability for  $(7/3, 2\d)$-quasi-geodesics in a $\d$-hyperbolic space (see \cite[III.H.1.7]{bridsonMetricSpacesNonPositive1999}). 

    Let $\phi_r : \Hbb_r \to \Gbb_r$ be the quasi-isometric embedding from \cref{thm:uniform_qc}. Let $\pi:G\to \bar G$ be the filling map and let $K$ be the filling kernel. Let  $\pi_r:\Gbb_r \to \bar \Gbb_r$ be the quotient map by $K$.

    \begin{claim}
        \label{claim:geodesic-stay-close-to-image}
        Let $\bar h \in \pi(H)$. Then for all $r>\max (r_\l,r_\d)$, any geodesic in $\bar \Gbb_r$ joining $1$ to $\bar h$ stays in an $(\l + R + 2)$-neighborhood of the image of $\pi\circ \phi_r$.
    \end{claim}

    We choose $h\in H$ of minimal $\Gbb_r$ length projecting to $\bar h$, and let $\g$ be a geodesic joining $1$ to $h$ in $\Gbb_r$. The same argument as in \cite[Claim 4.3.1]{agolResidualFinitenessQCERF2009} proves the claim.

    Let $x_1 ,x_2$ be two points in $\pi_r\circ \phi_r (\Hbb_r)$. If $x_1, x_2$ lie in the same horoball, then by \cref{lem:coned-off-horoball-is-convex} any geodesic joining them stays in a $2\d +2$-neighborhood of $\pi_r\circ \phi_r (\Hbb_r)$.

    Suppose therefore that $x_1, x_2$ lie in different horoballs. For $i=1,2$, $x_i$ is connected by a vertical geodesic and at most $2$ cone edges to some $h_ic$ for $h_i\in \pi (H)$ and $|c|_{\Gbb_r} < \l_\phi$, where $\l_\phi$ is the constant from \cref{lem:qc-extend-cusped-map-is-QI}. The vertical segments and the cone edges are in the image of $\Hbb_r$, so a similar argument as in \cite[Proposition 4.3]{agolResidualFinitenessQCERF2009} shows that any geodesic between $x_1,x_2$ stays within $\l+R+2+4\d +\l_\phi$ of $\pi_r\circ \phi_r(\Hbb_r)$.

    We now prove that $\bar H$ is $(\bar G, \bar {\Pcal})$--quasiconvex. 
    
    Given two numbers A and B, we say A and B are $(K,L)$-coarsely equivalent, denoted as $A \asymp_{K,L} B$, if  $\frac 1 K A -L \leq B \leq K A + L$. We say $A,B$ are coarsely equivalent, denoted as $A \asymp B$, if $A,B$ are $(K,L)$-coarsely equivalent for some $K,L$. Clearly a map $\phi:X\to Y$ is a $(K,L)$--quasi-isometric embedding if and only if for any pair $x,y\in X$, $d_X(x,y)\asymp_{K,L} d_Y(\phi(x),\phi(y))$

    \begin{claim}
    \label{claim:geodesic-project-to-qgeod}
        Let $\bar h \in \pi(H)$. Then for all $r>\max (r_\l,r_\d)$, there exists $\L_r,M_r$ such that the length of any geodesic in $\bar \Gbb_r$ joining $1$ to $\bar h$ is $(\L_r, M_r)$--coarsely equivalent to the length of any geodesic in $\bar \Hbb_r$ joining $1$ to $\bar h$.
    \end{claim}
    
    We choose $h\in H$ of minimal $\Hbb_r$ length projecting to $\bar h$, and let $\g_H$ be a geodesic joining $1$ to $h$ in $\Hbb_r$. By \cref{thm:GendGreendlinger,thm:GendGreendlinger-qc}, and the minimality of $h$, $\pi (\g_H) $ is a quasi-geodesic in $\bar \Hbb_r$. 

    Let $\g_G$ be a geodesic joining $1$ to $h$ in $\Gbb_r$.  Apply \cref{thm:GendGreendlinger} to $\g_G$. If conclusion (1) holds, then $\pi(\g_G)$ is a quasi-geodesic in $\bar \Gbb_r$. Recall that  $\phi_r$ is a quasi-isometric embedding. Let the quasi-isometric constants be $\l_r$ and $\mu_r$.  It follows that:
    
    \[ 
        |\bar h|_{\bar \Hbb_r} \asymp_{7r/3L, 10r\d /L} \pi_r(\g_H) \asymp_{\l_r, \mu_r} \pi_r(\g_G) \asymp_{7r/3L, 10r\d /L} | \bar h |_{\bar \Gbb_r}.
    \] 

    If conclusion (2) of \cref{thm:GendGreendlinger} holds, then by \cref{thm:GendGreendlinger-qc}, there exists $k\in K_H$ such that $| kh |_{\Gbb_r} < |h|_{\Gbb_r}$. We choose $k$ so that $kh$ has minimal $\Gbb_r$ length.  
    Since $kh\in H$, by the minimality of $h$, $|kh|_{\Hbb_r} > |h|_{\Hbb_r}$. It follows that:
    \[
      \frac{1}{\l_r} |h|_{\Hbb_r} - \mu_r  < \frac{1}{\l_r} |kh|_{\Hbb_r} - \mu_r   \leq   |kh|_{\Gbb_r}  < |h|_{\Gbb_r}  \leq \l_r |h|_{\Hbb_r}  + \mu_r.
    \]
    Let $\g'_G$ be a geodesic in $\Gbb_r$ connecting $1$ to $kh$. By \cref{thm:GendGreendlinger,thm:GendGreendlinger-qc}, and the minimality of $kh$, $\pi(\g'_G)$ is a quasi-geodesic in $\bar\Gbb_r$. It follows that:
    \[
         |\bar h|_{\bar \Hbb_r} \asymp_{7r/3L, 10r\d /L} \pi_r(\g_H) \asymp_{\l_r, \mu_r} \pi_r(\g'_G) \asymp_{7r/3L, 10r\d /L} | \bar h |_{\bar \Gbb_r}.
    \]
    Combining the coarse equivalent constants, we get $\L = \L(r, L, \d, \l_r,\mu_r)$ and $M = M (r,L, \d, \l_r, \mu_r)$ such that 
    \[
        |\bar h|_{\bar \Hbb_r} \asymp_{\L_r , M_r }| \bar h |_{\bar \Gbb_r} .
    \]

    Notice that $\L_r, M_r$ do not depend on the choice of $\bar h$. 
        
    Consider now the extension of inclusion $\check \iota : \hat \G_{\bar H} = \hat\G(\bar H, \bar \Dcal, \bar Y) \to  \hat \G_{\bar G} =  \hat\G(\bar G, \bar P, \bar X)$ (See \cref{def:extension-of-map-to-coned-off}).

    For any $u,v \in \hat \G_{\bar H}$, one of the following cases holds:
    \begin{case}
        Both $u,v$ are elements in $\bar H$. 
    
    \end{case}
        Then $v=hu$ for some $h\in \bar H$. Recall that \cref{construction:pi_Kr_to_hatK} and \cref{construction:iota_hatK_to_Kr} provide quasi-isometries between the coned-off Cayley graphs and the corresponding coned-off cusped Cayley graphs. Fix $r>\max(r_\l, r_\d)$, it follows that:
    \[
        d_{\hat \G_{\bar H}}(1, \bar h) \asymp_{r+1,0}  d_{\hat \G_{\bar \Hbb_r}}(1, \bar h) \asymp_{\L_r, M_r} d_{\hat \G_{\bar \Gbb_r}}(1, \bar h) \asymp_{2(r+1),0}  d_{\hat \G_{\bar G}}(1, \bar h).
    \]

    \begin{case}
        Both $u,v$ are apices.
    \end{case}

    In this case, let $\eta$ be a geodesic connecting $u,v$. Then $\eta$ is the concatenation of two cone edges and a geodesic $\eta'$ connecting $u',v'$, where both $u',v'\in \bar H$.

    Recall that a cone edge $\{h . D,  h\}$ is sent to the concatenation of $\{  h  c_{ D}. P,  h  c_{ D}\}$ and $ h . \nu_{ D}$. Notice that $\pi(\nu_D)$ is a path connecting $1$ and $\pi(c_D)$ with length at most $\l_\phi$ (See \cref{cons:qc-induce-map-coned-off-graph} and \cref{lem:qc-induce-map-is-qi}).     

    We then have:
    \begin{align*}
      d_{\hat \G_{\bar G}} (\check \iota (u'), \check\iota (v'))  -2\l_\phi 
      \leq  d_{\hat \G_{\bar G}} (\check \iota (u), \check\iota (v)) 
      \leq d_{\hat \G_{\bar G}} (\check \iota (u'), \check\iota (v')) + 2 \l_\phi, \\
      d_{\hat \G_{\bar G}} (\check \iota (u'), \check\iota (v')) \asymp_{2(r+1)^2\L_r, M_r} d_{\hat \G_{\bar H}} (u' , v') =  d_{\hat \G_{\bar H}} (u, v) -2. 
    \end{align*}

    It follows that:
    \[
        d_{\hat \G_{\bar G}} (\check \iota (u), \check\iota (v))  \asymp_{2(r+1)^2\L_r, 3M_r+2\l_\phi}  d_{\hat \G_{\bar H}} (u, v) .
    \]

    \begin{case}
        One of $u,v$ is an apex.
    \end{case}
    This case is similar to the case where both $u,v$ are apices. A similar argument shows that \[ d_{\hat \G_{\bar G}} (\check \iota (u), \check\iota (v))  \asymp_{2(r+1)^2\L_r, 2M_r+ \l_\phi}  d_{\hat \G_{\bar H}} (u, v) . \]

    Thus we can find $\L_r' , M_r'$ such that $d_{\hat \G_{\bar G}} (\check \iota (u), \check\iota (v))  \asymp_{\L_r', M_r'}  d_{\hat \G_{\bar H}} (u, v)$ for all pairs of $u,v$. It follows that $\check \iota$ is a quasi-isometric embedding. By \cref{prop:equiv-gp-qc-by-coned-off-cayley-graph}, $\bar H$ is $(\bar G, \bar \Pcal)$--quasiconvex.
\end{proof}
\subsection{Meta-Conditions}
\label{sec:meta}

A key ingredient in the proof of \cref{thm:main_cube} is \cite[Theorem 6.5]{grovesHyperbolicGroupsActing2023}. This section is dedicated to the proof of \cref{thm:metacondition}, a generalization of \cite[Theorem 6.5]{grovesHyperbolicGroupsActing2023}.

\begin{theorem}
\label{thm:metacondition}
 Suppose $\Hcal$ is a finite collection of full $\gp$--quasiconvex groups. For $1\leq i \leq n$, let $p_i\in G$, $H_i\in \Hcal$ and $S_i\subseteq H_i$ be chosen to satisfy:
\begin{equation}
    \label{eq:metacond}
    1\notin p_1 (H_1-S_1) \cdots p_n(H_n -S_n)
\end{equation}
Then for sufficiently long $\Hcal$-fillings $G\rightarrow G/K$, the kernel $K$ contains no element of the form 
\begin{equation}
    \label{eq:metaresult}
    p_1 t_1 \cdots p_n t_n
\end{equation}
\end{theorem}
The proof is very similar to the original version modulo some constant manipulations.

We first quote \cite[Lemma 5.10]{dahmaniHyperbolicallyEmbeddedSubgroups2017} below:

\begin{lemma}
\label{thm:greendlinger}
Let $\mathbb{X}$ be a hyperbolic geodesic space, equipped with a $200\delta$-separated very rotating family $\mathcal{C}=(C,\{G_c,c\in C\})$, let $Rot=\langle G_c| c\in C\rangle$, given $g\in Rot\backslash \{1\}$ and $x_0\in\mathbb{X}$, either $g\in G_c$ and $d(x_0,c)\leq 25\delta$ for some $c\in C$, or any geodesic $[x_0,gx_0]$ contains some $c\in[x_0,gx_0] \cap C$ and two points $x_1,x_2\subseteq [x_0,gx_0]$ such that $c\in [x_1,x_2]$, $d(c,x_1),d(c,x_2)\in [25\delta, 30\delta]$, and there exists $h\in G_c\backslash \{1\}$ such that $d(x_1,hx_2)\leq 5\delta$.
\end{lemma}

The following proposition imitates \cite[A.6]{manningSeparationRelativelyQuasiconvex2009}. 
\begin{proposition}
\label{thm:mmp09A6}
Suppose $H$ is $\gp$--quasiconvex, then there is a constant $R=R(G, \Pcal,\Scal,H)$ such that for every $r>R$, whenever a coned-off $0$-horoball is $R$-penetrated by $H$, the intersection of $H$ with the stabilizer of the horoball is infinite.
\end{proposition} 

\begin{remark}
\label{rmk:stabilizer-same-for-all-r}
    Notice that for every $Q\in \Qcal$ as in \cref{def:kr}, the stabilizer of $\Hcal_r(Q)$ for every $r$ is the same.
\end{remark}

Because of \cref{rmk:stabilizer-same-for-all-r}, we can use a series of space $\Gbb_r$ instead of a single space to prove \cref{thm:mmp09A6}. The proof is very similar to the original one, except we use the following lemma instead of \cite[Lemma A.8]{manningSeparationRelativelyQuasiconvex2009}. 

\begin{lemma}
\label{lem:qc-deep-pen-infinite-intersect}
    Suppose $H$ is $\gp$--quasiconvex. Let $Q\in \Qcal$ and $\Hcal_r(Q)$ be as in \cref{def:kr}. Let $P^t$ be the stabilizer of $\Hcal_r(Q)$, $P\in \Pcal$. If for all $R>0$, $\Hcal_r(Q)$ is $R$-penetrated by $H$ for every $r>R+3$, then $H\cap P^t$ is infinite.
\end{lemma}
\begin{proof}

    Since $H$ is $\gp$--quasiconvex, then by \cref{thm:uniform_qc} there exists $\l, r_\l$ such that for all $r>r_\l$, its image in $\Gbb_r$ is $\l$-quasiconvex. It follows that the quasiconvexity constant provided in \cite[Proposition 9.4]{hruskaRelativeHyperbolicityRelative2010} does not depend on the choice of $r$.

    Then for each $M>0$, we can choose $r>M+C+3$ so that $\Hcal_r(Q)$ is $(M+C)$-penetrated, and the rest of the proof is the same.
\end{proof}

\cref{thm:mmp09A6} gives a quick corollary which is analogous to \cite[Lemma 6.8]{grovesHyperbolicGroupsActing2023}

\begin{corollary}
\label{thm:closehoroball}
Suppose that $H$ is full $\gp$--quasiconvex. Then there exists a constant $\kappa$ satisfying the following:

Suppose that $g\in G$ and that $x_1, x_2\in gH$. Suppose that $\g$ is a geodesic in $X$ between $x_1$ and $x_2$. Further, suppose that $aP$ (for $a\in G$ and $P\in \Pcal$) is a coset so that $\g$ intersects the horoball corresponding to $aP$ to depth at least $\kappa$. Then $P^a\cap H^g$ is infinite.
\end{corollary} 

\begin{proof}[Proof of \cref{thm:metacondition}]
    Let $\Gbb_r$ be the coned-off cusped space associated to $(G,\Pcal)$, where $r=\max\{ |p_i|, \kappa \} + 2(n+100)\delta$ and $\kappa$ is the constant from \cref{thm:closehoroball} and suppose that $\Gbb_r$ is $\delta$-hyperbolic. Then $\Gbb_r$ is $200\delta$-separated. Suppose that $K$ is the kernel of a filling that is long enough to satisfy the conclusion of \cref{thm:greendlinger} with these constants.

    Suppose that there is an element $g\in K$ which is of the form
    \begin{equation}
    \label{eq:proofmetacond}
        g=    p_1 t_1 \cdots p_n t_n,
    \end{equation}
    where $t_i\in H_i - ((K\cap H_i)S_i)$, and suppose that $g$ is chosen so that $d_X(1,g)$ is minimal among all such choices. Notice that $g\neq 1$ because of \cref{eq:metacond}. Then \cref{eq:proofmetacond} can be represented by a geodesic $(2n+1)$-gon in $\Gbb_r$. Let $x_0$ be the basepoint with depth $0$. Let $\g$ be the geodesic for $g$, $\rho_i$ be the geodesic for $p_i$ and $\tau_i$ for the geodesic for $t_i$.

    Since $g\neq 1$ and $d(x_0,c)\geq r >25\delta$, by \cref{thm:greendlinger} there exists a horoball $A=sP$ for some $s\in G$ and $P\in\Hcal$, an element $k$ that stabilizes $A$, and points $a,b\in \g$ at depth at least $r-30\delta$ such that $d(a,kb)\leq 5\delta$. The geodesic $(2n+1)$-gon is $(2n-1)\delta$-thin, so $b$ lies within distance $(2n-1)\delta$ of some point $b'$ on some side other than $\g$. By construction of $r$, $\rho_i$ do not go deep enough into any horoballs, hence $b'$ lies in some $\tau_i$. Because of the choice of $r$, $b'$ has depth at least $\kappa$.

    Notice that $\tau_i$ is a geodesic between $p_1t_1\cdots p_i x_0$ and $p_1t_1\cdots p_i t_i x_0$, two points inside $p_1t_1\cdots p_i H_i$. By \cref{thm:closehoroball}, $H_i^{p_1t_1\cdots p_i} \cap P^a $ is infinite, hence by assumption of $\Hcal$-filling we have $k\in N_P^a \subseteq H_i^{p_1t_1\cdots p_i} $.

    Let $k' = k ^{(p_1t_1\cdots p_i)^{-1}}$, and let $t_i' = k't_i$. Then $k'\in K\cap H_i$. Note that $kg = p_1t_1\cdots p_i t_i' p_{i+1} \cdots p_nt_n $. Since $t_i\notin (K\cap H_i)S_i$, we have that $t_i'\notin (K\cap H_i)S_i$. Therefore $kg$ is another element of the form \cref{eq:proofmetacond}. It also has a shorter length, contradicting the choice of $g$ as the shortest such:
    \begin{align*}
        l(kg) &\leq d(x_0, a) + d(a,kb) + d(kb, kgx_0) \\
        & \leq d(x_0, a) + 5\delta + d(b, gx_0) \\ &< l(\g).
    \end{align*}
\end{proof}

\subsection{Proof of Application}
\label{subsec:main-proof}

\begin{proof}[Proof of \cref{thm:main_cube}]
    The kernels of Dehn fillings are generated by parabolic elements which act elliptically by assumption. So $C/K$ is simply-connected. See \cite[Theorem 4.1]{grovesHyperbolicGroupsActing2023}.

    By \cref{thm:finite_index_qc}, $Q_i$ is $\gp$--quasiconvex for every $i$.  Notice that $Q_i$ has finite cosets in $G_i$. Let $T=\{t_1,\dots, t_s\}$ be a set of representatives of cosets of $Q_i$, let $I = \max \{ l(t_j)\}$. Then by \cref{thm:seperate}, for sufficiently long $\Qcal$-fillings, $\pi(t_j)\notin \pi(Q_i)$ for every $j$.     
    It then follows that $G_i\cap K\subseteq Q_i$. For if there exists some $g\in G_i\cap K$ such that $g\notin Q_i$, then $g=t_jh$ for some $1\leq j\leq s$ and $h\in Q_i$. Since $\pi(g)=1$,  $\pi(t_j)=\pi(h^{-1})\in \pi(Q_i)$, a contradiction. Hence by \cite[Proposition 4.3]{grovesHyperbolicGroupsActing2023} $C/K$ is a cube complex.
	
    It follows from \cref{thm:metacondition}, \cite[Theorem 5.6-5.11]{grovesHyperbolicGroupsActing2023}, and \cite[Theorem II.5.20]{bridsonMetricSpacesNonPositive1999} that for sufficiently long $\Hcal$-fillings, each link of each cell in $C/K$ is non-positively curved.
\end{proof}

    \bibliographystyle{alpha}
    \bibliography{../Library}
\end{document}

%% file: macro.tex
\usepackage[english]{babel}

\usepackage{amsmath}
\usepackage{amsthm}
\usepackage{amssymb}
\usepackage{amsfonts}
\usepackage{mathrsfs}
\usepackage{thmtools}
\usepackage{thm-restate}
\usepackage{graphicx}
\usepackage{hyperref}
\usepackage{tikz-cd} 
\usepackage{quiver}
\usepackage{enumitem}
\usepackage{xcolor}
\usepackage{etoolbox}
\usepackage[colorinlistoftodos]{todonotes}
\usepackage{float}

\usepackage[capitalise]{cleveref}

\definecolor{Nord0}{HTML}{2e3440}

\definecolor{Nord4}{HTML}{d8dee9}

\definecolor{Nord7}{HTML}{8fbcbb}
\definecolor{Nord8}{HTML}{88c0d0}
\definecolor{Nord9}{HTML}{81a1c1}
\definecolor{Nord10}{HTML}{5e81ac}
\definecolor{Nord11}{HTML}{bf616a}
\definecolor{Nord12}{HTML}{d08770}
\definecolor{Nord13}{HTML}{ebcb8b}
\definecolor{Nord14}{HTML}{a3be8c} 
\definecolor{Nord15}{HTML}{b48ead} 

\newtheorem{theorem}{Theorem}[section]
\newtheorem{proposition}[theorem]{Proposition}
\newtheorem{corollary}[theorem]{Corollary}
\newtheorem{lemma}[theorem]{Lemma}

\theoremstyle{definition}
\newtheorem{definition}[theorem]{Definition}
\newtheorem{construction}[theorem]{Construction}

\newtheorem{assumption}[theorem]{Assumption}
\newtheorem{convention}[theorem]{Convention}

\theoremstyle{remark}
\newtheorem{remark}[theorem]{Remark}
\newtheorem*{remark*}{Remark}
\newtheorem{claim}{Claim}
\newtheorem*{claim*}{Claim}
\newtheorem{case}{Case}

\numberwithin{equation}{subsection}

\newcommand{\mc}{\mathcal}
\newcommand{\mb}{\mathbb}

\newcommand{\ms}{\mathscr}

\setuptodonotes{inline, color=Nord13} 

\newcommand{\A}{\mathbb{A}}
\newcommand{\N}{\mathbb{N}}

\newcommand{\Gbb}{\mb{G}}
\newcommand{\Hbb}{\mb{H}}
\newcommand{\K}{\mathbb{K}}

\renewcommand{\a}{\alpha}

\newcommand{\g}{\gamma}
\renewcommand{\d}{\delta}

\renewcommand{\l}{\lambda}

\renewcommand{\t}{\tau}

\newcommand{\D}{\Delta}
\renewcommand{\L}{\Lambda}
\newcommand{\G}{\Gamma}

\newcommand{\Acal}{\mc{A}}
\newcommand{\Bcal}{\mc{B}}

\newcommand{\Dcal}{\mc{D}}
\newcommand{\Ecal}{\mc{E}}

\newcommand{\Hcal}{\mc{H}}

\newcommand{\Ncal}{\mc{N}}
\newcommand{\Ocal}{\mc{O}}
\newcommand{\Pcal}{\mc{P}}
\newcommand{\Qcal}{\mc{Q}}
\newcommand{\Rcal}{\mc{R}}
\newcommand{\Scal}{\mc{S}}

\newcommand{\Xcal}{\mc{X}}
\newcommand{\Ycal}{\mc{Y}}

\newcommand{\he}{\hookrightarrow_h }
\newcommand{\gp}{( G,\Pcal )}
\newcommand{\hd}{( H,\Dcal )}

\newcommand{\GG}{\hat\G(G,\Pcal,X) }
\newcommand{\GH}{\hat\G(H,\Dcal,Y) }
\newcommand{\GGx}[1]{\hat\G(G,\Pcal,#1)}

\newcommand{\Gd}{\G(G,\Pcal\sqcup X)}

\newcommand{\CAT} {\ensuremath{\operatorname{CAT}}}
\newcommand{\CATz}{\CAT$(0)$ }

\renewcommand{\emptyset}{\varnothing}

\renewcommand{\tilde}{\widetilde}
\renewcommand{\hat}{\widehat}
\renewcommand{\bar}{\overline}